\documentclass[12pt]{amsart}
\usepackage{amssymb}
\usepackage{amsfonts}
\usepackage{latexsym}
\usepackage{amscd}
\usepackage{amsmath,delarray}
\usepackage[mathscr]{euscript}
\usepackage{xy} \xyoption{all}

\newcommand{\ignore}[1]{\bigskip TIKZ GRAPHIC\bigskip}

\usepackage{color,graphics}

\newcommand{\N}{{\mathbb{N}}}
\newcommand{\Z}{{\mathbb{Z}}}
\newcommand{\Q}{{\mathbb{Q}}}
\newcommand{\C}{{\mathbb{C}}}

\newcommand{\Zplus}{{\mathbb{Z}^+}}

\newcommand{\calM}{\mathcal{M}}

\newcommand{\ol}{\overline}
\newcommand{\xbar}{\ol{x}}
\newcommand{\ybar}{\ol{y}}
\newcommand{\zbar}{\ol{z}}
\newcommand{\ubar}{\ol{u}}
\newcommand{\Cbar}{\ol{C}}
\newcommand{\Ebar}{\ol{E}}

\newcommand{\Mbar}{\ol{\calM}}

\newcommand{\Lab}{L^{{\rm ab}}}

\newcommand{\uloopr}[1]{\ar@'{@+{[0,0]+(-4,5)}@+{[0,0]+(0,10)}@+{[0,0] +(4,5)}}^{#1}}
\newcommand{\uloopd}[1]{\ar@'{@+{[0,0]+(5,4)}@+{[0,0]+(10,0)}@+{[0,0]+ (5,-4)}}^{#1}}
\newcommand{\dloopr}[1]{\ar@'{@+{[0,0]+(-4,-5)}@+{[0,0]+(0,-10)}@+{[0, 0]+(4,-5)}}_{#1}}
\newcommand{\dloopd}[1]{\ar@'{@+{[0,0]+(-5,4)}@+{[0,0]+(-10,0)}@+{[0,0 ]+(-5,-4)}}_{#1}}
\newcommand{\luloop}[1]{\ar@'{@+{[0,0]+(-8,2)}@+{[0,0]+(-10,10)}@+{[0, 0]+(2,2)}}^{#1}}

\newcommand{\dotedge}{\ar@{.}}
\newcommand{\eqedge}{\ar@{=}}

\DeclareMathOperator{\soc}{soc}

\DeclareMathOperator{\rank}{rank}
\DeclareMathOperator{\rk}{rk}

\DeclareMathOperator{\ped}{ped}

\theoremstyle{plain}
\newtheorem{theorem}{Theorem}[section]
\newtheorem{lemma}[theorem]{Lemma}
\newtheorem{proposition}[theorem]{Proposition}
\newtheorem{corollary}[theorem]{Corollary}

\theoremstyle{definition}
\newtheorem{definition}[theorem]{Definition}
\newtheorem{example}[theorem]{Example}

\newtheorem{remark}[theorem]{Remark}

\newtheorem*{remark*}{Remark}
\newtheorem*{remarks*}{Remarks}
\newtheorem{question}[theorem]{Question}

\newtheorem*{assumption*}{Assumption}

\newtheorem{construction}[theorem]{Construction}

\numberwithin{equation}{section}

\begin{document}

\title[The realization problem for some wild monoids
and the Atiyah problem]
{The realization problem for some wild monoids \\
and the Atiyah problem}

\author{P. Ara}

\address{Departament de Matem\`atiques, Universitat Aut\`onoma de Barcelona,
08193 Bellaterra (Barcelona), Spain.} \email{para@mat.uab.cat}

\author{K. R. Goodearl}

\address{Department of Mathematics, University of
California, Santa Barbara, CA 93106. }\email{goodearl@math.ucsb.edu}

\thanks{The first-named
author was partially supported by DGI MINECO
MTM2011-28992-C02-01, by FEDER UNAB10-4E-378
``Una manera de hacer Europa", and by the
Comissionat per Universitats i
Recerca de la Generalitat de Catalunya. Part of this research was undertaken while the second-named author held a sabbatical fellowship from the Ministerio de Educaci\'on y Ciencias de Espa\~na at the Centre de Recerca Matem\`atica in Barcelona during spring 2011.
He thanks both institutions for their support and hospitality.}

\begin{abstract}
The {\it Realization Problem} for (von Neumann) regular rings asks what are the conical refinement monoids which can be obtained as the monoids of isomorphism classes of
finitely generated projective modules over a regular ring. The analogous realization question for the larger class of exchange rings is also of interest.
A refinement monoid is said to be {\it wild} if it cannot be expressed as a direct limit of finitely generated refinement monoids. In this paper, we consider the problem of realizing some
concrete wild refinement monoids by regular rings and by exchange rings. The most interesting monoid we consider is the monoid $\mathcal M$ obtained by successive
refinements of the identity $x_0+y_0=x_0+z_0$. This monoid is known to be realizable by the algebra $A= K[\mathcal F] $ of the monogenic free inverse monoid $\mathcal F$, for any choice of field $K$,
but $A$ is not an exchange ring. We show that, for any uncountable field $K$,  $\mathcal M$ is not realizable by a regular $K$-algebra, but that
a suitable universal localization $\Sigma ^{-1}A$ of $A$ provides an exchange, non-regular, $K$-algebra realizing $\mathcal M$. For any countable field $F$, we show that a skew version of the above
construction gives a regular $F$-algebra realizing $\mathcal M$. Finally, we develop some connections with the Atiyah Problem for the lamplighter group. We prove that the algebra $A$ can be naturally
seen as a $*$-subalgebra of the group algebra $kG$ over the lamplighter group $G= \Z_2\wr \Z$, for any
subfield $k$ of $\C$ closed under
conjugation, and we determine the structure of the $*$-regular closure of $A$ in $\mathcal U G$. Using this, we show that the subgroup of $\mathbb R$ generated by the von Neumann dimensions of matrices over $kG$ contains $\Q$.
\end{abstract}

\maketitle


\section{Introduction}

\subsection{The Realization Problem}
An important invariant in non-stable K-theory is the commutative monoid $V(R)$ associated to any ring $R$, consisting of the isomorphism classes of finitely generated projective (left, say)
$R$-modules, with the operation induced from direct sum. If $R$ is a (von Neumann) regular ring or a C*-algebra with real rank zero,
then $V(R)$ is a conical refinement monoid. This is also the case for all exchange rings, a class which provides a common generalization of the
above-mentioned rings and algebras  (e.g., \cite[Corollary 1.3, Theorem 7.3]{AGOP}). The \emph{Realization Problem} asks which conical refinement monoids appear as a $V(R)$ for $R$ in one of those classes. This problem encompasses all questions as to which kinds of direct sum decomposition behavior can occur for finitely generated projective modules over regular rings or exchange rings, or for projection matrices over C*-algebras with real rank zero. For instance, the fundamental \emph{separativity problem} asks whether the finitely generated projective modules over any regular ring $R$ satisfy \emph{separative cancellation}: ($A\oplus A \cong A\oplus B \cong B\oplus B \implies A \cong B$); equivalently, does $(a+a = a+b = b+b \implies a=b$) hold in $V(R)$? It is known that non-separative conical refinement monoids exist, so the problem is whether any of them are realizable over regular rings. 
This is open, and it is also open over exchange rings and over C*-algebras with real rank zero.

Wehrung \cite[Corollary 2.12 ff.]{W98IsraelJ} constructed a conical refinement monoid of cardinality $\aleph_2$ which is not isomorphic to $V(R)$ for any regular
ring $R$, but it is an open problem to determine whether every countable conical refinement monoid can be realized as $V(R)$ for some regular $R$. The countable case is the most important one, since problems such as the separativity problem can be reduced to realizability questions for appropriate countable conical refinement monoids. Another example of Wehrung's, \cite[Theorem 13.6]{W13ART}, shows that the realization problems differ for the above-mentioned classes of rings and algebras: There exist conical refinement monoids (which can be chosen of cardinality $\aleph_3$) which are realizable over exchange rings, but not realizable over any regular ring or over any C*-algebra of real rank zero.

We refer the reader to \cite{Areal} for a survey on the Realization Problem, and to \cite{directsum} for a survey on direct sum decomposition problems over regular rings.

A class of conical refinement monoids is provided by the construction of the
monoid $M(E)$ associated to any directed graph $E$, see \cite{AMP} for the row-finite case, and \cite{AG} for the general case.
In the row-finite case, these monoids can be realized by regular rings thanks to the work in \cite{AB}.
If $E$ is a finite directed graph and $K$ is any field, then the regular algebra of $E$ is $Q_K(E)=\Sigma^{-1}L_K(E)$, where $L_K(E)$ is the Leavitt path algebra
of $E$ and $\Sigma^{-1}L_K(E)$ is a suitable universal localization of $L_K(E)$. The algebra $Q_K(E)$ is regular and $V(Q_K(E))\cong V(L_K(E))\cong M(E)$.
Direct limit arguments then give the realization of every row-finite graph monoid $M(E)$ by a regular $K$-algebra $Q_K(E)$. In \cite{Aposet}, the first-named author obtained a realization result
for the class of monoids $M(\mathbb P)$ associated to finite posets $\mathbb P$. Here $M(\mathbb P)$ is the monoid generated by $\{ a_p : p\in \mathbb P \}$ subject to the relations
$a_p= a_p+a_q$ for $q<p$ in $\mathbb P$.

\subsection{Tame and wild refinement monoids}
The largest known classes of realizable refinement monoids consist of inductive limits of simple ingredients, such as
finite direct sums of copies of $\Zplus$ or $\{0,\infty\}$. In the case of graph monoids, these can be expressed
as direct limits of graph monoids associated to finite graphs (see the second proof of \cite[Theorem 3.1]{TW}).
In all cases, the corresponding refinement monoid is a direct limit of \emph{finitely generated refinement monoids}.
These monoids are more universally realizable in the sense that they can be realized as $V(R)$ for regular
algebras $R$ over any prescribed field. By contrast, examples are known of countable conical refinement monoids which are realizable only for regular algebras over \emph{some} countable
field (see \cite[Section 4]{Areal}).

These considerations led us in \cite{TW} to separate the class of refinement monoids into subclasses of \emph{tame} and \emph{wild} refinement monoids, where the tame ones are the inductive
limits of finitely generated refinement monoids and the rest are wild. The reader is referred to \cite{TW} for the basic theory of tame monoids.
Let us just mention that tame refinement monoids satisfy a number of desirable properties, such as separative cancellation and lack of perforation.

There are not many known explicit examples of wild monoids, though there is  evidence suggesting that they are really abundant.
Two explicit examples were studied in detail in \cite{TW}. These examples can be considered the easiest and best behaved examples of
wild refinement monoids. These monoids will be denoted in this paper by $\calM$ and $\Mbar$, as in \cite{TW}. (Their definitions will be recalled later.)
The monoid $\Mbar$ is a quotient of $\calM$ by an order-ideal, and can be understood as an auxiliary object in our investigation.

\subsection{Realizability of $\calM$}
In this paper, we study the particular wild monoid $\calM$ in terms of the Realization Problem
for exchange rings and regular rings. Although this conical refinement monoid is quite well-behaved--it is separative, unperforated, and archimedean--the subtleties of the
problem arise already for it. Indeed, $\calM$ admits a faithful state and it is not cancellative, so it cannot be realized by any regular
algebra over an uncountable field, by \cite[Proposition 4.1]{Areal}. In contrast, our main results provide two realizability theorems: (I) Over an arbitrary field $K$, there is an \emph{exchange} $K$-algebra which realizes $\calM$; (II) Over an arbitrary \emph{countable} field $F$, there is a regular $F$-algebra which realizes $\calM$. The algebras appearing in both (I) and (II) are explicitly constructed.

To give a bit more detail: The main construction in \cite{AE}, applied to $\calM$ and some additional data, gives rise to a well-behaved algebra $A$ over an arbitrary field $K$, which is
in some sense the analogue of the Leavitt path algebra of a directed graph closely related to $\calM$. Mimicking the construction of the regular algebra of a graph,
we can form the universal localization $\Sigma^{-1}A$ with respect to a suitable
set $\Sigma$ of elements of $A$, but the result of this construction is only an exchange algebra. We show in Theorem \ref{thm:exchangerealization} that
$V(\Sigma^{-1}A)\cong \calM$. This is apparently the first example of a countable conical refinement monoid which can be realized by
an exchange $K$-algebra, but not by a regular $K$-algebra, for any choice of uncountable field $K$.
It is interesting also to notice that the algebra $A$ has already appeared in the literature. It is precisely the semigroup algebra of the \emph{free monogenic inverse monoid}, that is, of the free
inverse monoid in one generator. Free inverse semigroups and their associated algebras have been extensively studied in the literature, see e.g. \cite{Preston, Munn, HR}. In particular,
Crabb and Munn computed the center of the algebra of a free inverse monoid in \cite{CrabbMunn}.

Our second main result is the realization of the monoid $\calM$ by a regular $F$-algebra, where $F$ is an arbitrary \emph{countable} field (Theorem \ref{thm:realizing}).
The algebra $\Sigma^{-1} A$ is not regular for any choice of base field $K$, so we cannot use this construction. Rather, we build a ``skew'' version of the above
algebra, working inside the algebra $\prod_{n=1}^{\infty} M_n(F)$. The idea is to build a new algebra $R$ over which the relations characterizing $\Sigma^{-1}A$ hold only in a ``relaxed" way, making it possible
for the new algebra to satisfy $V(R)\cong \calM$.

\subsection{Connections with the Atiyah Problem}
We analyse in the final section of the paper the relationship of our constructions with some questions closely connected with the
Atiyah Problem for the lamplighter group. Given any discrete group $\Gamma $, we may consider the
$*$-regular ring $\mathcal U (\Gamma )$, which is the classical ring of quotients of the von Neumann
algebra $\mathcal N (\Gamma )$, see \cite{luck} for details. Atiyah's original question asked whether certain
analytic $L^2$-Betti numbers of manifolds are rational. It has evolved into conjectures about the von Neumann dimensions of matrices
over the complex group algebras of general discrete groups $\Gamma$ (e.g., see the introduction to \cite{GLSZ}). Namely:

{\bf Strong Atiyah Conjecture}: If $T$ is an $m\times n$ matrix over $\C \Gamma$, then
$$\dim_{\mathcal N(\Gamma)}(\ker T) \; \in \; \sum \left\{ \frac1{|H|} \Z \biggm| H \;\; \text{is a finite subgroup of} \;\; \Gamma \right\},$$

\noindent where $\ker T$ denotes the kernel of the operator
$l^2(\Gamma )^m\to l^2(\Gamma)^n$ given by right mutiplication by
$T$ (see \cite[page 369]{luck}). While the conjecture has been
disproved (see \cite{GZ} and \cite{GLSZ}), and even examples have
been found where $\dim_{\mathcal N(\Gamma)}(T)$ is irrational (see
\cite{austin}, \cite{LW}, \cite{PSZ}, \cite{Grabowski},
\cite{Grab2}), thus giving a negative answer  to the question
originally posed by Atiyah, many questions about the values of von
Neumann dimensions over group algebras remain. Such questions are
closely linked to the structure of projections in $\mathcal
N(\Gamma)$ and the structure of certain $*$-regular subalgebras of
$\mathcal U(\Gamma)$.

To allow for general coefficients, let $k$ be a subfield of $\C$
closed under conjugation. The \emph{$*$-regular closure} $\mathcal R
(k\Gamma , \mathcal U (\Gamma ))$ of the group algebra $k\Gamma $ in
$\mathcal U (\Gamma )$ is defined as the smallest $*$-regular
subring of $\mathcal U (\Gamma )$ containing $k\Gamma $ (\cite{LS}).
The Strong Atiyah Conjecture can be reformulated in terms of the
ranks of the matrices over $k \Gamma $, seen as matrices over the
$*$-regular rank ring $\mathcal R (k\Gamma , \mathcal U (\Gamma ))$
(see Section \ref{sect:Atiyah} for the definitions). In many
instances, $\mathcal R (k\Gamma , \mathcal U (\Gamma ))$ is an Ore
localization of $k \Gamma $, see for instance \cite{LinnD} or
\cite{luck}. However, this is not the case for the lamplighter group
$G:=\Z _2 \wr \Z$ (see \cite{LLS}), and it is an interesting open
problem to determine exactly the structure of $\mathcal R :=
\mathcal R (k G , \mathcal U ( G ))$, especially for $k=\Q$. Observe
that all the kernel and range projections of elements in $k G$ are contained 
in $\mathcal R $, and therefore a
detailed knowledge of the structure of the lattice of projections of
this $*$-regular algebra will greatly help in the problem of
completely determining the set of values taken by the von Neumann
dimensions of the elements in $k G$.

We solve these problems for a certain subalgebra $A$ of $k G$, thus
uncovering a portion of the structure of the more elusive algebra
$\mathcal R $. Denote by $t$ the generator corresponding to $\Z$ and
by $a_i$ the generator corresponding to the $i$-th copy of $\Z _2$
in the wreath product $G = (\oplus_{\Z}\, \Z_2)\rtimes \Z$. Let
$e_0:=\frac{1}{2}(1+a_0)$ be the averaging idempotent corresponding
to $a_0$. We consider the $*$-subalgebra $A$ of $k G$ generated by
$s: = e_0t$. Notice that $A$ contains the element $T:=s+s^*$ (twice
the Markov operator on $G$), which is the element first considered
in \cite{GZ} and \cite{DS} in order to give a negative answer to the
Strong Atiyah Conjecture. We prove that this subalgebra is
$*$-isomorphic to the algebra $k[\mathcal F]$ of the monogenic free
inverse monoid, thus establishing a link with the theory developed
in the previous sections of the paper. Using this fundamental
observation, we are able to completely determine the structure of
the $*$-regular closure $\mathcal E := \mathcal R (A , \mathcal U (G
))$, see Theorem \ref{thm:regclosureA}. In particular, we obtain
that $\mathcal E$ is also the division closure of $A$. Note that
$\mathcal E \subseteq \mathcal R $, so we uncover part of the
structure of the $*$-regular algebra $\mathcal R$. In particular,
this sheds some light  on Linnell's question (\cite[Problem
4.4]{Linn-LMS}) whether the rational closure of $kG$ is a regular
ring, showing that at least this is the case for the $*$-subalgebra
$A$ of $kG$. As an additional piece of information, we mention that
the rank closure of the $*$-algebra $\mathcal R$ has been computed
in \cite{elek2}.

For a subring $S$ of $\mathcal U (G)$, denote by $\mathcal C (S)$ the subset of $\mathbb R^+$ consisting of all the von Neumann dimensions of matrices over $S$,
and by $\mathcal G (S)$ the additive subgroup of $\mathbb R$ generated by $\mathcal C (S)$. Then, using
our structural results, we show that $\mathcal C (\mathcal E) =\Q^+$ and that
$\mathcal G (A) = \Q$, so that we have $\Q \subsetneq \mathcal G (k G) $ (Corollary \ref{cor:QplusiscE}),
where the fact that the containment is strict follows from \cite{Grab2}.

\section{Background and preliminaries}\label{sect:prelims}

We collect here background definitions, concepts, and results needed below.

 \subsection{Commutative monoids} Let $M$ be a commutative monoid, written additively. We say that $M$ is \emph{conical} if $(x+y=0 \implies x=y=0)$ for any $x,y\in M$, and that $M$ is \emph{stably finite} if $(x+y=x \implies y=0)$ for any $x,y \in M$. It is \emph{separative} provided $(x+x= x+y= y+y \implies x=y)$ for any $x,y \in M$. In case $M$ is conical, an element $x\in M$ is called \emph{irreducible} provided $x\ne 0$ and $(a+b=x \implies a=0 \;\; \text{or} \;\; b=0)$ for any $a,b \in M$. The \emph{pedestal} of $M$ is the submonoid $\ped(M)$ generated by the irreducible elements of $M$.

The \emph{algebraic ordering} on $M$ is the translation-invariant pre-order $\le$ given by existence of subtraction: $(x\le y \iff y= x+a \;\; \text{for some} \;\; a \in M)$. An \emph{order-unit} in $M$ is any element $u$ such that every element of $M$ is bounded above by some nonnegative multiple of $u$. A pair of elements $x,y \in M$ need not have a greatest lower bound in $M$, but when a greatest lower bound exists, it will be denoted by $x\wedge y$. The monoid $M$ is \emph{unperforated} if $(nx\le ny \implies x\le y)$ for any $x,y \in M$ and $n\in \N$.

To say that $M$ is a \emph{refinement monoid} means that $M$ satisfies the \emph{Riesz refinement property}: given any $x_1,x_2,y_1,y_2 \in M$ with $x_1+x_2= y_1+y_2$, there exist $z_{ij} \in M$ such that $z_{i1}+z_{i2} = x_i$ for $i=1,2$ and $z_{1j}+z_{2j} = y_j$ for $j=1,2$.

An \emph{o-ideal} in $M$ is any submonoid $I$ which is hereditary with respect to the algebraic ordering, i.e., $(x+y\in I \implies x,y\in I)$ for all $x,y\in M$. Given an o-ideal $I$, we define a congruence $\equiv_I$ on $M$ by $(x \equiv_I y \iff x+a= y+b \;\; \text{for some} \;\; a,b \in I)$, and we denote the quotient monoid $M/{\equiv_I}$ by $M/I$.  Such quotients are always conical. If $M$ is a refinement monoid, so is $M/I$ \cite[p.~476]{Dob}.

In a conical refinement monoid, irreducible elements $x$ always cancel from sums: $(x+a= x+b \implies a=b)$ \cite[Lemma 1.1]{TW}, and the pedestal is an o-ideal \cite[Proposition 1.2]{TW}.


\subsection{Rings and fields}
We refer the reader to \cite{vnrr} for the general theory of von
Neumann regular rings, and to \cite{lam} for the general theory of rings. In particular, Chapters 4 and 16 of \cite{vnrr} provide the 
necessary background on unit-regular rings and (pseudo-)rank functions, respectively.
We shall need the concepts of
the socle $\soc (R)$ and the second socle $\soc _2 (R)$ of a
ring $R$. The right socle of $R$ is the sum of all
the minimal right ideals of $R$. A similar definition gives the left
socle of $R$. Both the left socle and the right socle of $R$ are (two-sided) ideals.
If $R$ is a semiprime ring, then its right and left
socles coincide (\cite[p. 186]{lam}), and this ideal is denoted by $\soc (R)$.
Moreover, if $R$ is semiprime, all the minimal right or left ideals are
generated by idempotents $e$ such that $eRe$ is a division ring (\cite[Corollary 10.23]{lam}). 
If $R$ and $R/\soc (R)$ are semiprime, the
second socle $\soc _2 (R)$ is defined by the property that $\soc_2(R)/\soc
(R)=\soc (R/\soc (R))$. (Of course, there are left and right versions of the second socle 
when $R$ or $R/\soc (R)$ is not semiprime.)

We will use (commutative) fields of three different kinds, and we
will use different notations for them. We will use
 $k$ to denote a subfield of $\mathbb C$ closed under complex conjugation (Section \ref{sect:Atiyah}),
 $F$ to denote a countable field (Section \ref{sect:realizMbyvnrrr}),
and $K$ to denote either a general field or a field which is not
necessarily of the above forms.


 \subsection{V-monoids} For a unital ring $R$, the \emph{V-monoid} $V(R)$ is the set of all isomorphism classes $[P]$ of finitely generated projective left $R$-modules $P$,
 equipped with the addition operation induced from direct sum: $[P]+ [Q] := [P\oplus Q]$. We work instead with the alternative \emph{idempotent picture} of $V(R)$, which consists
 of writing it as the set of equivalence classes $[e]$ of idempotent matrices $e$ over $R$, with $[e]+ [f] := \left[ \left( \begin{smallmatrix} e&0\\ 0&f \end{smallmatrix} \right) \right]$.
 (Recall that idempotent matrices $e$ and $f$ are \emph{equivalent}, written $e \sim f$, if there exist rectangular matrices $u$ and $v$ such that $uv=e$ and $vu=f$.)

In the non-unital case, $V(R)$ can be defined either via the idempotent picture (exactly as in the unital case), or via the following \emph{projective picture}: choose a unital ring $\widetilde{R}$ that
contains $R$ as a two-sided ideal, and define $V(R)$ to be the monoid of isomorphism classes of those finitely generated projective left $\widetilde{R}$-modules $P$ for  which $RP=P$, with addition
induced from direct sum. (See \cite[\S5.1]{GZane} for details.) For any ideal $I$ in a ring $R$, the monoid $V(I)$ is naturally isomorphic to a submonoid of $V(R)$, and we identify $V(I)$ with its image in $V(R)$.

A unital ring $R$ is an \emph{exchange ring} provided that for each $a\in R$, there is an idempotent $e\in aR$ such that $1-e\in (1-a)R$. (This is equivalent to Warfield's original definition \cite{War}
by \cite[p.~167]{GW} or \cite[Theorem 2.1]{nichol}, and it is left-right symmetric by \cite[Corollary 2]{War}.) In the non-unital case, $R$ is an exchange ring if for each $a\in R$, there exist
an idempotent $e\in R$ and elements $r,s\in R$ such that $e = ra = a+s-sa$ \cite{Aext}. All regular rings are exchange rings (\cite[Theorem 3]{War}, \cite[Example (3)]{Aext}),
as are all C*-algebras with real rank zero (\cite[Theorem 7.2]{AGOP}, \cite[Theorem 3.8]{Aext}). Whenever $R$ is an exchange ring, $V(R)$ is a refinement
monoid (\cite[Corollary 1.3]{AGOP}, \cite[Proposition 1.5]{Aext}). Moreover, if $R$ is an exchange ring and $I$ an ideal of $R$, then $V(I)$ is an o-ideal of $V(R)$
and $V(R)/V(I) \cong V(R/I)$ (this was proved for $R$ unital in \cite[Proposition 1.4]{AGOP}, but the same proof works in general).

If $R$ is a semiprime exchange ring, then $V(\soc (R))= \ped (V(R))$, where $\soc{R}$ denotes the socle of $R$.


 \subsection{Universal localization}   We will use the theory of universal localization, although only for elements of a ring. If $\Sigma $ is a subset of a ring $R$, we denote by $\Sigma ^{-1}R$ the ring obtained by universally adjoining to $R$ inverses for the elements of $\Sigma$. This can be formalized by using a construction with generators and relations.
 See \cite{Cohn} and \cite{schofield} for the general case, where square matrices and homomorphisms between finitely generated projective modules are considered, respectively.
 There is a canonical ring homomorphism $\iota_{\Sigma}\colon R\to \Sigma^{-1}R$, and $\Sigma^{-1}R$ is generated as a ring by $\iota_{\Sigma}(R)$ and
 the elements $\iota_{\Sigma}(x)^{-1}$ for $x\in \Sigma$. The universal localization is characterized by its \emph{universal property\/}: given any ring homomorphism
 $\psi \colon R\to S$ such that $\psi(x)$ is invertible for every $x\in \Sigma$, there exists a unique homomorphism $\widetilde{\psi}\colon \Sigma^{-1}R\to S$ such that $\psi=
 \widetilde{\psi}\circ \iota _{\Sigma}$.

 If $I$ is an ideal of $R$, we will denote by $\Sigma^{-1}I$ the ideal of $\Sigma^{-1}R$ generated by $\iota_{\Sigma}(I)$.
 By \cite[Proposition 4.3]{DPP}, we have $\Sigma^{-1}R/(\Sigma^{-1}I)\cong \ol{\Sigma}^{-1}(R/I)$, where $\ol{\Sigma}$ denotes the
 image of $\Sigma$ under the canonical projection $R\to R/I$. We will use this fact in the sequel without explicit reference.

 If $R$ is an algebra over a field $K$, then the universal localization $\Sigma^{-1}R$ is also a $K$-algebra, and it satisfies the corresponding
 universal property in the category of $K$-algebras.

 \subsection{Separated graphs}
A \emph{separated graph}, as defined in \cite[Definition 2.1]{AG},  is a pair $(E,C)$ where $E$ is a directed graph,  $C=\bigsqcup
_{v\in E^0} C_v$, and
$C_v$ is a partition of $s^{-1}(v)$ (into pairwise disjoint nonempty
subsets) for every vertex $v$. (In case $v$ is a sink, we take $C_v$
to be the empty family of subsets of $s^{-1}(v)$.) The pair $(E,C)$ is called \emph{finitely separated} provided all the members of $C$ are finite sets, and it is \emph{non-separated} if $C_v = \{s^{-1}(v)\}$ for all non-sinks $v\in E^0$.

The \emph{Leavitt path algebra of the
separated graph} $(E,C)$ with coefficients in the field $K$ is the
$K$-algebra $L_K(E,C)$ with generators $\{ v, e, e^*\mid v\in E^0,\
e\in E^1 \}$, subject to the following relations:
\begin{enumerate}
\item[] (V)\ \ $vv^{\prime} = \delta_{v,v^{\prime}}v$ \ for all $v,v^{\prime} \in E^0$ ,
\item[] (E1)\ \ $s(e)e=er(e)=e$ \ for all $e\in E^1$ ,
\item[] (E2)\ \ $r(e)e^*=e^*s(e)=e^*$ \ for all $e\in E^1$ ,
\item[] (SCK1)\ \ $e^*e'=\delta _{e,e'}r(e)$ \ for all $e,e'\in X$, $X\in C$, and
\item[] (SCK2)\ \ $v=\sum _{ e\in X }ee^*$ \ for every finite set $X\in C_v$, $v\in E^0$.
\end{enumerate}
Note that, for any given involution on $K$, there is a canonical involution $*$ on $L_K(E)$ making it a $*$-algebra.
This involution fixes the vertices and sends $e$ to $e^*$, for $e\in E^1$.

We only give the definition of the \emph{graph monoid} $M(E,C)$ associated to a finitely separated graph $(E,C)$. This is the commutative monoid presented by the set of generators $E^0$ and the relations
$$v = \sum_{e \in X} r(e) \qquad \text{for all} \; v \in E^0 \; \text{and} \; X \in C_v \,.$$

Lemma 4.2 of \cite{AG} shows that $M(E,C)$ is conical, and that it is nonzero as long as $E^0$ is nonempty.
Otherwise, $M(E,C)$ has no special properties, in contrast to the fact that, in the non-separated case, $M(E,C)$ is a separative unperforated refinement monoid (\cite[Propositions 4.4, 4.6, Theorem 6.3]{AMP}, \cite[Corollary 5.16]{AG}, \cite[Theorem 3.1]{TW}).

\begin{theorem}  \label{MECarb}
{\rm\cite[Proposition 4.4]{AG}} Any conical commutative monoid is isomorphic to $M(E,C)$ for a suitable finitely separated graph $(E,C)$.
\end{theorem}

 \subsection{The monoids $\calM$ and $\Mbar$} \label{subs:MandMbar}
 The monoids of the three separated graphs introduced in \cite[Section 4]{TW} are labeled here in the same way:
$$\calM_0:= M(E_0,C^0), \qquad\quad \calM := M(E,C), \qquad\quad \Mbar := M(\Ebar,\Cbar).$$
The first of these has the monoid presentation
$$\calM_0 = \langle x_0, \, y_0, \, z_0 \mid x_0+y_0 = x_0+z_0 \rangle.$$

We can present $\calM$ by the generators
$$x_0,y_0,z_0,a_1,x_1,y_1,z_1,a_2,x_2,y_2,z_2,\dots$$
and the relations
\begin{equation}
\begin{gathered}
x_0+y_0 = x_0+z_0\,,\qquad\qquad  y_l = y_{l+1}+a_{l+1}\,,\qquad\qquad  z_l = z_{l+1}+a_{l+1}\,,  \\
x_l = x_{l+1}+y_{l+1} =x_{l+1}+z_{l+1}\,.
\end{gathered}  \label{MECrelns}
\end{equation}
By \cite[Proposition 5.9 or Theorem 8.9]{AG}, $\calM$ is a refinement monoid. A direct proof of this is given in \cite[Proposition 4.4]{TW}.
The canonical order-unit in $\calM_0$ and $\calM$ is $u :=x_0+y_0$.

The monoid $\Mbar$ can be presented by the generators
$$x_0,y_0,z_0,x_1,y_1,z_1,x_2,y_2,z_2,\dots$$
and the relations
\begin{equation}
\begin{gathered}
x_0+y_0 = x_0+z_0\,,\qquad\qquad  y_l = y_{l+1}\,,\qquad\qquad  z_l = z_{l+1}\,,  \\
x_l = x_{l+1}+y_{l+1} =x_{l+1}+z_{l+1}\,.
\end{gathered}  \label{Mbarrels}
\end{equation}
The generators $y_n$ and $z_n$ for $n>0$ are redundant, and we write the remaining generators with overbars to avoid confusion between $\Mbar$ and $\calM$. Thus,
putting $\ybar = \ybar_n$ and $\zbar = \zbar_n$ for all $n\ge 0$, we have that $\Mbar$ is presented by the generators
$$\xbar_0,\ybar,\zbar,\xbar_1,\xbar_2,\dots$$
and the relations
\begin{equation}  \label{relnsMbar}
\xbar_0+\ybar= \xbar_0+\zbar\,, \qquad\qquad \xbar_l= \xbar_{l+1}+\ybar= \xbar_{l+1}+\zbar \,.
\end{equation}
By \cite[Lemmas 4.8 and 4.9]{TW},  $\Mbar$ is the quotient of $\calM$ modulo the pedestal of $\calM$, which is precisely the o-ideal generated
by $a_1,a_2, a_3,\dots $. In particular, $\Mbar$ is a conical refinement monoid.

\section{Bergman-Menal-Moncasi type rings}

In this section, we analyze and construct regular rings realizing the monoid $\Mbar$. The regular rings constructed by Bergman in \cite[Example 5.10]{vnrr} and by Menal and Moncasi in \cite[Example 2]{MM} realize $\Mbar$, but we shall need an example satisfying a suitable universal property.

The basic structure of a unital regular ring $B$ which realizes $\Mbar$ is easily described, as follows. Assume that we have an isomorphism
$$\phi: (V(B),[1]) \rightarrow (\Mbar,\ubar),$$
meaning a monoid isomorphism $V(B) \rightarrow \Mbar$ that maps $[1]$ to $\ubar$.
Since $\ybar$ and $\zbar$ are distinct irreducible elements in $\Mbar$ (\cite[Lemma 4.14(a)]{TW}), there are inequivalent primitive idempotents $e_1,e_2\in B$ such that $\phi([e_1]) = \ybar$ and $\phi([e_2]) = \zbar$.
Moreover, since $\ubar=\xbar+\ybar=\xbar+\zbar$, we can also assume that $\phi([1-e_1])= \phi([1-e_2]) = \xbar$, so that $1-e_1 \sim 1-e_2$.
Further, $\ybar$ and $\zbar$ generate the pedestal of $\Mbar$ (\cite[Lemma 4.14(e)]{TW}), and $n\ybar,n\zbar\le \ubar$ for all $n\in\N$. It follows that $\soc(B) = Be_1B \oplus Be_2B$ and each $(Be_iB)_B$ is an infinite direct sum of copies of $e_iB$. Hence, $(Be_iB)_B$ is not finitely generated. Finally, since $\Mbar/\ped \Mbar \cong \Zplus$ with the class of $\ubar$
mapping to $1$ (\cite[Lemma 4.14(c)]{TW}), we conclude that $B/\soc(B)$ must be a division ring.

We next look at regular rings whose $V$-monoids have the kind of structure just described.

\begin{lemma}  \label{VMMstructure}
Let $S$ be a unital regular ring such that $J:= \soc(S) \ne 0$, no homogeneous component of $J$ is finitely generated {\rm(}as a right or left ideal of $S${\rm\/)},
and $S/J$ is a division ring. Let $\{ e_i\mid i\in I \}$ be a complete, irredundant set of representatives for the equivalence classes of primitive idempotents in $S$.

{\rm (a)} The $[e_i]$ for $i\in I$ are distinct irreducible elements in $V(S)$.

{\rm (b)} The map $\sigma: (\Zplus)^{(I)} \rightarrow V(S)$ given by $\sigma(f)= \sum_{i\in I} f(i) [e_i]$ is an isomorphism
of $(\Zplus)^{(I)}$ onto $\ped V(S)$, which is an ideal of $V(S)$.

{\rm (c)} For each $f\in (\Zplus)^{(I)}$, there is a unique $c(f)\in V(S)$ such that $c(f)+ \sigma(f)= [1]$.

{\rm (d)} $V(S)$ is generated by $\{ [e_i] \mid i\in I\} \cup \{ c(f)\mid f\in (\Zplus)^{(I)} \}$.
\end{lemma}

\begin{proof} Since $J_S$ is not finitely generated, it cannot be a direct summand of $S_S$. This, together with the hypothesis that $J$ is a maximal right ideal of $S$, forces $J_S$ be essential in $S_S$.

(a) Primitivity implies the $[e_i]$ are irreducible. Irredundancy implies that $e_i \nsim e_j$ for all $i\ne j$, so the $[e_i]$ are distinct.

(b) Any irreducible element in $V(S)$ equals $[e]$ for some primitive idempotent $e\in S$. Then $e\sim e_i$ for some $i$, due to the essentiality of $J_S$ in $S_S$. This shows that the family $([e_i])_{i\in I}$
consists of all the irreducible elements of $V(S)$, so part (b) follows from \cite[Proposition 1.2]{TW}.

(c) By assumption, each homogeneous component $Se_iS$ of $J$ is infinitely generated as a right (say) ideal,
so $\aleph_0(e_iS) \lesssim S_S$. Thus, given $f\in (\Zplus)^{(I)}$, there exist right ideals $A_i \le S_S$ such
that $A_i\cong f(i)(e_iS)$ for all $i$. These $A_i$ are independent (because the simple modules $e_iS$ are pairwise
nonisomorphic), so $A := \sum_{i\in I} A_i \le S_S$ is a direct sum and $A\cong \bigoplus_{i\in I} f(i)(e_iS)$.
Since the support of $f$ is finite, $A$ is finitely generated, and there is an idempotent $p\in S$ such that $pS= A$.
Consequently, $[p]= \sigma(f)$, and $c(f) := [1-p]$ is an element of $V(S)$ such that $c(f)+ \sigma(f)= [1]$. Uniqueness follows because
irreducible elements cancel from sums (\cite[Lemma 1.1]{TW}).

(d) Refinement (or regularity) implies $V(S)$ is generated by the classes $[p]$ for idempotents $p\in S$. If $p\in J$, then $pS \cong \bigoplus_{i\in I} f(i)(e_iS)$ for some $f\in (\Zplus)^{(I)}$,
and then $[p]= \sum_{i\in I} f(i) [e_i]$. If $p\notin J$, then $1-p\in J$ because $S/J$ is a division ring. From the previous argument, $[1-p]= \sum_{i\in I} f(i) [e_i]= \sigma(f)$ for some $f\in (\Zplus)^{(I)}$.
Consequently, $[p]+\sigma(f)= [1]$, so $[p]= c(f)$ by the uniqueness of $c(f)$.
\end{proof}

\begin{lemma}  \label{VMMstructure2}
Continue with the hypotheses of Lemma {\rm\ref{VMMstructure}}. Assume that $S$ is stably finite, $I=\{1,2\}$, and $1-e_1\sim 1-e_2$ in $S$. Then there is an isomorphism $\psi: \Mbar\rightarrow V(S)$ such that
$$\psi(\ybar)= [e_1],\qquad \psi(\zbar)= [e_2],\qquad \psi(\xbar _0)= [1-e_1].$$
\end{lemma}

\begin{proof} Note that $\ped V(S)$ is generated by the irreducible elements $a_1= [e_1]$ and $a_2= [e_2]$. Identify $(\Zplus)^{(I)}$ with $(\Zplus)^2$, and take $c_n := c(n+1,0)$ for all $n \ge -1$. In particular, $c_0= [1-e_1]$ and $c_{n+1}+ [e_1]= c_n$ for all $n$. By induction, it follows that $c_n+n[e_1]= c_0$ for all $n \ge 0$.
Our assumptions imply that also $c_0= [1-e_2]$, and so $c_0= c(0,1)$. In particular, $c_0+[e_2]= [1]= c_0+[e_1]$. Hence,
$$c_{n+1}+[e_2]+ (n+1)[e_1]= c_0+[e_2]= c_0+[e_1]= c_n+ (n+1)[e_1]$$
for all $n$. Since the irreducible element $[e_1]$ cancels from sums, $c_{n+1}+[e_2]= c_n$ for all $n$. Similarly, we find that $c_{i+j-1}= c(i,j)$ for all $(i,j) \in (\Zplus)^2$.

The elements $c_0,[e_1],[e_2],c_1,c_2,\dots$ in $V(S)$ satisfy the defining relations for $\Mbar$, so there is a unique homomorphism $\psi: \Mbar \rightarrow V(S)$ such that
$$\psi(\ybar)= [e_1],\qquad \psi(\zbar)= [e_2],\qquad \psi(\xbar_n)= c_n \quad (n\ge0).$$
We have arranged $[e_1],[e_2]\in \psi(\Mbar)$ and $c((\Zplus)^2) \subseteq \psi(\Mbar)$, so it follows from Lemma \ref{VMMstructure}(d) that $\psi$ is surjective.

Consider $\mu,\mu' \in \Mbar$ such that $\psi(\mu)= \psi(\mu')$. If $\mu= r\ybar+s\zbar$ for some $r,s\in \Zplus$, then $\psi(\mu') =\psi(\mu)  \in \Zplus[e_1]+ \Zplus[e_2]$,
from which we see that $[1] \nleq \psi(\mu') +n[e_1]$ for all $n$. It follows that $\mu'$ cannot involve any $\xbar_n$, that is, we must have $\mu'= r'\ybar+s'\zbar$ for
some $r',s'\in \Zplus$. Consequently, $r[e_1]+s[e_2]= r'[e_1]+ s'[e_2]$, whence $r=r'$ and $s=s'$, that is, $\mu= \mu'$.

Now assume that $\mu,\mu' \notin \Zplus \ybar+ \Zplus \zbar$. By \cite[Lemma 4.11(b)]{TW},
we can write $\mu= r\xbar_n+s\ybar$ and $\mu'= r'\xbar_{n'} +s'\ybar$ for some $r,r'\in \N$ and $n,n',s,s' \in \Zplus$. Under the composition of $\psi$ with the canonical map $V(S) \rightarrow V(S/J)$, all the $\xbar_n$ map to $[1]$. Since $S/J$ is a division ring, it follows that $r=r'$. If $n<n'$, then since
$\xbar_n= \xbar_{n+i}+i\ybar$ for all $i\ge0$, we can rewrite $\mu$ as $r\xbar_{n'}+ (s+n'-n)\ybar$, and similarly if $n>n'$. Thus, we may assume that $n=n'$.

If $p_n$ is an idempotent in $S$ with $[p_n]= \xbar_n$, then we have $r.p_n \oplus s.e_1 \sim r.p_n \oplus s'.e_1$. This relation contradicts stable finiteness of $S$ if either $s<s'$ or $s>s'$. Hence, $s=s'$, and so $\mu=\mu'$. This completes the proof that $\psi$ is injective.
\end{proof}

We now present an explicit example of a unital regular $K$-algebra representing $\Mbar $, for an arbitrary field $K$.
This will be used in the next two sections. The construction, via a universal localization and a pullback, will provide a universal property needed later.

\begin{construction}
\label{constrQ}
Let $K$ be a field. Let $Q_1$ be the universal localization $Q_1= \Sigma^{-1}K\langle x,y\mid
xy=1\rangle$, where $\Sigma$ is the set $\{ f\in K[x] \mid f(0)=1\}$. Then $Q_1$ is regular, and it has a unique proper nonzero ideal
$I_1$, such that $Q_1/I_1\cong K(x)$, see \cite[Example 4.3]{AB}. Moreover, $I_1 =\soc(Q_1)$, this ideal is homogeneous (as a right or left semisimple $Q_1$-module),
and $Q_1\cong Q_1\oplus Q_1e$ for any idempotent $e\in I_1$ [ibid].

Let $Q$ be the pullback:
$$\xymatrixrowsep{3pc} \xymatrixcolsep{6pc}
\xymatrix{
Q \ar[r]^{\pi_1} \ar[d]_{\pi_2}  &Q_1 \ar[d]^{\pi'_1}  \\
Q_1^{\text{opp}} \ar[r]^{\pi'_2}  &K(x) }$$
where $\pi'_1$ and $\pi'_2$ are the quotient maps with kernels $I_1$ and $I_1^{\text{opp}}$.
The ideal $I_1\oplus I_1^{\text{opp}}$ of $Q$ is regular, and $Q/(I_1\oplus I_1^{\text{opp}}) \cong K(x)$, so $Q$ is regular.

We will view $Q$ as the subalgebra of $Q_1\times Q_1^{\mathrm{opp}}$ of those elements $(\alpha,\beta )$ such that
$\pi_1'(\alpha) =\pi _2' (\beta)$.
\end{construction}

\begin{proposition}
\label{prop:realizingMbar}
The $K$-algebra $Q$ of Construction {\rm\ref{constrQ}} is regular, and there is a monoid isomorphism
$\psi \colon \Mbar \longrightarrow V(Q)$
sending
$$\xbar _n \mapsto [(y^{n+1}x^{n+1}, 1)]= [(1, x^{n+1} \circ y^{n+1})], \qquad \ybar \mapsto [e_1], \qquad \zbar \mapsto [e_2],$$
where $e_1 : = (0,1-x\circ y)$ and $e_2 :=(1-yx, 0)$, and $\circ$ denotes multiplication in $Q_1^{\text{opp}}$.
\end{proposition}

\begin{proof}
 We have observed before that $Q$ is regular. The socle of $Q$ is $I_1\oplus I_1^{\text{opp}}$, which
 has two infinitely generated homogeneous components, and the quotient $Q/\soc (Q)$ is the field $K(x)$.
 The method of \cite[Lemma 13]{MM} can be applied to obtain that $Q$ is stably finite. We are thus in the hypotheses
 of Lemma \ref{VMMstructure2}, and so we obtain a monoid isomorphism $\psi \colon \Mbar \to V(Q)$ sending $\xbar_0$, $\ybar$, and $\zbar$ to $[1-e_1]$, $[e_1]$, and $[e_2]$, respectively.

 For $n>0$, the relations in $\Mbar$ imply that $\xbar_0= \xbar_n+ n \zbar$, whence
 $$\psi(\xbar_n) + (n+1)[e_2] = \psi(\xbar_0+ \zbar) = [1_Q].$$
 Since the irreducible element $[e_2]$ cancels from sums \cite[Lemma 1.1]{TW}, to verify that $\psi(\xbar_n) = [(y^{n+1}x^{n+1}, 1)]$ it suffices to show that
 \begin{equation}
 \label{psixnverif}
 [(y^{n+1}x^{n+1}, 1)] + (n+1) [e_2] = [1_Q].
 \end{equation}
For $i\in \Z^+$, we have
$$(x^i,x^i) (y^i(1-yx),0) = e_2  \qquad \text{and} \qquad (y^i(1-yx),0) (x^i,x^i) = (y^ix^i- y^{i+1}x^{i+1}, 0)$$
in $Q$, so that $e_2 \sim (y^ix^i- y^{i+1}x^{i+1}, 0)$. Since the idempotents $(y^ix^i- y^{i+1}x^{i+1} ,0)$ are pairwise orthogonal, we obtain $(n+1) [e_2] = [(1- y^{n+1}x^{n+1} ,0)]$, and \eqref{psixnverif} follows.;
\end{proof}

We now describe a method to prove that certain rings have associated monoid $\calM$.
This will be used in the proofs of our main realization results (Theorems \ref{thm:exchangerealization} and
\ref{thm:realizing}).
\medskip

\begin{proposition}
\label{prop:generaliso}
 Let $R$ be a unital exchange algebra over a field $K$ with
$\soc (R) \cong \bigoplus_{i=1}^{\infty}
M_i(K)$ {\rm(}as $K$-algebras\/{\rm)}, and such that $\soc (R)$ is essential in $R$.
Assume also that there is a monoid homomorphism $\tau \colon \calM \to V(R)$, with $\tau(u)= [1_R]$, such that
$\tau $ restricts to an isomorphism from $\ped (\calM )$ onto $V(\soc (R))$, and such
that the induced map
$$\ol{\tau} \colon \calM/\ped (\calM) \longrightarrow V(R/\soc (R))$$
is also an isomorphism. Then the  map $\tau $ is an isomorphism.
\end{proposition}

\begin{proof}
 Observe that our hypotheses imply that $\soc(R) = \bigoplus_{i=1}^\infty Re_iR$, where $e_1,e_2,\dots$ are pairwise inequivalent primitive idempotents with $(Re_iR)_R \cong i(e_iR)$.

Recall that the monoid $V(R)$ is a refinement monoid, because $R$ is an exchange ring (\cite[Corollary 1.3]{AGOP}).

Essential properties of $\calM$ and $V(R)$ that we shall need are the
following, which hold in any conical refinement monoid:
\begin{enumerate}
\item Elements in the pedestal can be canceled: $a+s=b+s$ implies $a=b$
if $s$ is in the pedestal of the respective monoid (\cite[Lemma 1.1]{TW}).
\item If $a+s=b+t$ with $s$ and $t$ in the pedestal, and if $s\wedge t=0$, then
$a=c+t$ and $b=c+s$ for some $c$ in the monoid.
\end{enumerate}
We shall also need the following specific property of $\calM$:
\begin{enumerate}
\item[(3)] If $c\in \calM$ and $v\in \ped(\calM )$, then $c=d+w$ for some $d\in \calM $ and $w\in \ped(\calM )$ such that $d\wedge(v+w)=0$.
\end{enumerate}
First, by \cite[Lemma 4.8]{TW}, $v\in \Zplus a_1+ \cdots+ \Zplus a_n$ for some $n$. Arguing as in
\cite[Lemma 4.8(e)]{TW}, there are $r_i\in \Zplus$ for $i=1,\dots,n$ such that $r_ia_i \le c$ but $(r_i+1)a_i \nleq c$.
Set $w= r_1a_1+ \cdots+ r_na_n$. By Riesz decomposition, $c=d+w$ for some $d\in \calM$. But $a_i\nleq d$ for all $i$, so $d\wedge(v+w)=0$, as required.

By \cite[Lemma 4.8(e)]{TW}, we have, for each $n\ge 1$, that $na_n\le u $ but $(n+1)a_n\nleq u$. Therefore $\tau (a_n)$ is an irreducible element in $V(R)$ (since $\tau |_{\calM}$
is an isomorphism onto $V(\soc (R))$), and $n\tau (a_n)\le \tau (u)=[1_R]$ but $(n+1)\tau(a_n) \nleq [1_R]$. We now show that $\tau (a_n)= [e_n]$. Note first that $[e_1]=\tau (a_i)$ for some $i\ge 1$. Since $2[e_1]\nleq [1_R]$, we get that $i=1$. Now assume that
$\tau (a_j)= [e_j]$ for $j=1,\dots , n$, for some $n\ge 1$. Then $[e_{n+1}]= \tau (a_i)$ for some $i\ge 1$ and, since $(n+2)[e_{n+1}]\nleq [1_R]$, we see that $i\le n+1$. Since we already know that
$\tau (a_j)=[e_j]$ for $j=1,\dots ,n$, we necessarily have $i=n+1$.

Now we are going to show that $[e_n]\wedge \tau (x_{m}+y_{m+1}+z_{m+1})=0$ for all $n\ge 1$ and $m\ge n-1$. To show this, it is enough to prove that
$[e_n]\nleq \tau (x_{n-1})$, $[e_n]\nleq \tau (y_n)$, and $[e_n]\nleq \tau (z_n)$. Observe that
$$\tau (y_{n-1}) = \tau (y_n)+\tau (a_n) = \tau (y_n)+[e_n]\, . $$
This implies that $[e_n]\le \tau (y_{n-1})$, and so $[e_n]\le \tau (y_j)$ for $j=0,1,\dots ,n-1$.
If $[e_n]\le \tau (y_n)$, then $2[e_n]\le \tau (y_{n-1})$, and we have
$$[1_R]= \tau (u) =\tau (y_0)+\tau (y_1)+\cdots + \tau (y_{n-1}) +\tau (x_{n-1})\ge (n+1)[e_n]+\tau (x_{n-1}) \, ,$$
which gives a contradiction, because  $(n+1)[e_n]\nleq [1_R]$.  For the same reason, we see that $[e_n]\nleq \tau (x_{n-1})$.
A similar argument shows that $[e_n]\nleq \tau (z_n)$.

We have a commutative diagram with isomorphisms as shown:
$$\begin{CD}
\ped (\calM) @>>> \calM @>>> \calM/\ped (\calM)\\
@V{\cong}VV  @V{\tau}VV  @V{\ol{\tau}}V{\cong}V \\
V(\soc (R)) @>>> V(R) @>>> V(R/\soc (R))
\end{CD}
$$

In showing
that $\tau$ is an isomorphism, we need the following:
\begin{enumerate}
\item[(4)] If $d\in \ped(\calM)$ and $b\in \calM$ with $\tau(d)\le \tau(b)$, then $d\le b$.
\end{enumerate}
It is enough to prove (4) for irreducible elements, by induction on the number of irreducible elements summing to $d$. Hence, we may assume that $d= a_n$ for some $n$. Thus, $[e_n]= \tau(d) \le \tau(b)$.

Now write $b$ as a $\Zplus$-linear combination of the $x_m$, $y_m$, $z_m$, $a_m$. We can
replace any $x_m$ by $x_{m+1}+y_{m+1}$, and we can replace any $y_m$ or $z_m$ by $y_{m+1}+a_{m+1}$ or $z_{m+1}+a_{m+1}$.
Consequently, we can assume that any $x_m$, $y_m$, or $z_m$ appearing in the expression for $b$ has index $m\ge  n$.
Then $\tau(b)$ is a $\Zplus$-linear combination of $\tau(x_m)$, $\tau(y_m)$, $\tau(z_m)$ with
$m\ge  n$ and $\tau(a_i)$ with $i\ge 1$. The irreducible element $[e_n]$ is $\le$ this combination. Invoking Riesz decomposition,
we find that $[e_n]$ must be $\le$ either some $\tau(x_m)$, $\tau(y_m)$, $\tau(z_m)$ with $m\ge  n$ or some $\tau(a_i)$.
But we have shown above that for $m\ge n$, we have $[e_n]\nleq \tau(x_m)$, $[e_n]\nleq \tau(y_m)$, and $[e_n]\nleq \tau(z_m)$.
Thus, the only possibility is that $[e_n]\le \tau(a_i)= [e_i]$ for some $i$. This forces $n=i$,
meaning that $a_n$ is one of the terms in the $\Zplus$-linear decomposition of $b$. Hence, $d= a_n \le b$, and (4) is
established.

We next show surjectivity. Let $a\in V(R)$. Then there are $b\in
\calM$ and $s,t\in \ped V(R)$ such that $a+s=\tau (b)+t$. After cancelling $s\wedge t$, we may assume that $s\wedge t =0$. In particular, it follows that $s\le \tau(b)$. Since $s=\tau(d)$ for some $d\in \ped(\calM)$, (4) implies that $b=d+e$ for some
$e\in \calM$. Applying (2) to the equation $a+s=\tau (b)+t$, we obtain $a=c+t$ and $\tau
(b)=c+s$, for some $c\in V(R)$. At this point, $c+s= \tau(b)= \tau(d+e)= s+\tau(e)$, so $c=\tau(e)$ because $s$ cancels. Since also $t\in \ped V(R)= \tau(\ped(\calM))$, we conclude that $a=c+t \in \tau(\calM)$, showing surjectivity.

Finally, we show injectivity. Let $a,b\in \calM$ such that $\tau (a)= \tau (b)$. Then, using that $\ol{\tau}$ is an isomorphism, we get $a+s=b+t$ for some
$s,t\in \ped (\calM)$. We can cancel $s\wedge t$ and
assume therefore that $s\wedge t=0$. It follows from (2) that $a=c+t$ and $b=c+s$ for some $c\in \calM$. By (3), $c=d+w$ for some $d\in \calM$ and $w\in \ped(\calM)$ such that $d\wedge(t+w) =0$. Since $\tau(w)$ cancels from the equation $\tau(a)= \tau(b)$, we have $\tau(d+t)= \tau(d+s)$, and it suffices to show that $d+t=d+s$. Hence, after replacing $a$, $b$, $c$ by $d+t$, $d+s$, $d$, we may assume that $c\wedge t=0$. Similarly, we may assume that $c\wedge s =0$.

Since $\tau(t) \le \tau(a)= \tau(b)= \tau(c)+ \tau(s)$ and $\tau(s)\wedge \tau(t) =0$, we must have $\tau(t)\le \tau(c)$. Then $t\le c$ by (4), whence $t=0$. Similarly, $s=0$, and therefore $a=c=b$, proving injectivity of $\tau$.

\end{proof}


\section{An exchange ring realizing $\calM$}
\label{sect:exchangerealization}

Let $(E_0,C^0)$ be the separated graph described
in Figure \ref{fig:partialsometry}, with $C_v= \bigl\{ \{
\alpha _1,\alpha_2\}, \{ \beta _1,\beta _2 \} \bigr\}$. The corner $vL_K(E_0,C^0)v$ is easily seen to be isomorphic to
the universal unital
K-algebra generated by a partial isometry. Indeed the partial isometry corresponds to the element $\alpha_1 \beta_1^*$ (see \cite[Example 9.6]{AE}).
Since in this paper we are using the reverse conventions to those in \cite{AE}, we draw Figure \ref{fig:partialsometry} with arrows reversed from those in \cite[Example 9.6, Figure 4]{AE}.

\begin{center}{
\begin{figure}[htb]
$$
\xymatrixrowsep{4.5pc}\xymatrixcolsep{5pc}\def\labelstyle{\displaystyle}
\xymatrix{
 &v  \ar@/_4ex/[dl]_{\alpha_2}  \ar@/_4ex/[dl]|(0.65){\circ}="e1"  \ar@/_1pc/[d]_(0.6){\alpha_1}  \ar@/_1pc/[d]|(0.45){\circ}="e2"  \ar@/^1pc/[d]^(0.6){\beta_1}  \ar@/^1pc/[d]|(0.45){\circ}="f2"  \ar@/^4ex/[dr]^{\beta_2}  \ar@/^4ex/[dr]|(0.65){\circ}="f1"    \dotedge"e1";"e2"  \dotedge"f1";"f2"  \\
w_2 &w_1 &w_3
}$$
\caption{The separated graph of a partial isometry}
\label{fig:partialsometry}
\end{figure}
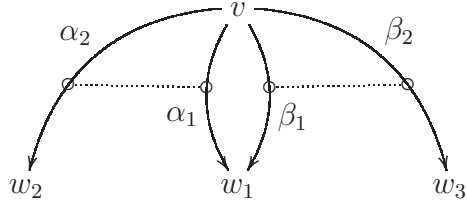}
\end{center}

Let $U$ be the multiplicative subsemigroup of $L_K(E_0,C^0)$ generated by $E_0^1\cup (E_0^1)^*$, and set
$$\Lab _K(E_0,C^0) := L_K(E_0,C^0)/J\, ,$$
where $J$ is the ideal of $L_K(E_0,C^0)$ generated by all the additive commutators $[e(u), e(u')]$, for $u,u'\in U$,
where $e(u)=uu^*$ for $u\in U$. It is shown in \cite[Corollary 5.8]{AE} that $V(\Lab _K(E_0,C^0))\cong M(F_{\infty}, D^{\infty})$,
where $(F_{\infty}, D^{\infty})$ is the complete multiresolution of $(E_0,C^0)$. As we observed
in \cite[\S 4.1]{TW}, the complete multiresolution of $(E_0, C^0)$ is precisely the separated graph $(E,C)$ considered in Subsection \ref{subs:MandMbar} and in \cite[Section 4]{TW},
so we obtain
\begin{equation}
\label{eq:VLabE,C}
V(\Lab_K(E_0,C^0))\cong M(E,C)\cong \calM.
 \end{equation}

Recall that an \emph{inverse semigroup} is a (not necessarily commutative) semigroup $S$ such that for each $a$ in $S$ there is a unique $a^*$
in $S$ such that $aa^*a=a$ and $a^*aa^*=a^*$.
This is equivalent to saying that $S$ is von Neumann regular and that the idempotents of $S$
form a commutative family of elements of $S$ (\cite[Theorem 3]{Lawson}). Moreover, the map $a \mapsto a^*$ is an involution on $S$. The reader is refered to \cite{Lawson} for the general theory of inverse semigroups.
The free inverse semigroup has been studied in several places, see for example \cite{Munn}.
The structure of the C*-algebra of the monogenic free inverse semigroup (i.e., the free inverse semigroup on one element)
has been studied in e.g. \cite{HR}.

Here we are interested in the semigroup algebra of the monogenic free inverse monoid $\mathcal F$, which is just the unitization of the monogenic free inverse semigroup. The canonical generators of $\mathcal F$ will be denoted by $s$, $s^*$.

\begin{lemma}
\label{Fpres}
The monoid $\mathcal F$ can be presented by the generators $s$ and $s^*$ together with the relations
$$ss^*s = s, \;  \; s^*ss^* = s^*, \; \; \text{and} \; \; pq = qp \; \; \text{for} \; \; p,q \in \{ s^i(s^*)^i, \; (s^*)^js^j \mid i,j \in \Z^+ \}.$$
\end{lemma}

\begin{proof}
Let $X$ be the monoid presented by two generators $x$ and $x^*$ with two relations $xx^*x = x$ and $x^*xx^* = x^*$. The natural involution on the free monoid generated by $x$, $x^*$ induces an involution $^*$ on $X$. We first prove that $\mathcal F \cong X/{\sim}$, where $\sim$ is the congruence on $X$ generated by the commutation relations $(uu^*)(vv^*) \sim (vv^*)(uu^*)$ for $u,v \in X$. The involution on $X$ induces an involution on $X/{\sim}$, and there is a unique $*$-homomorphism $f: X/{\sim} \to \mathcal F$ sending $\ol{x}$ to $s$. To see that $f$ is an isomorphism, we only need to show that $X/{\sim}$ is an inverse semigroup, since then the universal property of $\mathcal F$ provides a unique $*$-homomorphism $g: \mathcal F \to X/{\sim}$ sending $s$ to $\ol{x}$, and $g$ will be an inverse for $f$.

Write elements of $X/{\sim}$ in the form $\ol{u} = \ol{u}_1 \ol{u}_2 \cdots \ol{u}_n$ with all $\ol{u}_i \in \{ \ol{x}, \ol{x}^*\}$. An induction on $n$ implies that $\ol{u} \ol{u}^* \ol{u} = \ol{u}$. Indeed, this is trivial or given if $n=0,1$, and if it holds for an element $\ol{v} := \ol{u}_1 \ol{u}_2 \cdots \ol{u}_{n-1}$, then
$$\ol{u} \ol{u}^* \ol{u} = \ol{v} (\ol{u}_n \ol{u}_n^*) (\ol{v}^* \ol{v}) \ol{u}_n = \ol{v} (\ol{v}^* \ol{v}) (\ol{u}_n \ol{u}_n^*) \ol{u}_n = \ol{v} \ol{u}_n = \ol{u}.$$
It remains to show that all idempotents in $X/{\sim}$ commute.
Since by construction elements of the form $\ol{u}\ol{u}^*$ commute with each other, it is enough to show that idempotents have this form.
If $e= e^2$ in $X/{\sim}$, then
$$e= ee^* e = (ee^*)(e^*e) = (e^*e)(ee^*)= e^*ee^*= e^*.$$
Therefore $e= ee^*$, and idempotents in $X/{\sim}$ commute with each other. By \cite[Theorem 3]{Lawson}, we get that $X/{\sim}$ is an inverse semigroup, and thus
that $X/{\sim} \cong \mathcal F$.

Now let $Y := X/{\approx}$, where $\approx$ is the congruence on $X$ generated by the relations
$$pq \approx qp \; \; \text{for} \; \; p,q \in \{ x^i(x^*)^i, \, (x^*)^j x^j \mid i,j \in \Z^+ \},$$
and write $y$ for the congruence class of $x$ in $Y$. The involution on $X$ induces an involution on $Y$, and we shall show that $(uu^*)(vv^*) = (vv^*)(uu^*)$ for all $u,v \in Y$.This will imply that $\approx$ equals $\sim$, completing the proof of the lemma.

Observe that $y^i (y^*)^i y^i = y^i$ for all $i\ge 0$, which is trivial or given in the cases $i=0,1$. If it holds for some $i$, then
$$y^{i+1} (y^*)^{i+1} y^{i+1} = y \bigl[ \bigl( y^i(y^*)^i \bigr) (y^*y) \bigr] y^i = y \bigl[ (y^*y) \bigl( y^i(y^*)^i \bigr) \bigr] y^i = yy^i = y^{i+1},$$
establishing an induction. Applying the involution, we obtain as well $(y^*)^i y^i (y^*)^i = (y^*)^i$ for all $i \ge 0$.

We claim that any element in $Y$ can be written in the form $y^i (y^*)^j y^k$ with $j\ge i\ge 0$ and $j \ge k\ge 0$. It suffices to show that the set of elements of this form is closed under left multiplication by $y$ and $y^*$. Given such an element $u:= y^i (y^*)^j y^k$, if we multiply it on the left by $y$, there are two cases to consider. If $i<j$, then $yu= y^{i+1}(y^*)^j y^k$, with $j \ge i+1$. If $i= j$ then
\begin{align*}
yu &= y^{j+1} (y^*)^{j} y^k= y^{j+1}(y^*)^{j+1} y^{j+1}(y^*)^{j}y^k = y^{j+1}(y^*)^j(y^*y)(y^j(y^*)^j) y^k  \\
 &= y^{j+1}(y^*)^j (y^j(y^*)^j) (y^*y) y^k = y^{j+1}(y^*)^{j+1} y^{k+1}.
 \end{align*}
If we multiply $u$ on the left by $y^*$, we can assume that $i>0$, in which case we get
$$y^*u = (y^*y)(y^{i-1}(y^*)^{i-1}) (y^*)^{j-i+1} y^k = (y^{i-1}(y^*)^{i-1}) (y^*y) (y^*)^{j-i+1} y^k =
 y^{i-1}(y^*)^j y^k.$$

Finally, for $u:= y^i (y^*)^j y^k$ with $j\ge i\ge 0$ and $j \ge k\ge 0$, we have
\begin{align*}
uu^* &= y^i(y^*)^j y^k(y^*)^k y^j (y^*)^i = y^i(y^*)^{j-k} (y^*)^ky^k (y^*)^k y^j(y^*)^i = y^i(y^*)^j y^j (y^*)^i  \\
 &= (y^i(y^*)^i)((y^*)^{j-i}y^{j-i})(y^i(y^*)^i)=
  (y^i(y^*)^i)((y^*)^{j-i}y^{j-i}).
  \end{align*}
Therefore, by definition of $Y$, its elements of the form $uu^*$ commute with each other, as required.
\end{proof}

\begin{theorem}
 \label{theo:monfreeinvsem} Let $K$ be any field with involution, and endow the semigroup algebra $K[\mathcal F ]$ with its natural involution.
 Then there is a $*$-algebra isomorphism
 $$K[\mathcal F ] \cong v\Lab _K(E_0,C^0)v $$
which sends $s\mapsto \alpha_1\beta_1^*$. Moreover we have a monoid isomorphism
 $$V(K[\mathcal F])\cong \calM \, .$$
 \end{theorem}

\begin{proof}
 As we mentioned above, $vL_K(E_0,C^0)v$ is just the universal unital $K$-algebra generated by a partial isometry.
 This is the same as the semigroup algebra of the monoid given by two generators $x$, $x^*$ and defining relations
 $x=xx^*x$, $x^*=x^*xx^*$. Using this, it is straightforward to show that $K[\mathcal F ]$ is isomorphic to $v\Lab _K(E_0,C^0) v$ in the described manner.

Finally, observe that $v$ is a full corner in $\Lab _K(E_0,C^0)$, and thus, using  (\ref{eq:VLabE,C}), we get
$$V(K[\mathcal F ]) \cong V(v\Lab _K(E_0,C^0) v)\cong V(\Lab _K (E_0,C^0)) \cong M(E,C)\cong \calM \, ,$$
as desired.
 \end{proof}

We now proceed to summarize the algebraic structure of the semigroup algebra $K[\mathcal F]$.
To simplify notation, set $A :=K[\mathcal F ]$. Recall that
the set $\{ s^i(s^*)^i ,\; (s^*)^js^j : i,j\in \Z^+ \}$ is a commuting set of projections in $A$. (That these elements are projections stems from the fact that $s^i(s^*)^i s^i = s^i$ and $(s^*)^js^j (s^*)^j = (s^*)^j$.)
Observe also that $(s^n(s^*)^n)(s^m(s^*)^m)= s^n(s^*)^n$ whenever $n\ge m$. A similar relation holds for the projections $(s^*)^ns^n$. Thus,
\begin{equation}
\label{ss*seq}
1 \ge ss^* \ge s^2(s^*)^2 \ge \cdots \qquad \text{and} \qquad 1 \ge s^*s \ge (s^*)^2s^2 \ge \cdots \; .
\end{equation}
The commutativity of the projections in $S$ also yields $(s^{i-1}(s^*)^{i-1}) (s^*s) = (s^*s) (s^{i-1}(s^*)^{i-1})$ and $((s^*)^{i-1}s^{i-1}) (ss^*) = (ss^*) ((s^*)^{i-1}s^{i-1})$. Multiplying the first equation on the left by $s$ and the second by $s^*$, and also applying the involution, we obtain
\begin{equation}
\label{sis*is}
\begin{aligned}
s^i(s^*)^is &= s s^{i-1}(s^*)^{i-1}  &\qquad (s^*)^is^is^* &= s^* (s^*)^{i-1}s^{i-1}  \\
s^* s^i (s^*)^i &= s^{i-1} (s^*)^{i-1} s^*  &\qquad s (s^*)^i s^i &= (s^*)^{i-1} s^{i-1} s
\end{aligned}
\end{equation}
for all $i>0$.

For $i,j\in \Z^+$, set
\begin{equation}
\label{q-ij}
q_{-i,j} := (s^i(s^*)^i-s^{i+1}(s^*)^{i+1})((s^*)^js^j-(s^*)^{j+1}s^{j+1}) \in A .
\end{equation}
Observe that the elements $q_{-i,j}$ are projections in $A$, such that $q_{-i,j} \le s^i(s^*)^i$ and $q_{-i,j} \le (s^*)^js^j$.

In the following, we denote by $(e_{ij})_{i,j=1}^N$ the standard matrix units in any matrix algebra $M_N(K)$, and we extend the involution on $K$ to the $*$-transpose involution on $M_N(K)$.

\begin{lemma}
 \label{lem:centralprojhn}
 The projections $q_{-i,j}$ are mutually orthogonal minimal projections in $A$.
For $n\in \Z^+$, the projections
$$h_n:= q_{0,n}+q_{-1,n-1}+ \cdots + q_{-n,0}$$
are pairwise orthogonal central projections in $A$ such that $h_nA\cong M_{n+1}(K)$, via $*$-algebra isomorphisms sending $sh_n \mapsto \sum_{j=1}^n e_{j+1,j}$ and $s^*h_n \mapsto \sum_{i=1}^n e_{i,i+1}$.
\end{lemma}

\begin{proof}
In view of \eqref{ss*seq}, each of the sequences
$$ ( s^i(s^*)^i-s^{i+1}(s^*)^{i+1} \mid i \in \Z^+ ) \qquad \text{and} \qquad ( (s^*)^js^j-(s^*)^{j+1}s^{j+1} \mid j \in \Z^+ )$$
consists of mutually orthogonal projections. Since all these projections commute, it follows that the $q_{-i,j}$ are mutually orthogonal.
 Now consider the following representations of $A$ on $M_{n+1}(K)$, for $n\in \Z^+$.
 First, for $n=0$, we send $s$ and $s^*$ to $0$. Then $q_{0,0}$ is sent to $1$ so $q_{0,0}\ne 0$.
 Moreover $h_0=q_{0,0}=(1-ss^*)(1-s^*s)$ satisfies $h_0s=h_0s^*=sh_0=s^*h_0=0$, so $h_0$ is a central element
 and $h_0A= h_0K$.

 For $n\ge1$, set $t_n := \sum _{j=1}^n e_{j+1,j} \in M_{n+1}(K)$ and observe that all the matrices $t_n^i(t^*_n)^i$ and $(t^*_n)^jt_n^j$, for $i,j\ge0$, are diagonal and so commute with each other. Moreover, $t_nt^*_nt_n = t_n$ and $t^*_nt_nt^*_n = t^*_n$. Hence, by Lemma \ref{Fpres}, there exists a $*$-homomorphism from $\mathcal F$ to the multiplicative monoid of $M_{n+1}(K)$ such that $s \mapsto t_n$. This map extends to a $*$-algebra homomorphism $A\rightarrow M_{n+1}(K)$. The projections $q_{-i,j}$, with $i,j\in \Z^+$ and $i+j=n$, are sent to the minimal projections
 $e_{i+1,i+1}$ in $M_{n+1}(K)$. Therefore $q_{-i,j}\ne 0$ for all $i,j$.

 Observe that $s^*q_{0,n}=0= q_{0,n}s$ and, using \eqref{sis*is}, that $q_{-i,j}s= sq_{-i+1,j+1}$ if $i\ge 1$.
 Similarly, $sq_{-n,0}=0= q_{-n,0}s^*$ and $sq_{-i,j} = q_{-i-1,j-1}s$ if $j\ge 1$. It follows that, for each $n\ge 1$, the projections
 $q_{0,n}, q_{-1,n-1},\dots , q_{-n,0}$ are pairwise equivalent.

 Now observe that
\begin{align*}
 h_ns & = (q_{0,n}+q_{-1,n-1}+\cdots + q_{-n,0})s= q_{-1,n-1}s+\cdots +q_{-n,0}s \\
 & =sq_{0,n}+\cdots + sq_{-n+1,1}=
 s(q_{0,n}+\cdots + q_{-n+1,1}+q_{-n,0})=sh_n.
\end{align*}
Hence, $h_ns= sh_n$. Applying the involution, we get $h_ns^*=s^*h_n$ and so $h_n$ is central. Clearly $h_nh_m=0$ if $n\ne m$.

 We are going to check that $q_{-i,j}Aq_{-i,j}=q_{-i,j}K$.
In order to prove that, note first that every element in $\mathcal F$ can be written in the form $s^k(s^*)^ls^m$, where
$l\ge k\ge 0$ and $l\ge m\ge 0$ (see e.g. \cite{Preston}). Now if we have a product of the form $q_{-i,j}s^k(s^*)^ls^m$,
and $l=k+m$ then $q_{-i,j}s^k(s^*)^ls^m$ is $0$ if either $k>i$ or $m>j$ and it is $q_{-i,j}$ if $k\le i$ and $m\le j$.
Therefore $q_{-i,j}s^k(s^*)^ls^mq_{-i,j}$ is either $0$ or $q_{-i,j}$ in this case. If $l\ne k+m$, then
$q_{-i,j}s^k(s^*)^ls^m$ is either $0$ or $s^k(s^*)^ls^mq_{-i+k+m-l,j+k+m-l}$ and so $q_{-i,j}s^k(s^*)^ls^mq_{-i,j}$ is $0$ in this case.
This shows that $q_{-i,j}Aq_{-i,j}=q_{-i,j}K$. In particular, $q_{-i,j}$ is a minimal projection.

Now each $h_n$ is the sum of $n+1$ orthogonal equivalent projections  $q_{0,n}, q_{-1,n-1},\dots , q_{-n,0}$ with $q_{-i,j}Aq_{-i,j} = q_{-i,j}K$. It follows that there is an isomorphism $h_nA \cong M_{n+1}(K)$. To pin down a particular isomorphism, observe that the elements
$$\varepsilon_{ij} := \begin{cases} q_{1-i,n+1-i} (s^*)^{j-i} = (s^*)^{j-i} q_{1-j,n+1-j} &\text{(for} \; i\le j\text{)}  \\
q_{1-i,n+1-i} s^{i-j} = s^{i-j} q_{1-j,n+1-j} &\text{(for} \; i> j\text{)} \end{cases}$$
form a set of $(n+1)\times(n+1)$ matrix units in $h_nA$, with $\varepsilon_{11}+ \cdots+ \varepsilon_{n+1,n+1} = h_n$. Hence, we may choose the isomorphism $h_nA \rightarrow M_{n+1}(K)$ to send $\varepsilon_{ij} \mapsto e_{ij}$ for all $i$, $j$. Consequently,
$$sh_n= sq_{0,n}+ \cdots+ sq_{1-n,1} = \varepsilon_{21}+ \cdots+ \varepsilon_{n+1,n} \mapsto \sum_{j=1}^n e_{j+1,j}$$
and similarly $s^*h_n \mapsto \sum_{i=1}^n e_{i,i+1}$.
 This concludes the proof.
 \end{proof}

 \begin{lemma}
  \label{lem:AmodsocA}
  Let $B=A/S_1$, where $S_1 :=\sum _{n=0}^{\infty} h_nA$, and denote by $\ol{x}$ the class in $B$ of an element $x$ in $A$.
  Then $\soc (B)= \ol{I}\oplus \ol{J}$, where $\ol{I}=B(1-\ol{s}\ol{s}^*)B$ and $\ol{J}=B(1-\ol{s}^*\ol{s})B$, and $\soc (B)$ is an essential left or right ideal of $B$.
  Moreover, the family $\{ \ol{s}^i(1-\ol{s}\ol{s}^*)(\ol{s}^*)^j \mid i,j\ge 0 \}$ is a set of matrix units forming a $K$-basis for $\ol{I}$, and similarly the family
  $\{ (\ol{s}^*)^i(1-\ol{s}^*\ol{s})\ol{s}^j \mid i,j\ge 0 \}$ is a set of matrix units forming a $K$-basis for $\ol{J}$. We have $B/\soc (B) \cong K[x,x^{-1} ]$.
   \end{lemma}

   \begin{proof}
    Recall that every element in $\mathcal F$ can be written in the form $s^k(s^*)^ls^m$ with $l\ge k\ge 0$ and $l\ge m\ge 0$.
    We claim that
\begin{equation*}
(1-\ol{s}\ol{s}^*)\ol{s}^k(\ol{s}^*)^l\ol{s}^m= \begin{cases}  0  &(\text{if\ } k>0)\\
(1-\ol{s}\ol{s}^*)(\ol{s}^*)^{l-m}
&(\text{if\ } k=0). \end{cases}
\end{equation*}
    It is clear that the product is $0$ when $k>0$, so assume that $k=0$.
   If $m>0$, we have in $A$
 \begin{align*} (1-ss^*) & (s^*)^ls^m  = (1-ss^*)(s^*)^{l-m}((s^*)^ms^m) \\
 & = (1-ss^*)(s^*)^{l-m}((s^*)^{m-1}s^{m-1})-(1-ss^*)(s^*)^{l-m}((s^*)^{m-1}s^{m-1}-(s^*)^ms^m).
 \end{align*}
    So, by induction, it suffices to check that $(1-ss^*)(s^*)^{l-m}((s^*)^{m-1}s^{m-1}-(s^*)^ms^m)$ belongs to $S_1$.
   We have
\begin{align*}
    (1 -ss^*) &  (s^*)^{l-m}((s^*)^{m-1}s^{m-1}-(s^*)^ms^m)  \\
    &= (1 -ss^*)   (s^*)^{l-m} s^{l-m} (s^*)^{l-m} ((s^*)^{m-1}s^{m-1}-(s^*)^ms^m)  \\
    & = (s^*)^{l-m}\Big[s^{l-m}(1-ss^*)(s^*)^{l-m}((s^*)^{m-1}s^{m-1}-(s^*)^ms^m)\Big]  \\
    &= (s^*)^{l-m}q_{-(l-m), m-1}\in S_1 .
    \end{align*}
   Analogously, since each element of $\mathcal F$ can be also expressed in the form $(s^*)^ks^l(s^*)^m$ with $l\ge k\ge 0$ and $l\ge m\ge 0$,
   we obtain that an element of $B(1-\ol{s}\ol{s}^*)$ can be expressed as a linear combination of terms $\ol{s}^i(1-\ol{s}\ol{s}^*)$.
   Summing up, we get that any element of $B(1-\ol{s}\ol{s}^*)B$ can be expressed as a linear combination of the elements
   $\ol{s}^i(1-\ol{s}\ol{s}^*)(\ol{s}^*)^j$, where $i,j\ge 0$. Essentially the same computation shows that the family $\{ \ol{s}^i(1-\ol{s}\ol{s}^*)(\ol{s}^*)^j\mid i,j\ge 0 \}$
   is a set of matrix units, and that $(1-\ol{s}\ol{s}^*)B(1-\ol{s}^*\ol{s})= 0$. The latter gives that $\ol{I}\ol{J}=0$, and applying the involution, one gets $\ol{J} \ol{I} =0$. Since $\ol{I}$ is spanned by the matrix units $\ol{s}^i(1-\ol{s}\ol{s}^*)(\ol{s}^*)^j$, any nonzero ideal contained in $\ol{I}$ must contain $1-\ol{s}\ol{s}^*$ and so cannot be nilpotent. As $(\ol{I}\cap \ol{J})^2 =0$, we thus have $\ol{I}\cap \ol{J} = 0$.

   An analogous proof works for $\ol{J}$. Now observe that $(1-\ol{s}\ol{s}^*)B(1-\ol{s}\ol{s}^*)=(1-\ol{s}\ol{s}^*)K$, and so $1-\ol{s}\ol{s}^*$ is a minimal
   projection of $B$. Hence, $\ol{I}\subseteq \soc (B)$ and similarly $\ol{J}\subseteq \soc (B)$. It is clear that $B/(\ol{I}\oplus \ol{J})\cong K[x,x^{-1}]$. In particular, $\soc( B/(\ol{I}\oplus \ol{J}) ) = 0$, from which it follows that $\ol{I}\oplus \ol{J} = \soc(B)$.

It remains to prove that $\ol{I}\oplus \ol{J}$ is an essential left or right ideal of
   $B$. Because of the involution, it is enough to show the statement on the right. Since $B/(\ol{I}\oplus \ol{J})\cong K[x,x^{-1}]$ to show this it is enough to show that
   for any element $x$ in $B$ of the form $a_0+a_1\ol{s}+\dots + a_n\ol{s}^n+\alpha+\beta$, with $a_i\in K$, $a_0\ne 0$, $\alpha\in \ol{I}$ and $\beta \in \ol{J}$, there exists $r\in B$ such that
$xr\ne 0$ and $xr\in \ol{I}\oplus \ol{J}$. Observe that since $\ol{I}$ is spanned by the matrix units $\ol{s}^i(1-\ol{s}\ol{s}^*)(\ol{s}^*)^j$, we can choose $N\ge0$ so that $\alpha \ol{s}^N (1-\ol{s}\ol{s}^*) =0$. Since $\beta \ol{s}^N (1-\ol{s}\ol{s}^*) \in \ol{J}\ol{I} = 0$, we obtain
$$y := x  \ol{s}^N (1-\ol{s}\ol{s}^*) = \sum_{i=0}^n a_i  \ol{s}^{N+i} (1-\ol{s}\ol{s}^*) \in \ol{I}.$$
Moreover,
\begin{align*}
(1-\ol{s}\ol{s}^*) (\ol{s}^*)^N y &= (\ol{s}^*)^N \ol{s}^N \sum_{i=0}^n a_i  (1-\ol{s}\ol{s}^*) \ol{s}^i  (1-\ol{s}\ol{s}^*)  \\
 &= a_0  (\ol{s}^*)^N \ol{s}^N(1-\ol{s}\ol{s}^*) = a_0 (1-\ol{s}\ol{s}^*) \ne 0,
 \end{align*}
 so that $y\ne 0$. This concludes the proof.
   \end{proof}

 We are now ready to show that $S_1=\soc (A)$.

 \begin{proposition}
  \label{prop:S_1=socA}
  The ideal $S_1=\sum _{n=0}^{\infty} h_nA$ is essential in $A$ {\rm(}as a right or left ideal\/{\rm)}. Consequently, $S_1=\soc (A)$.
  Moreover, there is an embedding of $*$-algebras $A\hookrightarrow \prod_{i=1}^{\infty} M_i(K)$ such that
  $s\mapsto (x_i)$, where $x_i =\sum _{j=1}^{i-1} e_{j+1,j}$ for all $i\in \N$. The image of $\soc(A)$ under this embedding equals the
  ideal $\bigoplus _{i=1}^{\infty} M_i(K)$.
   \end{proposition}

 \begin{proof}
It is clear from Lemma \ref{lem:centralprojhn} that $S_1\subseteq \soc (A)$. So, in order to prove the equality, we only have to show that
  $S_1$ is essential as a left and as a right ideal. Again, it suffices to show the right-handed version. Since
  $\soc (B) $ is essential in $B$ by Lemma \ref{lem:AmodsocA}, and given the description of $\soc (B)$ obtained
  in that lemma, we see that it suffices to show that for an element $x \in A$ of the form
  $\alpha +\beta + \gamma$ such that
 $$\alpha =\sum _{i=0}^N\sum _{j=0}^M \lambda_{ij} s^i(1-ss^*)(s^*)^j \qquad \text{and} \qquad \beta =\sum_{k=0}^{N'} \sum_{l=0}^{M'} \mu _{k,l} (s^*)^k(1-s^*s)s^l,$$
 with $\lambda_{i,j},\mu _{k,l}\in K$ not all $0$, and $\gamma \in S_1$, there exists
  $r\in A$ such that $xr$ is a nonzero element of $S_1$. We will suppose that $\lambda _{N,M}\ne 0$. The case
  where all $\lambda _{ij}$ are $0$, and some $\mu _{kl}$ in nonzero is handled
  in a similar fashion. Now there exists a positive integer $L$ such that $\beta (s^*)^L=0= \gamma (s^*)^L$, and obviously
  $\alpha (s^*)^L\ne 0$. We may also assume that $L \ge N$. So multiplying $x$ on the right by $(s^*)^L$ and changing notation, we can assume that
  $x= \sum _{i=0}^N\sum _{j=0}^M \lambda_{ij} s^i(1-ss^*)(s^*)^j$, with $\lambda_{N,M}\ne 0$ and $M \ge N$.
  We now compute
  \begin{align*}
   x(s^M(s^*)^M(1-s^*s)s^M ) & = \Big[\sum _ {i=0}^N\sum _{j=0}^M \lambda_{ij}s^i(1-ss^*)(s^*)^j\Big]s^M(s^*)^M(1-s^*s) s^M \\
   & = \Big[ \sum _{i=0}^N \lambda _{iM} s^i(1-ss^*) (s^*)^M \Big] (1-s^*s) s^M \\
   & = \sum _{i=0}^N \lambda_{iM} [s^i(1-ss^*)(s^*)^i(s^*)^{M-i}(1-s^*s)s^{M-i}]s^i \\
   & = \sum _{i=0}^N \lambda _{iM} q_{-i,M-i}s^i
 = \sum_{i=0}^N \lambda _{iM} s^i q_{0,M} \ne 0.
  \end{align*}
This shows that $S_1$ is an essential ideal of $A$.

  The embedding $A\hookrightarrow \prod _{i=1}^{\infty} M_i(K)$ is essentially the regular representation of $A$ on $\soc (A)$.
  Namely, we define $\tau \colon A \to \prod_{n=0}^{\infty} h_nA$ by $\tau (x) = (xh_n)$. This is clearly a $*$-algebra homomorphism, and the kernel of
  $\tau $ is $\text{Ann} (\soc (A))$, the annihilator of $\soc (A)=S_1$. If $ \text{Ann}(\soc (A)) \ne 0$, then $T:= \soc (A)\cap \text{Ann} (\soc (A))\ne 0$,
  because we have shown before that the socle of $A$ is essential.   But then $T$ is a nonzero ideal of $\soc (A)$ with $T^2=0$.
  Since $\soc (A)$ is regular, this is impossible. So $\tau $ is injective, and clearly $\bigoplus _{n=0}^{\infty} h_nA \subseteq \tau (A)$.

  Now by Lemma \ref{lem:centralprojhn}, $h_nA\cong M_{n+1}(K)$ under a $*$-algebra isomorphism sending $sh_n$ to the partial isometry $x_{n+1}=\sum_{j=1}^n e_{j+1,j}$ of $M_{n+1}(K)$.
   \end{proof}

  Summarizing, the socle of $A$ is $S_1$, and the second socle $S_2$ of $A$ is the ideal of $A$ generated by $1-ss^*$ and $1-s^*s$.
 Note that $S_2$ is a regular ideal of $A$ and that $A/S_2\cong K[x,x^{-1}]$. So the ring $A$ does not seem to be far from being a regular ring. Nevertheless, we have the following negative result:

 \begin{proposition}
  \label{prop:calMnotrealizable-uncountable}
  The monoid $\calM$ cannot be realized by a regular algebra over any uncountable field.
   \end{proposition}

   \begin{proof}
    As observed in the proof of \cite[Lemma 4.5]{TW}, there is a homomorphism $s\colon \calM \to \mathbb Q$ such that
    $$s(x_n)=s(y_n)=s(z_n)=s(a_n)=1/2^n$$
    for all $n$, and $s^{-1}(0)=\{ 0 \}$. Moreover $\calM$ is not cancellative, since $x_0+y_0=x_0+z_0$ but $y_0\ne z_0$ (by \cite[Lemma 4.1]{TW}).
    Therefore, by \cite[Proposition 4.1]{Areal}, there is no regular algebra $R$ over an uncountable field such that $V(R)\cong \calM$.
    \end{proof}

 Returning to our main example, let $\Sigma $ be the set of elements of $A$ of the form $f(s)$, where $f\in K[x]$ is a polynomial such that $f(0)=1$.

 \begin{lemma}
  \label{lem:prodembeddSigmaA}
  With the above notation, the following properties hold:
  \begin{enumerate}
  \item The embedding $\rho_0: A \rightarrow \prod_{n=0}^\infty h_nA \cong \prod_{n=1}^\infty M_{n+1}(K)$ of Proposition {\rm\ref{prop:S_1=socA}} extends uniquely to a $K$-algebra homomorphism $\rho: \Sigma^{-1} A \rightarrow \prod_{n=1}^\infty M_{n+1}(K)$ such that $\rho \iota_\Sigma = \rho_0$, and the image of $\rho$ contains $\bigoplus_{n=1}^\infty M_{n+1}(K)$.
   \item The map $\iota _{\Sigma}\colon A\to \Sigma^{-1}A$ is injective, and $\iota _{\Sigma}$ induces an isomorphism from $\soc (A)$ onto $\Sigma^{-1}\soc(A)$.
   \item Let $S_2$ be the ideal of $A$ generated by $1-ss^*$ and $1-s^*s$. Then every element of $\Sigma ^{-1}S_2$ can be written
   as a linear combination of terms of the following forms {\rm(}where we suppress the map $\iota_{\Sigma}$ from the notation{\rm)}:
   \begin{enumerate}
   \item[(A)] $f^{-1}s^i(1-ss^*)(s^*)^j$, for $i,j\ge 0$ and $f\in \Sigma$,
   \item[(B)] $(s^*)^i(1-s^*s) s^jf^{-1}$, for $i,j\ge 0$ and $f\in \Sigma$,
   \item[(C)] $(s^*)^i(1-s^*s)s^jf^{-1}(1-ss^*)(s^*)^k$, for $i,j,k\ge 0$ and $f\in \Sigma$,
   \item[(D)] elements from $\soc (A)$.
   \end{enumerate}
 \item Let $I$ and $J$ be the ideals of $A$ generated by $1-ss^*$ and $1-s^*s$ respectively. Every element of $\Sigma^{-1} I$ {\rm(}respectively, $\Sigma^{-1} J${\rm)} can be written as a linear combination of terms of the forms {\rm (A), (C), (D) (}respectively, {\rm (B), (C), (D))}.
\end{enumerate}
  \end{lemma}

 \begin{proof}
  (1) For $n\in \Z^+$, consider the projection $A\to h_nA$ given by $x\mapsto h_nx$. With the identification $h_nA\cong M_{n+1}(K)$
  obtained in Lemma \ref{lem:centralprojhn}, we see that, for $f\in \Sigma$,   $h_nf$ corresponds to a unipotent matrix in $M_{n+1}(K)$, and thus
  to an invertible matrix. Hence, $\rho_0$ maps the elements of $\Sigma$ to invertible elements of $\prod_{n=1}^\infty M_{n+1}(K)$. This yields existence and uniqueness of $\rho$. The ideal $\bigoplus_{n=1}^\infty M_{n+1}(K)$ is already contained in the image of $\rho_0$, by Proposition \ref{prop:S_1=socA}.

 (2) Since $\rho_0$ is injective, so is $\iota_\Sigma$. Observe that the $\iota_{\Sigma}(h_n)$ are also central projections in $\Sigma^{-1}A$.
  Take $f\in \Sigma$. Then there are elements $y_n\in h_nA$ such that $(fh_n)y_n =y_n(h_nf)=h_n$. Hence,
  $$\iota_{\Sigma}(f)^{-1}\iota _{\Sigma}(h_n)=\iota_{\Sigma}(f)^{-1}\iota _{\Sigma}(f)\iota_{\Sigma}(h_n)\iota _{\Sigma}(y_n)= \iota_{\Sigma}(y_n)\in h_nA \, . $$
It follows that $\Sigma^{-1}\soc (A) =\iota_{\Sigma}(\soc (A))$, and so $\iota _{\Sigma}$ induces an isomorphism from
  $\soc (A)$ onto $\Sigma^{-1}\soc (A)$.

  From now on, we will identify $\soc (A)$ with $\Sigma^{-1} \soc (A)$.

  (3) It is clear that the set of linear combinations of elements of the forms (A), (B), (C) and (D) contains $\iota _{\Sigma} (S_2)$.
  So, to show that this set equals $\Sigma^{-1} S_2$ it is enough to prove that it is invariant under right and left multiplication by
  the elements $s,s^*, f^{-1}$, where $f\in \Sigma$. We will only check the corresponding property for right multiplication. The proof for
  the left multiplications is similar.

  To deal with a term of the form (A), it is enough to consider the particular case where the term is of the form $(1-ss^*)(s^*)^j$ for some $j\ge 0$.
  Then clearly $(1-ss^*)(s^*)^js^*=(1-ss^*)(s^*)^{j+1}$ is again of the form (A). Also, we have seen in the proof of Lemma \ref{lem:AmodsocA} that
  $(1-ss^*)(s^*)^js$ is congruent modulo $\soc (A)$ to $(1-ss^*)(s^*)^{j-1}$, so it is a linear combination of terms of the form (A) and (D).
  Now take $f$ in $\Sigma $. We will show the result for the term $(1-ss^*)(s^*)^jf^{-1}$ by induction on $j$.
  If $j=0$, we have $(1-ss^*)f= (1-ss^*)$ and so $(1-ss^*)f^{-1}= (1-ss^*)$. Assume that $j>0$ and that the result is true for terms of the form
  $(1-ss^*)(s^*)^kf^{-1}$ for $0\le k\le j-1$. Then using again that $(1-ss^*)(s^*)^js^m\equiv (1-ss^*)(s^*)^{j-m}$ if $m\le j$ and
  $(1-ss^*)(s^*)^js^m\equiv 0$ if $m>j$, where $\equiv$ denotes congruence modulo $\soc (A)$, we get
  $$(1-ss^*)(s^*)^jf\equiv (1-ss^*)(s^*)^j+ \sum _{k=0}^{j-1} \lambda _k (1-ss^*)(s^*)^k\, ,$$
  for some $\lambda _k\in K$. Therefore we get
  $$(1-ss^*)(s^*)^jf^{-1}\equiv (1-ss^*)(s^*)^j -\sum _{k=0}^{j-1} \lambda_k (1-ss^*)(s^*)^{k}f^{-1} \, ,$$
  and the result follows by induction. Note that only terms of the forms (A) and (D) appear in this case.

  Now we consider a term of the form (B). We can assume it is of the form $(1-s^*s)s^jf^{-1}$ for some $j \ge 0$ and $f\in \Sigma$.
  The products $(1-s^*s)s^jf^{-1}s$ and $(1-s^*s)s^jf^{-1}g^{-1}$, for $g\in \Sigma$, are clearly of the form (B).
 So we need to deal with the product $(1-s^*s)s^jf^{-1} s^*$.
  Assume first that $j=0$. Write $f=1+a_1s+\cdots +a_ns^n$. Then we have
  $$f^{-1}s^*= s^*-a_1f^{-1}ss^*-\cdots -a_n f^{-1}s^ns^*\, , $$
  so that
  $$(1-s^*s) f^{-1}s^*= -a_1(1-s^*s)f^{-1}(ss^*)-\cdots -  a_n (1-s^*s) f^{-1} s^ns^*$$
  and we reduce to the study of the case where $j>0$.
  So assume that $j>0$. Observe from \eqref{sis*is} that
  $s^j(s^*)^js^{j-1}  = s^js^*$.
  Using this we obtain
  \begin{align*}
   (1-s^*s)s^jf^{-1}s^* & = (1-s^*s)f^{-1} (s^js^*)= (1-s^*s) f^{-1}s^j(s^*)^j s^{j-1} \\
   & = (1-s^*s)f^{-1}s^{j-1}-(1-s^*s)f^{-1}(1-s^j(s^*)^j)s^{j-1} \, .
  \end{align*}
The first term $(1-s^*s) f^{-1}s^{j-1}$ is of the form (B), so we only have to deal with the second term, namely
$(1-s^*s)f^{-1}(1-s^j(s^*)^j)s^{j-1}$. We have
\begin{align*}
 (1 & -s^*s) f^{-1}(1-s^j(s^*)^j)s^{j-1} = (1-s^*s)f^{-1} \sum_{k=0}^{j-1} (s^k(s^*)^k - s^{k+1}(s^*)^{k+1}) s^{j-1}  \\
& = \sum _{k=0}^{j-1} (1-s^*s)f^{-1}s^k(1-ss^*)(s^*)^{j-1+k} \, ,
 \end{align*}
which is a sum of terms of the form (C). Therefore terms of the form (B) and (C) appear in this case.

  Now consider a term $[(s^*)^i(1-s^*s)][s^jf^{-1}(1-ss^*)(s^*)^k]$ of the form (C). Then the term $[s^jf^{-1}(1-ss^*)(s^*)^k]$ is of the form (A),
  and so when multiplied on the right by $s$, $s^*$ or $g^{-1}$, with $g\in \Sigma$, it becomes a linear combination of terms of the forms (A) and (D).
  Therefore the product $[(s^*)^i(1-s^*s)][s^jf^{-1}(1-ss^*)(s^*)^k]$ will become a linear combination of terms of the forms (C) and (D)
  when multiplied on the right by $s$, $s^*$ or $g^{-1}$, with $g\in \Sigma$.

  Since we have seen that $\soc(A)= \Sigma^{-1} \soc (A)$, the terms of the form (D) are stable under multiplication by elements in $\Sigma^{-1}A$.

 (4) This follows just as in (3).
 \end{proof}

  Observe that $\Sigma^{-1} A/\Sigma ^{-1} S_2\cong \ol{\Sigma}^{-1} (A/S_2)\cong \ol{\Sigma}^{-1} K[x,x^{-1}]\cong K(x)$.
  The linear span of $\mathcal S = \{ s^if(s)^{-1}  \mid i\in \Z^+ , f\in \Sigma \}$ is a subalgebra of $\Sigma ^{-1}A$ isomorphic to
  the localization $K[x]_{(x)}$ of $K[x]$ at the prime ideal $(x)$. Denote this subalgebra by $K[s]_{(s)}$.
   \medskip

We can now describe more precisely the structure of $\Sigma^{-1} A$.

 \begin{lemma}
  \label{lem:IJ}
  Let $\mathcal L $ be a subset of the family $\mathcal S$ described above such that the cosets of the
 elements
 of $\mathcal L$ form a $K$-basis of $K[s]_{(s)} / K[s]$. Let $I$ and $J$ be
 the ideals of $A$ generated by
 $1-ss^*$ and $1-s^*s$ respectively. Then $(\Sigma^{-1}I)(\Sigma^{-1}J)=\soc (A) \subseteq (\Sigma^{-1}J)(\Sigma^{-1}I)$, and the following is a $K$-basis
 for $(\Sigma^{-1}J)(\Sigma^{-1}I)/\soc (A)$:
 $$\mathcal B := \{ (s^*)^i(1-s^*s) \mathfrak f (1-ss^*) (s^*) ^k + \soc (A) \mid \,  i,k\in \Z^+, \;\mathfrak f\in \mathcal L \} .$$
  \end{lemma}

\begin{proof} Since the ideal $\soc(A)$ is generated by the projections
$$q_{-i,j} = s^i (1-ss^*) (s^*)^{i+j} (1-s^*s) s^j = (s^*)^j (1-s^*s) s^{j+i} (1-ss^*) (s^*)^i \, ,$$
we have $\soc(A) \subseteq IJ \subseteq (\Sigma^{-1}I)(\Sigma^{-1}J)$, and similarly $\soc(A) \subseteq (\Sigma^{-1}J)(\Sigma^{-1}I)$. The inclusion $(\Sigma^{-1}I)(\Sigma^{-1}J) \subseteq \soc(A)$ follows from Lemma \ref{lem:prodembeddSigmaA} and the fact that
 $(1-ss^*)(s^*)^j (1-s^*s)\in \soc (A)$ for all $j\ge 0$.

 The elements in $(\Sigma^{-1}J)(\Sigma^{-1}I)/\soc (A)$ must be linear combinations of terms of the form (C), and so they can be certainly
 expressed as linear combinations
 of the elements in $\mathcal B$. It remains to show that the latter are linearly independent. We make use of the homomorphism $\rho$ of Lemma \ref{lem:prodembeddSigmaA}(1).
It will be enough to show that  the elements
 $$b_{i,k,\mathfrak f} := \rho ((s^*)^i(1-s^*s) \mathfrak f (1-ss^*) (s^*) ^k ), \; \;  i,k\in \Z^+, \; \mathfrak f\in \mathcal L  $$
are linearly independent modulo $\bigoplus _{n=0}^{\infty} M_{n+1} (K) $.
Suppose that $\mathfrak f = s^j f^{-1}$,
 for some $j\in \Z^+$ and $f\in \Sigma $. Let $f(x)^{-1} = 1+\sum _{l=1}^{\infty} \beta_lx^l$ be the expansion of the element $f(x)^{-1}$ in the
 formal power series algebra $K[[x]]$. Then, for $n > j$,  the component of $b_{i,k,\mathfrak f}$  in $M_{n+1}(K)$ is exactly
 $\beta_{n-j}e_{n+1-i,k+1}$.

 Consider a linear combination of the elements $b_{i,k,\mathfrak f}$ which lies in  $\bigoplus _{n=0}^{\infty} M_{n+1} (K) $, and write this relation in the form
 $$\sum_{i,k,m=0}^t \alpha_{ikm} b_{i,k,\mathfrak f_m} \; \in \; \bigoplus _{n=0}^{\infty} M_{n+1} (K) ,$$
 where  the $\alpha_{ikm} \in K$ and $\mathfrak f_0, \dots, \mathfrak f_t$ are distinct elements of $\mathcal L$. Write $\mathfrak f_m = s^{j_m} f_m(s)^{-1}$ with $j_m \in \Z^+$ and $f_m\in \Sigma $, and let $f_m(x)^{-1} = 1+ \sum_{l=1}^\infty \beta_{ml} x^l$ in $K[[x]]$. For fixed $i$, $k$, we see from the previous paragraph that $\sum_{m=0}^t \alpha_{ikm} \beta_{m,n-j_m} = 0$ for $n \gg 0$. Thus, $\sum_{m=0}^t \alpha_{ikm} x^{j_m} f_m(x)^{-1} \in K[x]$. Since the $\mathfrak f_m$ are linearly independent modulo $K[s]$, it follows that $\alpha_{ikm} = 0$ for all $m$. Therefore the $b_{i,k,\mathfrak f}$ are indeed linearly independent modulo $\bigoplus _{n=0}^{\infty} M_{n+1} (K) $, as desired.
 \end{proof}

 \begin{proposition}
 \label{prop:socSigmaA} With the above notation, we have that
$\soc (A)$ is an essential ideal of $\Sigma^{-1}A$, and $\soc (\Sigma^{-1}A) =  \soc (A)$.
Moreover, the $K$-algebra homomorphism $\rho \colon \Sigma^{-1} A\to \prod_{n=0}^{\infty} M_{n+1}(K)$ is injective.
\end{proposition}

\begin{proof} Observe that the proof of Lemma \ref{lem:IJ} gives that $\rho $ induces an injective homomorphism
from $(\Sigma^{-1}J)(\Sigma^{-1} I)/\soc (A)$ into $\prod _{n=0}^{\infty} M_{n+1} (K)/\bigoplus _{n=0}^{\infty} M_{n+1} (K)$, where $I$ and $J$ are
 the ideals of $A$ generated by
 $1-ss^*$ and $1-s^*s$ respectively.
Using this, it follows that $\rho$ is injective on $(\Sigma^{-1}J)(\Sigma^{-1} I)$.

Next, consider an element $x \in \Sigma^{-1}I \setminus (\Sigma^{-1}J)(\Sigma^{-1} I)$. By Lemma \ref{lem:prodembeddSigmaA}, $x = f^{-1}y+z$ for some $f\in \Sigma$, $y \in I \setminus \soc(A)$, and $z\in (\Sigma^{-1}J)(\Sigma^{-1} I)$. Since the ideal $I+J$ of $A$ is regular, $y=yay$ for some $a\in A$. Then $yafz \in I(\Sigma^{-1}J) \subseteq \soc(A)$ by Lemma \ref{lem:IJ}, and so $yafx \in I \setminus \soc(A)$. In view of Lemma \ref{lem:AmodsocA}, there are elements $b,c\in I$ such that $b(yafx)c \equiv 1-ss^*$ modulo $\soc(A)$. We then see from Lemma \ref{lem:IJ} that for any choice of $\mathfrak f \in \mathcal L$, we have
$$(1-s^*s) {\mathfrak f} b(yafx)c \, \in \, (\Sigma^{-1}J)(\Sigma^{-1} I) \setminus \soc(A).$$
Since $\rho$ is injective on $(\Sigma^{-1}J)(\Sigma^{-1} I)$, it follows that $\rho(x) \ne 0$. This shows that $\rho$ is injective on $\Sigma^{-1}I$.

A similar argument now shows that $\rho$ is injective on $\Sigma^{-1}I + \Sigma^{-1}J$.

Finally, consider the ideal
$$L := \biggl\{ c \in \prod _{n=0}^{\infty} M_{n+1} (K) \biggm| \{ \rank(c_n) \mid n\in \Z^+ \} \; \text{is bounded} \biggr\}$$
of $\prod _{n=0}^{\infty} M_{n+1} (K)$, and observe that $\rho$ maps $1-ss^*$ and $1-s^*s$ into $L$. Since $\Sigma^{-1}I + \Sigma^{-1}J$ is a maximal ideal of $\Sigma^{-1}A$, we see that $\rho^{-1}(L) = \Sigma^{-1}I + \Sigma^{-1}J$. This, together with the injectivity of $\rho$ on $\Sigma^{-1}I + \Sigma^{-1}J$, implies that  $\rho $ is injective.

The statements about the socle of $\Sigma^{-1} A$ follow.
\end{proof}

We can now obtain one of our main results.

 \begin{theorem}
  \label{thm:exchangerealization} Let $K$ be a field with involution {\rm(}where the identity involution is allowed\/{\rm)}, let $A=K[\mathcal F]$ be the semigroup algebra of the monogenic free inverse monoid with its natural involution, and define $\Sigma \subset A$ as before. Then the universal localization
  $\Sigma^{-1}A$ is an exchange {\rm(}but not regular\/{\rm)} $K$-algebra and $V(\Sigma^{-1}A)\cong \calM$.
   \end{theorem}

\begin{proof} Set $R :=\Sigma^{-1} A$.
 We first show that $R$ is an exchange ring. We will use \cite[Theorem 2.2]{Aext}, which says that, whenever $J$
 is an ideal of a ring $T$, the ring $T$ is an exchange ring if and only if both  $J$ and $T/J$ are exchange rings
 and idempotents from $T/J$ can be lifted to idempotents in $T$. Indeed, in our applications of this result
 we will be under the hypothesis of \cite[Corollary 2.4]{Aext}, in which the lifting of idempotents is automatic.

Set $B=A/\soc (A)$ and $S= \Sigma ^{-1} B= R/\soc (A)= R/\soc (R)$. We first determine the structure of $S$.
Let $Q_0 :=K\langle x,y \mid xy=1 \rangle $ and $Q_1 := \Sigma_0^{-1} Q_0$, where $\Sigma_0 := \{ f\in K[x] \mid f(0) = 1\}$, and let $Q$ be the pullback built in Construction \ref{constrQ}.
Let $I$ and $J$ be the ideals of $A$ generated by $1-ss^*$ and $1-s^*s$ respectively. There is a surjective $K$-algebra
homomorphism $\gamma _1\colon B\to Q_0$ such that $\gamma _1 (\ol{s})= x$ and
$\gamma _1 (\ol{s}^*) = y$. The kernel of $\gamma _1$ is the ideal $\ol{I} := I/\soc(A)$. Similarly, there is a surjective $K$-algebra homomorphism $\gamma _2\colon
B\to Q_0^{\text{opp}}$ such that $\gamma _2 (\ol{s})= x$ and $\gamma _2 (\ol{s}^*)= y$, with kernel $\ol{J} := J/\soc(A)$.

These maps induce unique homomorphisms, also denoted by $\gamma _1 $ and $\gamma _2$, from $S=\Sigma^{-1} B $ onto
$Q_1$ and $
\Sigma^{-1}( Q_0 ^{\text{opp}})= Q_1^{\text{opp}}$,
with respective kernels $\Sigma^{-1}\ol{I}$ and $\Sigma^{-1}\ol{J}$.

Clearly we have $\pi_1'\circ \gamma _1 = \pi_2'\circ \gamma _2$, so the universal property of the pullback gives us a unique
homomorphism $\gamma \colon S \to Q$ such that $\gamma _i = \pi_i \circ \gamma $ for $i=1, 2$. Surjectivity of $\gamma _1$ and $\gamma _2$, together with the fact that $\gamma_1( \Sigma^{-1} \ol{J}) = I_1 = \ker(\pi'_1)$,
implies surjectivity of $\gamma $, so $Q$ is a factor ring of $S$. Observe that the kernel of $\gamma $ is precisely
$\mathcal J := \Sigma^{-1}\ol{I} \cap \Sigma^{-1} \ol{J}$. Note that $\mathcal J ^2=0$ because, by Lemma \ref{lem:IJ}, we have
$(\Sigma^{-1}\ol{I})(\Sigma^{-1}\ol{J})=0$. Now by Proposition \ref{prop:realizingMbar}, $Q$ is a regular ring. Since $\mathcal J$
is a radical ring, it is also an exchange (non-unital) ring by \cite[Example 2, p. 412]{Aext}. Moreover, idempotents lift
modulo nilpotent ideals, so by an application of \cite[Theorem 2.2]{Aext} or  \cite[Proposition 1.5]{nichol} we get that $S$ is
an exchange ring. By Lemma \ref{lem:IJ}, we have $\mathcal J\ne 0$, and so $S$ is not regular. Also observe that, by \cite[Proposition 1.4]{AGOP},
we have $V(Q) \cong V (S/\mathcal J) \cong V(S)/ V(\mathcal J)= V(S)$, so by Proposition  \ref{prop:realizingMbar}, there is an
isomorphism $\widetilde{\psi}\colon \Mbar \to V(S)$ sending $\ol{x}_n$ to $[\ol{s}^{n+1}(\ol{s}^*)^{n+1}]$, $\ol{y}$ to $[1-\ol{s}\ol{s}^*]$, and
$\ol{z}$ to $[1-\ol{s}^*\ol{s}]$.

Now we have that $\soc (R) = \soc (A)$ is a regular (non-unital) ring, and we have shown that  $S=R/\soc (R)$ is an exchange ring.
So it follows from \cite[Corollary 2.4]{Aext} that $R$ is an exchange ring. Since $S$ is not regular, we see that $R$ is not regular either.

Finally, there is a monoid homomorphism
 $\tau \colon \calM\to V(R)$ such that
\begin{align*}
\tau(x_n) &= [s^{n+1}(s^*)^{n+1}] = [(s^*)^{n+1}s^{n+1}]  &\tau(y_n) &= [s^n(1-ss^*)(s^*)^n] = [(1-ss^*)(s^*)^n s^n]  \\
\tau(a_{n+1}) &= [q_{0,n}] = [q_{-n,0}]  &\tau(z_n) &= [(s^*)^n(1-s^*s)s^n] = [s^n(s^*)^n(1-s^*s)]
\end{align*}
for all $n \in \Z^*$. To see this, observe that the indicated elements of $V(R)$ satisfy the relations \eqref{MECrelns}. Note also that $\tau$ sends the order-unit $u = x_0+y_0$ to $[1_R]$. Now by \cite[Lemma 4.8]{TW}, the map $(m_n)_{n=0}^\infty \mapsto \sum_{n=0}^\infty m_n a_{n+1}$ from $(\Z^+)^{(\Z^+)} \rightarrow \ped(\calM)$ is an isomorphism. In view  of Proposition \ref{prop:S_1=socA}, we also have an isomorphism $(\Z^+)^{(\Z^+)} \rightarrow V(\soc(R)) = V(\soc(A))$, given by $(m_n)_{n=0}^\infty \mapsto \sum_{n=0}^\infty m_n [q_{0,n}]$. Hence, $\tau$ restricts to an isomorphism from $\ped (\calM)$ onto $V(\soc (R))$.
Moreover, $\soc (R)$ is an essential ideal of $R$ (by Proposition \ref{prop:socSigmaA}) such that $\soc (R)\cong \bigoplus_{i=1}^{\infty} M_i(K)$, and the map
induced by $\tau$ from $\Mbar = \calM/ \ped(\calM)$ to $V(R/\soc (R))=V(S)$ is easily seen to agree with the map $\widetilde{\psi}$ described before,
which is a monoid isomorphism. Therefore we conclude from Proposition \ref{prop:generaliso} that $\tau $ is an isomorphism from $\calM$ onto
$V(R)$. The proof is complete.
\end{proof}

\begin{remark}
 \label{noinvo} The involution of $A$ cannot be extended to $R=\Sigma^{-1}A$. Indeed, suppose there is an involution $*$ on $R$
 extending the involution on $A$. Since $\soc (A)$ is $*$-invariant,
 the involution on $B=A/\soc (A)$ would extend to an involution on $S=R/\soc (A)$.
 Now the ideals
 $\widetilde{I}= \Sigma^{-1}\ol{I}$ and $\widetilde{J}= \Sigma^{-1} \ol{J}$ would be self-adjoint ideals of $S$, and
 $$0\ne (\widetilde{J}\widetilde{I})^*= \widetilde{I}\widetilde{J} =0, $$
 a contradiction.
 \end{remark}

\section{Realizing $\calM$ by a von Neumann regular ring}
\label{sect:realizMbyvnrrr}

In this section, we realize the monoid $\calM$ as the $V$-monoid of a regular $F$-algebra for any \emph{countable} field $F$.
The key idea is to use a ``skew'' version of the construction performed in the previous section, so that the fundamental relations
are satisfied only in a ``relaxed way''. More precisely, the regular algebra will be built as a subalgebra of $T :=\prod_{n=1}^{\infty} M_n(F)$
so that each of the relations holds modulo $\soc(T)=\bigoplus _{n=1}^{\infty} M_n(F)$. To make this construction we need the countability of the field $F$.

\medskip

Let $F$ be a countable field, and let $f_0=1, f_1, f_2, \dots$ be an
enumeration of the set
$$\Sigma := \{ f\in F[x] \mid f(0)=1 \}.$$
Also, set $F_i=f_0f_1\cdots f_i$ for all $i$. We do not assume any involution on $F$ or its matrix rings; the symbols $w^*_n$ below just denote elements paired with elements $w_n$ in our construction.

We will define  certain
elements $w_n$, $w_n^*$ in $M_n(F)$ for all $n$. We start with
$w_1=w_1^* :=0\in F$. Now for all $n>1$, we put
$$w_n := \sum _{i=1}^{n-1} e_{i+1,i} \in M_n(F).$$

For each $k\in \Z^+$, let $(1,a_{k,1}, \dots ,a_{k,N(k)})$  be the row of
coefficients of the polynomial $F_k$, so that the degree of $F_k$ is
$N(k)$.  Choose positive integers $M(1)< M(2)< \cdots$ such that
$$M(k) \ge 2(k+1)N(k)$$
for all $k$. Define
$$w_n^* := \sum _{i=1}^{n-1} e_{i, i+1} \in M_n(F)$$
for $n=2,\dots , M(1)-1$. In this range, we have $1-w_nw_n^*= e_{11}$ and $1-w_n^*w_n= e_{nn}$. For $k\in\N$ and $M(k)\le n< M(k+1)$, define
$$w_n^* :=\sum _{i=1}^{n-1}e_{i,i+1}- \sum _{j=1}^{N(k)} a_{k,j} (e_{j1} + e_{n,n-j+1}) \in M_n(F).$$
Then we have, for $M(k)\le n< M(k+1)$,
\begin{equation}
\label{rank1eqn}
1-w_nw_n^* = e_{11}+\sum _{j=1}^{N(k)} a_{kj} e_{j+1, 1}  \qquad\quad \text{and} \qquad\quad
1-w_n^*w_n= e_{nn} +\sum _{j=1}^{N(k)} a_{kj}e_{n,n-j}\, .
\end{equation}

\begin{lemma}
\label{wnwn*}
{\rm (a)} $w_nw_n^*w_n= w_n$ and $w_n^*w_nw_n^*= w_n^*$ for all $n$.

{\rm (b)} $1-w_nw_n^*$ and $1-w_n^*w_n$ are rank one idempotents for each $n$, and they are orthogonal for $n\ge 2$.

{\rm (c)} Let $0\le l\le i$. Then $(w_n^*)^lw_n^i(1-w_nw_n^*)= w_n^{i-l}(1-w_nw_n^*)$ for $n\gg 0$.

{\rm (d)} Let $0\le i< l$. Then $(w_n^*)^lw_n^i(1-w_nw_n^*)= 0$ for $n\gg 0$.

{\rm (e)} Let $0\le l\le j$. Then $(1-w_nw_n^*)(w_n^*)^jw_n^l= (1-w_nw_n^*)(w_n^*)^{j-l}$ for $n\gg 0$.

{\rm (f)} Let $0\le j<l$. Then $(1-w_nw_n^*)(w_n^*)^jw_n^l= 0$ for $n\gg 0$.
\end{lemma}

\begin{proof} (a) Direct computation. It is helpful to note that $w_n e_{nj} =0$ for all $j$ and that $e_{i1}w_n=0$ for all $i$.

(b) It is immediate from (a) that $1-w_nw_n^*$ and $1-w_n^*w_n$ are idempotents, and they clearly have rank one. Orthogonality is clear from \eqref{rank1eqn} when $n\ge M(1)$, and from the
relations $1-w_nw_n^*= e_{11}$ and $1-w_n^*w_n= e_{nn}$ when $2\le n< M(1)$.

(c) The relation obviously holds when $l=0$. Assume $i\ge l>0$. Choose $k_0$ such that $M(k_0)> 2i$, and take $n\ge M(k_0)$. Then $M(k)\le n< M(k+1)$ for some $k\ge k_0$.
Now $i< M(k)/2 \le n/2$ and $N(k)\le M(k)/2(k+1)\le n/2(k+1) \le n/4$.

All nonzero entries of $1-w_nw_n^*$ lie in the first column, in rows $1$ through $N(k)+1$. Since $i<n/2< n-N(k)$, we see
that forming $w_n^i(1-w_nw_n^*)$ just amounts to shifting the first column of $1-w_nw_n^*$ down $i$ positions without cutting off any nonzero entry.

Nonzero entries of $w_n^*$ lie in the superdiagonal, or in the first column, or in the last row beyond column $n-N(k)$. Since $i$ is positive, $e_{j1}w_n^i(1-w_nw_n^*) =0$ for all $j$. Also, $N(k)\le n/4$ implies $n-2N(k) \ge n/2> i$, so $N(k)+i < n-N(k)$, and hence $e_{nj} w_n^i (1-w_nw_n^*) =0$ for all $j> n-N(k)$. Thus,
\begin{equation}
\begin{aligned}
\label{leftmultwn*}
w_n^*w_n^i(1-w_nw_n^*) &= \biggl(\, \sum_{m=1}^{n-1} e_{m,m+1} \biggr) w_n^i(1-w_nw_n^*)  \\
 &= (1-e_{nn}) w_n^{i-1}(1-w_nw_n^*) = w_n^{i-1}(1-w_nw_n^*).
\end{aligned}
\end{equation}

We can repeat \eqref{leftmultwn*} (i.e., multiplying on the left by $w_n^*$) as long as $i-1>0$. Part (c) follows.

(d) By (c), $(w_n^*)^iw_n^i(1-w_nw_n^*)= 1-w_nw_n^*$ for $n\gg 0$. Since $w_n^*(1-w_nw_n^*) =0$, (d) follows.

(e) Assume $j\ge l >0$. Choose $k_0$ such that $M(k_0)> 2j$, and take $n\ge M(k_0)$. Then $M(k)\le n< M(k+1)$ for some $k\ge k_0$. Now $j<n/2$ and $N(k)\le n/4$.

First, observe that for $0\le m\le j$, all nonzero entries of $(1-w_nw_n^*)(w_n^*)^m$ lie in the upper
left $(N(k)+1)\times (m+1)$ corner. This is clear for $m=0$, so assume it holds for some $m<j$. Since $m+1<n$,
we have $(1-w_nw_n^*)(w_n^*)^m e_{nt} =0$ for all $t$, so $(1-w_nw_n^*)(w_n^*)^{m+1}$ consists of new entries
in the left column down to row $N(k)+1$, together with the upper left $(N(k)+1)\times (m+1)$ block of $(1-w_nw_n^*)(w_n^*)^m$
shifted one column to the right. Therefore $(1-w_nw_n^*)(w_n^*)^{m+1}$ has the claimed form.

In particular, $(1-w_nw_n^*)(w_n^*)^{j-1}e_{nt} =0$ for all $t$. Since $e_{s1}w_n=0$ for all $s$, it follows that
$$(1-w_nw_n^*)(w_n^*)^jw_n= (1-w_nw_n^*)(w_n^*)^{j-1} \biggl( \sum_{i=1}^{n-1} e_{i,i+1} \biggr) w_n=  (1-w_nw_n^*)(w_n^*)^{j-1} \biggl( \sum_{i=1}^{n-1} e_{ii} \biggr) .$$
This equals $(1-w_nw_n^*)(w_n^*)^{j-1}$, because the latter has no nonzero entries in its last column.

Therefore $(1-w_nw_n^*)(w_n^*)^jw_n= (1-w_nw_n^*)(w_n^*)^{j-1}$. Part (e) follows.

(f) This follows from (e).
\end{proof}

\begin{corollary}
\label{higherwnwn*}
Let $i,j>0$. Then
$$(1-w_n^i(w_n^*)^i)(1-(w_n^*)^jw_n^j) = (1-(w_n^*)^jw_n^j)(1-w_n^i(w_n^*)^i) =0$$
for $n\gg 0$.
\end{corollary}

\begin{proof} By Lemma \ref{wnwn*}(e), $(1-w_nw_n^*)(w_n^*)^{m+j}w_n^j= (1-w_nw_n^*)(w_n^*)^m$ for $m<i$ and $n\gg 0$. Hence,
\begin{align*}
(w_n^m(w_n^*)^m- w_n^{m+1}(w_n^*)^{m+1})(w_n^*)^jw_n^j &= w_n^m(1-w_nw_n^*)(w_n^*)^{m+j}w_n^j  \\
 &= w_n^m(1-w_nw_n^*)(w_n^*)^m= w_n^m(w_n^*)^m- w_n^{m+1}(w_n^*)^{m+1}
 \end{align*}
for $m<i$ and $n\gg 0$. Summing these equations gives $(1-w_n^i(w_n^*)^i)(w_n^*)^jw_n^j= 1-w_n^i(w_n^*)^i$ for $n\gg 0$. Thus, $(1-w_n^i(w_n^*)^i)(1-(w_n^*)^jw_n^j) = 0$ for $n\gg 0$. The second part of the corollary follows from Lemma \ref{wnwn*}(c) in the same way.
\end{proof}

Set $T:=\prod_{n=1}^\infty
M_n(F)$, and write $w :=(w_n)\in T$ and $w^* :=(w_n^*)\in T$. Observe
that every element of the set
$$\Sigma(w) := \{ f(w) \mid f \in \Sigma \}$$
 is invertible in $T$. Let $R$
be the unital subalgebra of $T$ generated by $\soc (T)$, $w$, $w^*$
and all the inverses of the elements of $\Sigma(w)$. Observe that $\soc(R_R)= \soc(_RR)= \soc(T)$.

To simplify computations in $R$, write $\equiv$ for congruence modulo $\soc(T)$. From parts (c)--(f) of Lemma \ref{wnwn*}, we obtain
\begin{equation}
\label{ww*modsoc}
\begin{aligned}
(w^*)^lw^i(1-ww^*) &\equiv \begin{cases} w^{i-l}(1-ww^*)  &(l\le i)\\  0  &(l>i) \end{cases}  \\
(1-ww^*)(w^*)^jw^l &\equiv \begin{cases} (1-ww^*)(w^*)^{j-l}  &(l\le j)\\  0  &(l>j). \end{cases}
\end{aligned}
\end{equation}

\begin{proposition}
\label{structure} Let $M_1$ be the ideal of $R$ generated by
$1-ww^*$ and $\soc (T)$, and let $M_2$ be the ideal of $R$ generated
by $1-w^*w$ and $\soc (T)$. Then $M_1M_2= M_1\cap M_2=\soc(T)$.
\end{proposition}

\begin{proof}
We aim to prove that
\begin{enumerate}
\item For any element $z\in M_1$, the matrix $z_n$ for $n\gg 0$ has all its nonzero entries in
an upper left corner of size less than $(n/2) \times (n/2)$.
\end{enumerate}
 A similar argument
gives that
\begin{enumerate}
\item[(2)] For any element $z\in M_2$, the matrix $z_n$ for $n\gg 0$ has all
its nonzero entries in a lower right corner of size less than $(n/2) \times (n/2)$.
\end{enumerate}
These properties clearly imply the proposition. Before proving them, we establish simplified forms for the elements of $M_1$.

Set $A:= \sum_{j\ge 0} F(1-ww^*)(w^*)^j$. We claim that
\begin{equation}
\label{A+soc}
(1-ww^*)R+\soc(T)= A+\soc(T).
\end{equation}
To prove this, it suffices to show that $A+\soc(T)$ is a right ideal of $R$. Obviously $A+\soc(T)$ is closed under right multiplication by $w^*$, and it follows from \eqref{ww*modsoc} that $A+\soc(T)$ is closed under right multiplication by $w$. It remains to show that
$$(1-ww^*)(w^*)^jf^{-1} \in A+\soc(T)$$
for $j\ge0$ and $f\in \Sigma(w)$. Assume that $\deg f >0$.

First, observe that if $f=1+a_1w+\cdots +a_tw^t$ we have
$$(1-ww^*)f=1-ww^*,$$
so that $(1-ww^*)f^{-1}=1-ww^* \in A+\soc(T)$. This covers the case $j=0$. For $j>0$, we have
\begin{align*}
(1-ww^*)(w^*)^jf &= (1-ww^*) \bigl( (w^*)^j+ a_1(w^*)^jw +\cdots+ a_t(w^*)^jw^t \bigr)  \\
 &\equiv (1-ww^*) \biggl( (w^*)^j+ \sum_{l=1}^{\min\{j,t\}} a_i (w^*)^{j-l} \biggr)
 \end{align*}
by \eqref{ww*modsoc}. Consequently,
\begin{align*}
(1-ww^*)(w^*)^jf^{-1} & \equiv (1-ww^*) \biggl[ (w^*)^jf -  \sum_{l=1}^{\min\{j,t\}} a_i (w^*)^{j-l} \biggr] f^{-1}  \\
 & \equiv (1-ww^*)(w^*)^j -  \sum_{l=1}^{\min\{j,t\}} a_i (1-ww^*)(w^*)^{j-l} f^{-1} \,,
 \end{align*}
which is in $A+\soc(T)$ by induction on $j$.

We have now established \eqref{A+soc}. In particular, it follows that $M_1= RA+\soc(T)$.

Now set $B:= \sum_{l,i\ge0} \sum_{g\in\Sigma(w)} (w^*)^l w^i g^{-1} A$. We claim that
\begin{equation}
\label{B+soc}
M_1= B+\soc(T).
\end{equation}
To prove this, it suffices to show that $B+\soc(T)$ is an ideal of $R$. It is clearly a right ideal, since $A+\soc(T)$ is a right ideal of $R$. Hence, we need only show that $B+\soc(T)$ is closed under left multiplication by $w$, $w^*$, and $f^{-1}$ for $f\in \Sigma(w)$. Closure under left multiplication by $w^*$ is built in.

As for closure under left multiplication by $w$, we need $w (w^*)^l w^i g^{-1} A \subseteq B+\soc(T)$ for $l,i\ge 0$ and $g\in \Sigma(w)$. This is clear if $l=0$, so assume that $l>0$. In that case,
\begin{align*}
w (w^*)^l w^i g^{-1} A &= \bigl[ 1-(1-ww^*) \bigr] (w^*)^{l-1} w^i g^{-1} A  \\
 &\subseteq (w^*)^{l-1} w^i g^{-1} A + (1-ww^*)R \subseteq B+A+\soc(T)= B+\soc(T).
 \end{align*}
 In fact, we shall need a bit more than this, namely
 \begin{equation}
 \label{bitmore}
 w^j (w^*)^l w^i g^{-1} A \subseteq \begin{cases} (w^*)^{l-j} w^i g^{-1} A + \sum_{m<j} w^mA +\soc(T)  &(j\le l)\\  w^{i+j-l}g^{-1} A + \sum_{m<j} w^mA +\soc(T)  &(j>l) \end{cases}
 \end{equation}
for $i,j,l\ge0$, and $g\in \Sigma(w)$, as we see from the above by induction on $j$.

 It remains to show that $f^{-1} (w^*)^l w^i g^{-1} A \subseteq B+\soc(T)$ for $l,i\ge 0$ and $f,g\in \Sigma(w)$. Assume that $\deg f>0$. When $l=0$, we have
 $$f^{-1} w^i g^{-1} A = w^i (fg)^{-1} A \subseteq B$$
 because $fg\in \Sigma(w)$. Now suppose that $l>0$, and write $f=1+a_1w+\cdots +a_tw^t$. Then
 $$fw^*= w^*+ a_1ww^*+ \cdots+ a_tw^tw^* = w^*+ (a_1+ \cdots+ a_tw^{t-1}) - (a_1+ \cdots+ a_tw^{t-1})(1-ww^*),$$
 and consequently
\begin{equation}
\label{finvw*}
f^{-1}w^*= w^* - f^{-1}(a_1+ \cdots+ a_tw^{t-1}) + f^{-1}(a_1+ \cdots+ a_tw^{t-1})(1-ww^*).
\end{equation}
Now $w^* \bigl( (w^*)^{l-1} w^i g^{-1} A \bigr) \subseteq B$ by definition of $B$, while
$$f^{-1}w^j (1-ww^*) \bigl( (w^*)^{l-1} w^i g^{-1} A \bigr) \subseteq f^{-1}w^j (1-ww^*)R \subseteq f^{-1}w^j A +\soc(T) \subseteq B+\soc(T)$$
for all $j$, by \eqref{A+soc} and the definition of $B$. Finally,
\begin{align*}
f^{-1} w^j \bigl( (w^*)^{l-1} w^i g^{-1} A \bigr) &\subseteq \begin{cases} f^{-1} \bigl( (w^*)^{l-1-j} w^i g^{-1} A + \sum_{m<j} w^mA +\soc(T) \bigr)  &(j< l)\\  f^{-1} \bigl( w^{i+j-l+1}g^{-1} A + \sum_{m<j}  w^mA +\soc(T) \bigr)  &(j\ge l) \end{cases}  \\
 &\subseteq \begin{cases} f^{-1} (w^*)^{l-1-j} w^i g^{-1} A + \sum_{m<j} w^m f^{-1}A +\soc(T)  &(j< l)\\   w^{i+j-l+1}(fg)^{-1} A + \sum_{m<j} w^m f^{-1}A +\soc(T)  &(j\ge l) \end{cases}  \\
 &\subseteq B+\soc(T)
 \end{align*}
by \eqref{bitmore} and induction on $l$. In view of \eqref{finvw*} and the above inclusions, we obtain
$$f^{-1} (w^*)^l w^i g^{-1} A = f^{-1} w^* \bigl( (w^*)^{l-1} w^i g^{-1} A \bigr) \subseteq B+\soc(A),$$
which concludes the final induction on $l$.

We have now verified \eqref{B+soc}. Taking this together with \eqref{A+soc}, we see that
\begin{equation}
\label{M1form}
M_1= \sum_{l,i,j\ge0} \sum_{g\in\Sigma(w)} F(w^*)^l w^i g^{-1} (1-ww^*)(w^*)^j +\soc(T).
\end{equation}

Property (1) is trivially satisfied for elements $z\in \soc(T)$, since then $z_n=0$ for $n\gg 0$. In view of \eqref{M1form}, it only remains to verify (1) for any $z= (w^*)^l w^i g^{-1} (1-ww^*)(w^*)^j$ where $l,i,j\ge 0$ and $g\in \Sigma(w)$.

Choose $k_0\ge 6$ such that $M(k_0) > 14\max\{i,j\}$ and $g= f_t$ for some $t\le k_0$. Then for all $k\ge k_0$, $g$ divides the polynomial
$F_k$, so $F_k=gG_k$ for some $G_k\in F[x]$. Recall that
$$F_k(x)= 1+a_{k,1}x+ \cdots+ a_{k,N(k)}x^{N(k)}.$$
In particular, $\deg G_k \le \deg F_k = N(k)$.

 For $M(k)\le n< M(k+1)$, we have
 $$F_k(w_n)e_{11} = e_{11}+ \sum_{i=1}^{N(k)} a_{k,i}e_{i+1,1} = 1-w_nw_n^*,$$
 whence
$F_k(w_n)^{-1} (1-w_nw_n^*)=e_{11}$.
It follows that
$$g(w_n)^{-1}(1-w_nw_n^*)= G_k(w_n)F_k(w_n)^{-1}(1-w_nw_n^*)=G_k(w_n)e_{11}\,,$$
so that all nonzero entries of $g(w_n)^{-1}(1-w_nw_n^*)$ lie in the left column and the first $N(k)+1$ rows.

Note that $n\ge M(k_0)\ge k_0\ge 6$ and $N(k)\le M(k)/2(k+1)\le n/14$, while $i< M(k)/14\le n/14$. In particular, $N(k)+i+2 < n/2$.

Since left multiplication by $w_n$ just moves matrix entries down one row, all nonzero entries of $w_n^i g(w_n)^{-1} (1-w_nw_n^*)$ lie in the left column and the first $N(k)+i+1$ rows, and hence in the upper left $(\lfloor n/2\rfloor-1)\times1$ corner. Left multiplication of $w_n^*$ on any matrix concentrated in this corner does not place nonzero entries outside that corner, because $N(k)< n/2$ and $n/2 < n-N(k)$. Thus, all nonzero entries of $(w_n^*)^l w_n^i g(w_n)^{-1} (1-w_nw_n^*)$ lie in the upper left $(\lfloor n/2\rfloor-1)\times1$ corner. Finally, for any matrix whose nonzero entries are in the upper left $(\lfloor n/2\rfloor-1)\times s$ corner for some $s<n$,  right multiplication by $w_n^*$ does not place nonzero entries outside the upper left  $(\lfloor n/2\rfloor-1)\times (s+1)$ corner. Since $j< n/14$ and so $j+2 < n/2$, we conclude that all nonzero entries of $z_n$ lie in the upper left  $(\lfloor n/2\rfloor-1)\times (\lfloor n/2\rfloor-1)$ corner. This establishes property (1)
 .

Property (2) is proved similarly.
\end{proof}

\begin{lemma}
\label{lem:standdecom} Let $R$ be a unital ring, let $n\ge 2$, and
let $a,a^*\in R$ such that $a=aa^*a$ and $a^*=a^*aa^*$. Assume that for
all $1\le i\le n$ we have
\begin{equation}
\label{eq:orth}
(1-a^i(a^*)^i)(1-(a^*)^ia^i)=0=(1-(a^*)^ia^i)(1-a^i(a^*)^i).
\end{equation}
Then the following properties hold:
\begin{enumerate}
\item $(a^i(a^*)^i)((a^*)^ja^j)= ((a^*)^ja^j)(a^i(a^*)^i)$ for all $1\le i,j\le
n$.
\item $a^i=a^i(a^*)^ia^i$ and $(a^*)^i= (a^*)^ia^i(a^*)^i$ for all $1\le i\le
n$. In particular, $a^i(a^*)^i$ and $(a^*)^ja^j$ are commuting
idempotents for all $1\le i,j \le n$, and
$$aa^* \ge a^2(a^*)^2 \ge \cdots\ge a^n(a^*)^n\,, \qquad\qquad a^*a \ge (a^*)^2a^2\ge \cdots\ge (a^*)^na^n \,.$$
\item For $0\le k\le n-1$, set
$f_k=a^k(1-aa^*)(a^*)^k$ and
 $g_k=(a^*)^k(1-a^*a)a^k$. Then
 \begin{enumerate}
 \item $f_0,f_1,\dots,f_{n-1},g_0,g_1,\dots,g_{n-1}$ are pairwise orthogonal idempotents.
 \item $af_ja^*=f_{j+1}$ and $a^*g_ja=g_{j+1}$ for $0\le j\le n-2$.
 \item $a^*f_ja=f_{j-1}$ and $ag_ja^*=g_{j-1}$ for $1\le j \le n-1$.
 \item $f_0\sim f_1\sim \cdots\sim f_{n-1}$ and $g_0\sim g_1\sim \cdots\sim g_{n-1}$.
 \end{enumerate}
\end{enumerate}
\end{lemma}

\begin{proof}
(1), (2) It follows from (\ref{eq:orth}) that
\begin{equation}
\label{eq:commut1} [a^i(a^*)^i,\, (a^*)^ia^i]=0 \quad \forall i\in
\{1,\dots ,n\}.
\end{equation}
It then follows that $a^2(a^*)^2$ and $(a^*)^2a^2$ are
commuting idempotents. For instance,
$$a^2(a^*)^2a^2= a(aa^*)(a^*a)a= (aa^*a)(aa^*a)= a^2$$
and thus $a^2(a^*)^2$ is an idempotent. We clearly have
$a^2(a^*)^2\le aa^*$, so that $1-aa^* \le 1-a^2(a^*)^2$ and
$$(1-aa^*)(1-(a^*)^2a^2)=0=(1-(a^*)^2a^2)(1-aa^*).$$
This shows that $[aa^*,\, (a^*)^2a^2]=0$ and similarly $[a^*a,\,
a^2(a^*)^2]=0$. It follows from these relations that
$a^3(a^*)^3a^3=a^3$ and $(a^*)^3a^3(a^*)^3=(a^*)^3$, and thus
$a^3(a^*)^3$ and $(a^*)^3a^3$ are commuting idempotents. We have
$aa^*\ge a^2(a^*)^2\ge a^3(a^*)^3$ and $a^*a\ge (a^*)^2a^2\ge
(a^*)^3a^3$ and thus
$$(1-a^i(a^*)^i)(1-(a^*)^3a^3)=0=(1-(a^*)^3a^3)(1-(a^*)^ia^i)$$
for $i=1,2$, and so
$$[a^i(a^*)^i,\, (a^*)^ja^j]=0\qquad \text{for } 1\le i,j\le 3\, .$$
We can continue in this way to prove the results. We also obtain
\begin{equation}
\label{moreorthog}
(1-a^i(a^*)^i)(1-(a^*)^ja^j) = (1-(a^*)^ja^j)(1-a^i(a^*)^i) = 0
\end{equation}
for $1\le i,j\le n$.

(3) Part (b) is clear.

It follows from \eqref{moreorthog} that
$$(a^j(a^*)^j)((a^*)^ka^k)=a^j(a^*)^j+(a^*)^ka^k-1=((a^*)^ka^k)(a^j(a^*)^j)$$
for $1\le j,k\le n$. Using this, we get that
$$f_jg_k=0=g_kf_j$$
for $0\le j,k\le n-1$. Part (a) follows. Since each $f_j$ is orthogonal to $g_0= 1-a^*a$ and each $g_j$ is orthogonal to $f_0= 1-aa^*$, we obtain
$$a^*f_{j+1}a = a^*af_ja^*a= f_j \qquad\quad \text{and} \qquad\quad ag_{j+1}a^*= aa^*g_jaa^*= g_j$$
for $0\le j\le n-2$, verifying (c).

Now for $1\le j\le n-1$, we have $a^*(f_ja) = f_{j-1}$ and $(f_ja)a^*= f_j$, whence $f_{j-1}\sim f_j$. Similarly, $a(g_ja^*)= g_{j-1}$ and $(g_ja^*)a= g_j$, so $g_{j-1} \sim g_j$. This establishes (d).
\end{proof}

\begin{theorem}
\label{thm:realizing} The algebra $R$ constructed above is regular, and $V (R)\cong \calM$.
\end{theorem}

\begin{proof}
We begin by applying Construction \ref{constrQ} with $K=F$.
Let $Q_1$ be the universal localization $\Sigma^{-1}F\langle x,y\mid
xy=1\rangle$, where $\Sigma$ (as above) is the set $\{ f\in F[x] \mid f(0)=1\}$. Then $Q_1$ is regular, and it has a unique proper nonzero ideal
$I_1$, such that $Q_1/I_1\cong F(x)$. Moreover, $I_1 =\soc(Q_1)$, this ideal is homogeneous (as a right or left semisimple $Q_1$-module), and
$Q_1\cong Q_1\oplus Q_1e$ for any idempotent $e\in I_1$.

Let $Q$ be the algebra in the pullback
$$\xymatrixrowsep{3pc} \xymatrixcolsep{6pc}
\xymatrix{
Q \ar[r]^{\pi_1} \ar[d]_{\pi_2}  &Q_1 \ar[d]^{\pi'_1}  \\
Q_1^{\text{opp}} \ar[r]^{\pi'_2}  &F(x) }$$
of Construction \ref{constrQ}.
By Proposition \ref{prop:realizingMbar}, $Q$ is regular, and $\Mbar \cong V(Q)$
through a specified isomorphism $\psi $.

We claim that $R/\soc (R)\cong Q$. We first note that $ww^*=1$ but
$w^*w\ne 1$ in $R/M_1$. Moreover, all elements $f(w)$, with $f\in
\Sigma$, are invertible in $R/M_1$. It follows that there is a
unique homomorphism $\gamma_1 : Q_1\to R/M_1$ sending $x$ to $w$ and $y$ to
$w^*$. Note that this map is surjective, because $R/M_1$ is
generated by $1$, $w$, $w^*$, and $f^{-1}$ for $f\in \Sigma(w)$. Since $1\ne w^*w$ in $R/M_1$,
the map is also injective. Similarly, there is an isomorphism
$\gamma_2: Q_1^{\text{opp}}\to R/M_2$ sending  $x$ to $w$ and $y$ to $w^*$.
Note that the compositions of the $\gamma_i^{-1}$ with the quotient maps $R \rightarrow R/M_i$ induce a homomorphism
$\gamma\colon R\to Q$ by the pullback property of $Q$. The kernel of
$\gamma$ is precisely $\soc (R)$ by Proposition \ref{structure}. Since $M_1+M_2$ is a maximal ideal of $R$, we see that $\ker \pi'_1\pi_1\gamma = M_1+M_2$, from which it follows easily that the map $\gamma$ is surjective. Thus, $\gamma$ is
an isomorphism from $R/\soc (R)$ onto $Q$. This shows that $R$ is regular, since $\soc(R)= \soc(T)$ is regular.

Finally we show that $\calM\cong V(R)$.

For $n\in\N$, let $h_n: T \rightarrow M_n(F)$ be the canonical projection, and set $H_n:= h_1+\cdots+ h_n$ and $G_n:= 1-H_n$. Because of Corollary \ref{higherwnwn*}, there exist integers $2\le K_1< K_2< \cdots$ such that
$$G_{K_n}(1-w^i(w^*)^i)(1-(w^*)^iw^i) = G_{K_n}(1-(w^*)^iw^i)(1-w^i(w^*)^i) =0$$
for $1\le i\le n+1$. It
then follows from Lemma \ref{lem:standdecom} that
\begin{gather*}
G_{K_n}(1-ww^*),\ G_{K_n}(ww^*- w^2(w^*)^2),\ \dots,\ G_{K_n}(w^n(w^*)^n-w^{n+1}(w^*)^{n+1})  \\
G_{K_n}(1-w^*w),\ G_{K_n}(w^*w- (w^*)^2w^2),\ \dots,\ G_{K_n}((w^*)^nw^n-(w^*)^{n+1}w^{n+1})
\end{gather*}
are lists of orthogonal, equivalent idempotents. In particular,
\begin{align*}
\rank\,  (w^i(w^*)^i- w^{i+1}(w^*)^{i+1})_t &= \rank\,  ((w^*)^iw^i- (w^*)^{i+1}w^{i+1} )_t = 1  \\
\rank\,  (1-w^{n+1}(w^*)^{n+1})_t &= \rank\,  (1-(w^*)^{n+1}w^{n+1})_t = n+1  \\
\rank\,  (w^{n+1}(w^*)^{n+1})_t &= \rank\,  ((w^*)^{n+1}w^{n+1})_t = t-n-1
\end{align*}
for $0\le i\le n$ and $t>K_n$.

We shall define a map $\tau: \calM\to V(R)$, starting by assigning values to the generators $a_n$, $x_n$, $y_n$, $z_n$ of $\calM$. First, set $\tau(a_n)= [p_n]$ where $p_n:= e_{11}\in M_n(F)$, while
\begin{align*}
\tau(x_0) &= [ww^*]= [w^*w]  &\tau(y_0) &= [1-ww^*]  &\tau(z_0) &= [1-w^*w].
\end{align*}
For $n>0$, set
$$\tau(x_n)= [H_{K_n}g_n+ G_{K_n}w^{n+1}(w^*)^{n+1}]= [H_{K_n}g_n+ G_{K_n}(w^*)^{n+1}w^{n+1}],$$
where $g_n$ is an idempotent in $\soc(T)$ such that $(g_n)_t \perp p_t$ for all $t$ and
$$\rank\, (g_n)_t= \begin{cases} t-n-1  &(n+1\le t\le K_n)\\  0  &(\text{otherwise}), \end{cases}$$
and set
\begin{align*}
\tau(y_n) &= [p_{n+1}+ \cdots+ p_{K_n}+ G_{K_n}(w^n(w^*)^n-w^{n+1}(w^*)^{n+1})]  \\
\tau(z_n) &= [p_{n+1}+ \cdots+ p_{K_n}+ G_{K_n}
((w^*)^nw^n-(w^*)^{n+1}w^{n+1})].
\end{align*}

The next step is to show that the elements $\tau(a_n)$, $\tau(x_n)$, $\tau(y_n)$, $\tau(z_n)$ in $V(R)$ satisfy the defining relations of $\calM$. Clearly $\tau(x_0)+\tau(y_0) = [1_R]= \tau(x_0)+ \tau(z_0)$. Since
$$\rank\bigl( p_t+ (g_1)_t \bigr) = t-1= \rank\bigl( (ww^*)_t \bigr)$$
for $2\le t\le K_1$, we see that
$$\tau(x_1)+\tau(y_1)= [p_2+\cdots+ p_{K_1}+ H_{K_1}g_1]+ [G_{K_1}ww^*] = [ww^*]= \tau(x_0).$$
For $n>1$, we have
$$\rank\bigl( p_t+ (g_n )_t \bigr) = t-n= \begin{cases} \rank\bigl( (g_{n-1})_t \bigr)  &(n+1\le t\le K_{n-1})\\  \rank\bigl( (w^n(w^*)^n)_t \bigr)  &(K_{n-1}<t \le K_n) \,, \end{cases}$$
and so
\begin{align*}
\tau(x_n)+\tau(y_n) &= [p_{n+1}+\cdots+ p_{K_n}+ H_{K_n}g_n]+ [G_{K_n}w^n(w^*)^n]  \\
 &= [H_{K_{n-1}}g_{n-1}]+ [G_{K_{n-1}}w^n(w^*)^n]= \tau(x_{n-1}).
 \end{align*}
Similarly, $\tau(x_n)+\tau(z_n)= \tau(x_{n-1})$ for all $n>0$.

We also have
\begin{align*}
\tau(y_1)+\tau(a_1) &= [p_1+ \cdots+ p_{K_1}]+ [G_{K_1}(ww^*-w^2(w^*)^2)]  \\
 &= [p_1+ \cdots+ p_{K_1}]+ [G_{K_1}(1-ww^*)]=  [1-ww^*] = \tau(y_0)
 \end{align*}
and
\begin{align*}
\tau(y_n)+\tau(a_n) &= [p_n+ \cdots+ p_{K_n}]+ [G_{K_n}(w^n(w^*)^n-w^{n+1}(w^*)^{n+1})]  \\
 &=  [p_n+ \cdots+ p_{K_n}]+ [G_{K_n}(w^{n-1}(w^*)^{n-1}-w^n(w^*)^n)]  \\
 &=  [p_n+ \cdots+ p_{K_{n-1}}]+ [G_{K_{n-1}}(w^{n-1}(w^*)^{n-1}-w^n(w^*)^n)] = \tau(y_{n-1})
 \end{align*}
for $n>1$. Similarly, $\tau(z_n)+\tau(a_n)= \tau(z_{n-1})$ for all $n>0$.

Therefore the defining relations \eqref{MECrelns} of $\calM$ are satisfied by $\tau(a_n)$, $\tau(x_n)$, $\tau(y_n)$, $\tau(z_n)$, and consequently these
assignments extend to a well-defined homomorphism $\tau \colon \calM\to V(R)$. Recall from \cite[Lemma 4.8]{TW} that $\ped(\calM) = \bigoplus_{n=1}^\infty \Z^+ a_n$. Clearly, the restriction of $\tau $ to $\ped (\calM)$ induces an isomorphism from $\ped (\calM)$ onto $V(\soc (R))$,
and the induced map $\Mbar =\calM/\ped (\calM)\to V(R/\soc (R)) \cong V(Q)$ is the isomorphism $\psi $ described in Proposition \ref{prop:realizingMbar}.
Moreover $\tau (u)=[1_R]$. We can therefore apply Proposition \ref{prop:generaliso} to conclude that $\tau $ is an isomorphism.
\end{proof}

\section{Relationship with Atiyah's Problem}
\label{sect:Atiyah}

Let $\Gamma$ be a discrete countable group, and denote by $\mathcal N (\Gamma )$ the von Neumann algebra of $\Gamma $ and by $\mathcal U (\Gamma )$
the regular ring of $\Gamma $, which is the classical ring of quotients of $\mathcal N (\Gamma )$. We refer the reader to \cite{luck} for information about
these algebras and background on the Atiyah Conjecture.

We recall the definition of a rank ring. Although the main interest is in regular rank rings, it is convenient here to define the notion for arbitrary
(unital) rings.

\begin{definition}
 \label{def:rankring} A \emph{rank function} on a unital ring $R$ is a function $N \colon R\to [0,1]$ satisfying the following properties:
 \begin{enumerate}
  \item $N(a) = 0$ if and only if $a= 0$, and $N(1)=1$.
  \item $N (a+b)\le N (a) + N (b)$ for all $a,b\in R$.
  \item $N (ab)\le N (a), N (b)$ for all $a,b\in R$.
  \item If $e,f \in R$ are orthogonal idempotents, then $N(e+f)=N (e) +N (f) $.
 \end{enumerate}
A \emph{regular rank ring} is a pair $(R,N )$, where $R$ is a regular ring and $N$ is a rank function on $R$. We will occasionally omit the reference to
$N$.
\end{definition}

We refer the reader to \cite[Chapter 16]{vnrr} for further information about regular rank rings. If $(R,N)$ is a regular rank ring, then $N$ extends to an
unnormalized rank function on $M_n(R)$ for all $n\ge 1$, also denoted by $N$, such that $N(I_n)= n$, where $I_n$ is the identity of $M_n(R)$.

A \emph{$*$-regular ring} is a regular ring endowed with a proper involution, i.e. $x^*x=0$ implies $x= 0$ for $x\in R$.
The $*$-transpose involution on $M_n(R)$ is proper if and only if $*$ is $n$-positive definite, i.e., $\sum _{i=1}^n x_i^*x_i =0$
implies $x_i =0$ for all $i=1,\dots ,n$. The regular ring $\mathcal U (\Gamma)$ is a $*$-regular ring, with a positive definite involution,
see \cite[Chapter 8]{luck}. If $R$ is a $*$-regular ring, then for each element $a$ in $R$ there is a unique $b\in eRf$,
termed \emph{the relative inverse of $a$}, such that $ab=f$ and $ba=e$,
where $f:=LP(a)$ and $e:=RP(a)$ are the left and right projections of $a$, that is, the unique projections in $R$ such that $aR=LP(a)R$
and $Ra=R\cdot RP(a)$, see \cite[Proposition 51.4(i)]{Berb}. All projections of $\mathcal U (\Gamma)$ are contained in $\mathcal N (\Gamma)$.

Moreover, $\mathcal U (\Gamma)$ is a regular rank ring, with a canonical rank function $\rk $ defined by $\rk (a) := \rm{tr}(LP(a))=
\rm{tr}(RP(a))$, where $\rm{tr}$ is the canonical trace on $\mathcal N (\Gamma)$.

We now show the existence of the $*$-regular closure of a subset of a $*$-regular ring, see \cite{elek},
\cite{LS}. Our proof is slightly different from that of
\cite[Proposition 3.1]{LS}. In particular, we observe that there is no need of any additional
hypothesis on the $*$-regular ring to define this closure. Moreover, the $*$-regular closure
has a description which is similar to that of the division closure of a subring of a ring.

\begin{proposition}
\label{prop:*regularclosure}
Let $R$ be a $*$-regular ring and let $S$ be a $*$-subring of $R$. Then there is a smallest $*$-regular subring
$\mathcal R (S, R)$ of $R$ containing $S$. Moreover, there is an increasing sequence
$$S= \mathcal R _0 (S,R)\subseteq \mathcal R _1 (S,R)\subseteq \mathcal R_2 (S, R)\subseteq \cdots $$
of $*$-subrings of $R$ such that $\mathcal R_{i+1}(S,R)$ is generated by $\mathcal R_i (S,R)$ and the relative inverses in
$R$ of the elements of $\mathcal R _i (S,R)$, and $\mathcal R (S, R)= \bigcup _{i=0}^{\infty} \mathcal R _i (S, R)$.
\end{proposition}

\begin{proof}
 We first show that any intersection $T:=\bigcap_{i\in I} R_i$ of $*$-regular subrings $R_i$ of $R$ is a $*$-regular subring.
 If $a\in T$ then there is a unique $b\in eRf$ such that $ab=f$ and $ba= e$, where $e=RP(a)$ and $f=LP(a)$. By uniqueness we have $b\in  R_i$
 for all $i\in I$ so $b\in \bigcap _{i\in I}R_i=T$. Thus $aba=a$ and $bab=b$, showing that $T$ is regular. Since $T$ is a $*$-subring of $R$, $R$ is
 $*$-regular.

 Clearly, there is a smallest $*$-regular subring of $R$ containing $S$, the intersection of all the $*$-regular subrings of $R$ containing $S$.

 For the second part, set $\mathcal R _0(S,R)=S$ and define $\mathcal R _{i+1}(S,R)$ inductively, by letting $\mathcal R _{i+1}(S,R)$
 be the $*$-subring of $R$ generated by $\mathcal R_i (S,R)$ and the relative inverses of the elements of $\mathcal R_i (S,R)$ in $R$.
 It is clear that $\bigcup _{i=0}^{\infty} \mathcal R _i (S, R)$ is the smallest $*$-regular subring of $R$ containing $S$.
 \end{proof}

If $X$ is a subset of a $*$-regular ring $R$, then the \emph{$*$-regular closure} $\mathcal R (X, R)$ of $X$ in $R$ is defined as the smallest
$*$-regular ring of $R$ containing the $*$-subring of $R$ generated by $X$. Note that the $*$-regular closure of a $*$-subring $S$ of $R$
always contains the division closure, which is the smallest subring $T$ of $R$ containing $S$ which is closed under inversion.
(Of course, this is due to the fact that the relative inverse of an invertible element is the inverse of the element.)
Also observe that the $*$-regular closure of a countable ring is also countable. This applies to the $*$-regular rings $\mathcal R (k\Gamma, \mathcal U (\Gamma))$,
where $k$ is a countable subfield of $\mathbb C$ and $\Gamma$ is (as usual) a countable group.

G. Elek has shown in \cite{elek} that the $*$-regular closure $R(\Gamma )$ of $\mathbb C \Gamma  $ in $\mathcal U (\Gamma)$ coincides with the $*$-regular closure of $\mathbb C \Gamma$
taken in some other $*$-regular overrings of $\mathbb C \Gamma $, whenever $\Gamma $ is a finitely generated \emph{amenable} group.

The question of how big is the $*$-regular closure $\mathcal R (k\Gamma, \mathcal U (\Gamma))$ for a subfield $k$ of $\mathbb C$ closed under complex conjugation
is closely related to the Atiyah Problem for $\Gamma$ and $k$.

We note the following important property of regular rank rings. For a proof see \cite{GM}.

\begin{theorem}
 \label{thm:rrrareur} Let $(R, N)$ be a regular rank  algebra over an uncountable field. Then $R$ is unit-regular.
 In particular, $\mathcal R (k\Gamma, \mathcal U (\Gamma))$ is unit-regular for every uncountable subfield $k$ of $\mathbb C$.
\end{theorem}

Note that the main example in Section \ref{sect:realizMbyvnrrr} shows that the above result is not true for algebras over countable fields, see also
\cite{CL}.

It is an old open problem, attributed to Handelman,  whether every $*$-regular ring is unit-regular, see \cite[Open Problem 48]{vnrr}.
In view of the above, the following question seems quite pertinent.

\begin{question}
\label{ques:unit-regular} Let $k$ be a countable subfield of $\mathbb C$, closed under complex conjugation, and let $G$ be a countable group.
 Is then the $*$-regular closure of $k \Gamma$ in $\mathcal U(\Gamma)$ a unit-regular ring?
\end{question}

A negative answer would solve Handelman's question in the negative. A positive answer
would be interesting in itself.

We now consider the lamplighter group, which is the wreath product $G:=\Z_2 \wr \Z=\bigl( \bigoplus_{i\in \Z} \Z_2 \bigr)\rtimes \Z$, see e.g. \cite{DS}.
We denote by $t$ the generator corresponding to $\Z$ and by $a_i$ the generator corresponding to the $i$-th copy of $\Z _2$.
We have $t^{-1}a_it = a_{i+1}$, and $a_ia_j=a_ja_i$ for all $i,j$. Let $e_i := \frac{1+a_i}{2}$ and $f_i :=1-e_i$, and set $s= e_0t$. It was shown by
Grigorchuk and Zuk \cite{GZ} that the trace of the spectral projection of $s+s^*$ corresponding to $0$ is $1/3$.
This gave the first counterexample to Atiyah's Conjecture \cite{GLSZ}.
This was generalized by Dicks and Schick \cite{DS} to the groups $\Z _p\wr \Z$. These authors computed the traces of the spectral projections working directly
on $\ell ^2 (\Z _p\wr \Z )$. We will use some of the computations in \cite{DS}.

From now on, $k$ will denote a subfield of $\mathbb C$ closed under conjugation, endowed with the involution given by complex conjugation.

Recall from Section \ref{sect:exchangerealization} the definition of the monogenic free inverse monoid $\mathcal F$.
The next result identifies the $*$-subalgebra of $k G$ generated by $s=e_0t$ with the semigroup algebra  $k [\mathcal F ]$.

\begin{proposition}
 \label{prop:Aembeds-in-U}
 Let $\mathcal F$ be the monogenic free inverse monoid. Then there exists a $*$-algebra embedding of
 $k[\mathcal F]$ into $k G$ which sends the canonical generator $s$ of $\mathcal F $ to $e_0t$.
 \end{proposition}

\begin{proof} Set $A :=k[\mathcal F ]$.
Denote momentarily by $\widetilde{s}$ the element $e_0t$.
Then note that $\widetilde{s}\widetilde{s}^*= e_0$ and $\widetilde{s}^* \widetilde{s} = t^{-1}e_0 t= e_1$, so $\widetilde{s}$
is a partial isometry. Observe that
\begin{align*}
\widetilde{s}^i( \widetilde{s}^*)^i &= e_{-i+1}e_{-i+2}\cdots e_{-1}e_0  &(\widetilde{s}^*)^j \widetilde{s}^j &= e_1e_2\cdots e_j
\end{align*}
for all $i,j \in \Z^+$.
In particular, all the projections $\widetilde{s}^i( \widetilde{s}^*)^i , (\widetilde{s}^*)^j \widetilde{s}^j $, $i,j \ge 0$, commute with each other,
and so by Lemma \ref{Fpres} there is a $*$-algebra map $\varphi \colon A\to k G$, with $\varphi(s)= \widetilde{s}$, whose image is the $*$-subalgebra generated by $\widetilde{s}$.

Since $\soc (A) = \bigoplus_{n=0}^\infty h_nA \cong \bigoplus_{n=0}^\infty M_{n+1}(k)$ is essential in $A$ (Lemma \ref{lem:centralprojhn} and Proposition \ref{prop:S_1=socA}), to show that $\varphi $
is injective, it suffices to prove that $\varphi(h_n) \ne 0$ for all $n\ge 0$. For this, it is enough to show that $\varphi (q_{-i,j})\ne 0$
for all $i,j\ge 0$, where $q_{-i,j}$ is as in \eqref{q-ij}.
We have
$$\varphi (q_{-i,j}) = f_{-i}(e_{-i+1}\cdots e_{-1}e_0e_1\cdots e_j)f_{j+1} ,$$
and $\rk (f_{-i}(e_{-i+1}\cdots e_{-1}e_0e_1\cdots e_j)f_{j+1} )= 2^{-(i+j+2)}$, so $\varphi (q_{-i,j}) \ne 0$.
This completes the proof.
\end{proof}

From now on, we will denote by $A$ the $*$-subalgebra of $k G$ generated by $s=e_0t$, and we will identify this subalgebra
with $k [ \mathcal F ]$. As in Section \ref{sect:exchangerealization}, set $h_n :=\sum _{i=0}^n  q_{-i,n-i}$ for $n\ge 0$.
The $h_n$ are pairwise orthogonal central projections in $A$ and $\rk (h_n)=(n+1)2^{-(n+2)}$. It follows that $\sum _{n=0}^{\infty} \rk(h_n)= 1$, cf.
the computation in \cite[Lemma 3.6]{DS}, and so $\sum _{n=0}^{\infty} h_n =1$ in the strong topology of $\mathcal N (G)$.
We obtain:

\begin{proposition}
 \label{prop:prodhnaembedsinU} With the above notation, we have that
$\prod _{n=0}^{\infty} h_n A\cong \prod_{n=0}^{\infty} M_{n+1}(k)$
embeds unitally in $\mathcal U (G)$. Therefore
$$\mathcal R (A,  \mathcal U (G))= \mathcal R (A, \prod _{n=0}^{\infty} h_n A).$$
\end{proposition}

\begin{proof}
 Since $\sum _{n=0}^{\infty} h_n =1$, it follows that the annihilator in $\mathcal U (G)$  of the family $\{ h_n : n\ge 0 \}$
 is $0$. Since $\mathcal U (G)$ is a self-injective ring, it follows that $\prod_{i=0}^{\infty} h_n A$ embeds unitally in $\mathcal U (G)$. The final statement follows because the image of $\prod_{n=0}^{\infty} h_n A$ is a $*$-regular subring of $\mathcal U(G)$.
\end{proof}

Observe that $\mathcal R (A,\mathcal U (G))\subseteq \mathcal R (kG, \mathcal U (G))$. Using Proposition \ref{prop:prodhnaembedsinU}, we are going
to identify some of the subalgebras of $\mathcal R (A, \mathcal U(G))$.  We denote by $\Z[\frac{1}{2}]$ the subring of $\Q$ generated by $\frac{1}{2}$.

\begin{proposition}
 \label{ranksatSigmaA}
 Let $A$ be the algebra defined before and set $\Sigma= \{ f(s)\in k[s]\mid f(0)=1 \}$. Then $\Sigma ^{-1}A$
 is isomorphic to a subalgebra $\mathcal B$ of $\mathcal U (G)$ containing $A$, and $\rk (e)\in \Z[\frac{1}{2} ]$ for all idempotents $e$
 in $M_n(\mathcal B)$, $n\ge 1$.
 \end{proposition}

\begin{proof}
 By Proposition \ref{prop:socSigmaA}, $\Sigma^{-1}A$ is naturally embedded in $\prod _{n=0}^{\infty} h_n A$, and so, by
 Proposition \ref{prop:prodhnaembedsinU}, $\Sigma ^{-1}A$ is naturally isomorphic to a subalgebra $\mathcal B$ of $\mathcal U (G)$.
 Observe that $\mathcal B $ is contained in the division closure of $A$ in $\mathcal U (G)$. Note that the rank function $\rk$ induces a
 state $\sigma$ on $V(\mathcal B)$ by $\sigma([e])= \rk (e)$ for any idempotent matrix over $\mathcal B$. It suffices to show that the image of $\sigma$ is contained in $\Z[\frac{1}{2} ]$.

Identify $\Sigma^{-1}A$ with $\mathcal B$, and let $\tau : \calM \to V(\mathcal B)$ be the isomorphism constructed in the proof of Theorem \ref{thm:exchangerealization}. We apply $\sigma\tau$ to the generators of $\calM$ and compute that
 \begin{align*}
 \sigma\tau(x_n) &= \sigma( [s^{n+1}(s^*)^{n+1}] ) = \rk (e_{-n} e_{-n+1} \cdots e_{-1} e_0) = 2^{-(n+1)}  \\
 \sigma\tau(y_n) &= \sigma( [s^n(1-ss^*)(s^*)^n] ) = \rk( f_{-n} e_{-n+1} \cdots e_{-1} e_0) = 2^{-(n+1)}
 \end{align*}
 for all $n\ge 0$. In view of \eqref{MECrelns}, it follows that $\sigma\tau(z_n) = \sigma\tau(y_n) = 2^{-(n+1)}$ for all $n\ge 0$ and $\sigma\tau(a_n) = \sigma\tau(y_{n-1})- \sigma\tau(y_n) = 2^{-(n+1)}$ for all $n>0$. Therefore the image of $\sigma\tau$, which equals the image of $\sigma$, is contained in $\Z[\frac{1}{2} ]$.
  \end{proof}

  Since $\mathcal B$ is contained in $\mathcal R _1(A, \prod_{n=0}^{\infty} h_nA)$, so is the algebra generated by $\mathcal B$ and $\mathcal B^*$. We are now going to describe this algebra.

  We will need to consider the algebra of rational series in one variable over $k$, endowed with the Hadamard product $\odot$.
   We will denote this algebra by $\mathcal R ^{{\rm o}}$. See \cite{BR} for details. The elements of $\mathcal R ^{{\rm o}}$
   are the formal power series in $k[[x]]$ which arise as the expansions of rational functions of the form $P(x)/Q(x)$,
   where $P(x)$ and $Q(x)$ are polynomials with coefficients in $k$, and $Q(0)=1$. The Hadamard product in $k[[x]]$ is defined by
$$(\sum _{i=0}^{\infty} a_ix^i)\odot (\sum _{i=0}^{\infty} b_ix^i) := \sum _{i=0}^{\infty} (a_ib_i)x^i. $$
Let $k[[x]]^{\rm{o}}$ denote the $k$-algebra $(k[[x]], \odot)$.

We shall consider the vector space embedding of $k[[x]]$ in $\prod_{n=0}^{\infty} h_nA$ given by $f(x)\mapsto f(s)$. We will also consider the map
$$\psi \colon k[[x]]^{{\rm o}}\to Z(\prod_{n=0}^{\infty} h_nA)$$
$$\psi (\sum _{n=0}^{\infty}
a_nx^n)= (h_na_n) .$$
The map $\psi $ is clearly an $*$-isomorphism from the algebra of power series with the Hadamard product onto the center
of $\prod_{n=0}^{\infty} h_nA$.

  \begin{proposition}
   \label{prop:BandB*}
   Let $\mathcal D $ be the $*$-subalgebra of $\prod_{n=0}^{\infty} h_nA$ generated by $\mathcal B + \mathcal B^*$. Then
   the ideal of $\mathcal D$ generated by $1-ss^*$ coincides with the ideal of $\mathcal D$ generated by $1-s^*s$.
Moreover, we have $\mathcal D /\langle 1-ss^* \rangle \cong k(x)$, with $s+ \langle 1-ss^* \rangle$ and $s^*+ \langle 1-ss^* \rangle$ mapping to $x$ and $x^{-1}$, respectively.
    \end{proposition}

    \begin{proof}
     Let $\mathcal I$ be the ideal of $\mathcal D $ generated by $1-ss^*$. By looking at the components in $\prod_{n=0}^{\infty} h_nA\cong
     \prod_{n=0}^{\infty} M_{n+1}(k)$, it is straightforward to check that
     \begin{equation}
     \label{eq:equivalence1}
      (1-ss^*) (1-s^*)^{-1}(1-s^*s) (1-s)^{-1} (1-ss^*)= 1-ss^*
     \end{equation}
\begin{equation}
\label{eq:equivalence2}
      (1-s^*s) (1-s)^{-1}(1-ss^*) (1-s^*)^{-1}(1-s^*s) = 1-s^*s .
    \end{equation}
    which proves that $\mathcal I =\langle 1-ss^*\rangle =\langle 1-s^*s \rangle $.

    We now prove that $\mathcal D / \mathcal I \cong k(x)$.
    Note that, if $a\in \mathcal I$, then
$\rank (h_na)<K$ for all $n\ge 0$, for some constant $K$, where $\rank$ is the usual rank of matrices in $h_nA\cong M_{n+1}(k)$.
It follows that $\mathcal I$ is a proper ideal of $\mathcal D$.
There is a surjective algebra homomorphism $\gamma \colon (\Sigma\cup \Sigma^*)^{-1}A\to k(x)$, whose kernel is the ideal $\mathcal T$ of
$(\Sigma\cup \Sigma^*)^{-1}A$ generated by $1-ss^*$ and $1-s^*s$. On the other hand, there is a natural surjective homomorphism
$\tau \colon  (\Sigma\cup \Sigma^*)^{-1}A\to
\mathcal D$ sending $\mathcal T $ onto $\mathcal I$. Therefore we obtain a surjective $k$-algebra homomorphism
$$\ol{\tau}\colon k(x) \cong (\Sigma\cup \Sigma^*)^{-1}A/\mathcal T \longrightarrow \mathcal D/\mathcal I$$
sending $x$ and $x^{-1}$  to the classes of $s$ and $s^*$ in $\mathcal D/\mathcal I$ respectively. Since $k(x)$ is a field,
$\ol{\tau}$ must be an isomorphism.
\end{proof}

Now let $A(x)=\sum _{n=0}^{\infty} a_nx^n$, $B(x) = \sum _{n=0}^{\infty} b_nx^n$ be two formal power series, and write
$\ol{A}(x) =\sum _{n=0}^{\infty} \ol{a}_n x^n$. Then we have
\begin{equation}
 \label{eq:mainrationalbehaviour}
 (1-ss^*)A(s)^*(1-s^*s)B(s)(1-ss^*)= \psi (\ol{A}(x)\odot B(x))(1-ss^*).
\end{equation}
 Similarly
\begin{equation}
 \label{eq:mainrationalbehav2}
 (1-s^*s)A(s)(1-ss^*)B(s)^*(1-s^*s)= \psi (A(x)\odot \ol{B}(x))(1-s^*s).
\end{equation}
Using these formulas, we will find later normal forms for the elements in $\mathcal I =\langle 1-ss^*\rangle $.

Recall from \cite[Theorems 5.3 and 6.1]{BR} that $\mathcal R ^{{\rm o}}$
is closed under the Hadamard product.

  The algebra $\mathcal R ^{{\rm o}}$ is not closed under inversion in $k[[x]]^{{\rm o}}$. As a preparation for the next result, and for latter use,
  we recall the following theorem, a proof of which can be found in \cite[IV.4]{BR}. Here a \emph{quasi-periodic} subset of $\Z^+$ is any subset of the form
$$F\cup \bigcup_{i=1}^r \{k_i+nN\mid n\in \Z^+ \}$$
where $F$ is finite, $N$ is a non-negative integer (the period)
and $k_1,\dots ,k_r\in \{0,1,\dots , N-1 \}$. These sets are precisely the supports of the expansion in base $2$ of rational numbers $q$
with $0\le q\le 1$.

  \begin{theorem}[Skolem-Mahler-Lech]
   \label{thm:SML}
   Let $K$ be a field of characteristic $0$, and let $S=\sum a_nx^n$ be a rational series with coefficients in $K$. Then the set
   $$\{ n\in \Z^+ \mid a_n=0 \}$$
is a quasi-periodic set.
  \end{theorem}

  \begin{lemma}
   \label{lem:closureofmathcalR}
   Let $\mathcal Q$ be the classical ring of quotients of $\mathcal R^{{\rm o}}$. Then $\mathcal Q$ is a commutative $*$-regular ring,
   isomorphic to the $*$-regular closure of $\mathcal R^{{\rm o}}$ in $k[[x]]^{{\rm o}}$.
   \end{lemma}

   \begin{proof}
    Recall that the classical ring of quotients is the ring of quotients obtained by inverting all non-zero-divisors.
    Since $k[[x]]^{{\rm o}}$ is $*$-regular (and we are dealing with a commutative localization) we have that there is
    a natural injective $*$-homomorphism from $\mathcal Q$ into $k[[x]]^{{\rm o}}$. We will identify $\mathcal Q$ with its
    image in $k[[x]]^{{\rm o}}$.

Given a purely periodic set $A=   \bigcup_{i=1}^r \{k_i+nN\mid n\in \Z^+ \}$, with $k_1,\dots ,k_r\in \{0,1,\dots, N-1 \}$, we may consider
the element
$$e= \sum _{i=1}^r (x^{k_i}(1-x^N)^{-1})$$
in $\mathcal R^{{\rm o}}$. Observe that $e$ is idempotent in $\mathcal R^{{\rm o}}$ and that its support is exactly $A$.
If $A$ is just quasi-periodic, we can also obtain $A$ as the support of an element of $\mathcal R^{{\rm o}}$ by adding a polynomial to the above
element $e$.
By the Skolem-Mahler-Lech Theorem (Theorem \ref{thm:SML}), this implies that
the annihilator in $k[[x]]^{{\rm o}}$ of any element of $\mathcal R ^{{\rm o}}$ is generated by a projection in $\mathcal R^{{\rm o}}$.

In order to avoid confusion, we will denote the quasi-inverse of an element $a$ in $k[[x]]^{\rm{o}}$ by $a^\dag$. In particular
the inverse in $k[[x]]^{{\rm o}}$ of an invertible element $b$ is denoted by $b^\dag$.

Let $a\odot b^{\dag}$ be an arbitrary element of $\mathcal Q$, where $a,b\in \mathcal R ^{{\rm o}}$ and $b$ is invertible
in $k[[x]]^{{\rm o}}$.  Let $e$ be the unique idempotent of $\mathcal R^{{\rm o}}$ such that the annihilator of $a$ in $k[[x]]^{{\rm o}}$
is the ideal generated by $e$. Then $a+e\in \mathcal R ^{{\rm o}}$ is invertible in $k[[x]]^{{\rm o}}$, and the quasi-inverse
of $a\odot b^{\dag}$ is $((1-e)\odot b)\odot (a+e)^{\dag}\in \mathcal Q$. This shows that $\mathcal Q$ is $*$-regular.
Therefore $\mathcal Q$ must be the $*$-regular closure of $\mathcal R^{{\rm o}}$ in $k[[x]]^{{\rm o}}$.
\end{proof}

Let $\mathcal E$ be the subalgebra of $\prod_{n=0}^{\infty} h_nA$ generated by $\mathcal D$ and $\psi (\mathcal Q)(1-ss^*)$. We can also say that $\mathcal E$ is generated by
$$\{s,\, s^*\} \cup \{f(s)^{-1},\, (f(s)^*)^{-1} \mid f\in \Sigma \} \cup \{ \psi(q)(1-ss^*) \mid q\in \mathcal Q\}.$$
Note that $\mathcal E $ is a $*$-subalgebra. Our aim is to show that $\mathcal E$ is the $*$-regular closure of $A$. To show this, we need some preparations. Since $A\subseteq \mathcal B \subseteq \mathcal D \subseteq \mathcal E$ and $\soc(A)= \bigoplus_{n=0}^\infty h_nA$ (Proposition \ref{prop:S_1=socA}), we see that $\soc(A)$ is a semisimple essential right and left ideal of $\mathcal E$. It follows that $\soc(A) = \soc(\mathcal E)$.

First, we obtain a formula to compute elements of the form $(f(s)^*)^{-1}$, for $f\in \Sigma$. Recall that we use $\equiv $ to denote that two elements in $\prod_{n=0}^{\infty} h_n A$
are congruent modulo the socle $\bigoplus _{n=0}^{\infty} h_n A$.

\begin{lemma}
 \label{lem:computinginverse}
 Let $f(x)= 1+a_1x+\cdots +a_nx^n\in \Sigma$, and set $f_1(x)= \ol{a}_n+\ol{a}_{n-1}x+\cdots +\ol{a}_1 x^{n-1}+x^n$.
 Then there are polynomials $P_i(x)$, $i=0,\dots ,n-1$ of degree $\le n-1$ such that
 $$(f(s)^*)^{-1}\equiv f_1(s)^{-1}s^n -\sum _{i=0}^{n-1} f_1(s)^{-1} s^i (1-ss^*)P_i(s)^*(f(s)^*)^{-1} .$$
\end{lemma}

\begin{proof} We use the polynomials
$$P_i(x) := - \sum_{j=0}^i a_{n-j} x^{i-j}, \qquad i=0,\dots,n-1.$$
To prove the lemma, it suffices to show that
\begin{equation}
\label{f1cong}
f_1(s) \equiv s^nf(s)^* - \sum_{i=0}^{n-1} s^i(1-ss^*) P_i(s)^*.
\end{equation}
Let $\rho_0: A \rightarrow \prod_{i=1}^\infty M_i(k)$ be the $*$-algebra embedding of Proposition \ref{prop:S_1=socA}, and for $m\ge0$ let $\pi_m : \prod_{i=1}^\infty M_i(k) \to M_m(k)$ be the canonical projection. To establish \eqref{f1cong}, we need to show that
\begin{equation}
\label{pim=0}
\pi_m \rho_0 \biggl( f_1(s)- s^nf(s)^* + \sum_{i=0}^{n-1} s^i(1-ss^*) P_i(s)^* \biggr) = 0
\end{equation}
for all but finitely many $m$.

Fix $m\ge n$ for the remainder of the proof. For $j=0,\dots,n-1$, we have
$$\pi_m \rho_0 \bigl( s^j(1-s^{n-j} (s^*)^{n-j}) \bigr) = \biggl(\, \sum_{l=1}^{m+1-j} e_{j+l,l} \biggr) \biggl(\, \sum_{l'=1}^{n-j} e_{l'l'} \biggr) = \sum_{l=1}^{n-j} e_{j+l,l} .$$
Moreover,
$$f_1(s)- s^nf(s)^* = \biggl(\, \sum_{j=0}^{n-1} \ol{a}_{n-j} s^j+ s^n \biggr) - s^n \biggl( 1+ \sum_{j=0}^{n-1} \ol{a}_{n-j} (s^*)^{n-j} \biggr) = \sum_{j=0}^{n-1} \ol{a}_{n-j} s^j ( 1-s^{n-j} (s^*)^{n-j} ),$$
and consequently
\begin{equation}
\label{pim1}
\pi_m \rho_0 \bigl( f_1(s)- s^nf(s)^* \bigr) = \sum_{j=0}^{n-1} \ol{a}_{n-j} \sum_{l=1}^{n-j} e_{j+l,l} = \sum_{i=0}^{n-1} \sum_{j=0}^i \ol{a}_{n-j} e_{i+1,i+1-j} .
\end{equation}
On the other hand,
$$\pi_m \rho_0 \bigl( s^i(1-ss^*) (s^*)^{i-j} \bigr) = \biggl(\, \sum_{l=1}^{m+1-j} e_{i+l,l} \biggr) e_{11} \biggl(\, \sum_{l'=1}^{m+j-i} e_{l',i-j+l'} \biggr) = e_{i+1,i+1-j} $$
for $0\le j\le i<n$, whence
\begin{equation}
\label{pim2}
\pi_m \rho_0 \biggl(\, \sum_{i=0}^{n-1} s^i(1-ss^*) P_i(s)^* \biggr) = - \sum_{i=0}^{n-1} \sum_{j=0}^i \ol{a}_{n-j} e_{i+1,i+1-j} .
\end{equation}
Combining \eqref{pim1} and \eqref{pim2}, we obtain \eqref{pim=0}, as desired.
 \end{proof}

In the next lemma, we obtain a normal form for elements in the ideal of $\mathcal E$ generated by $1-ss^*$.

 \begin{lemma}
 \label{lem:normalformmathcalI}
 Every element in the ideal $\mathfrak I$ of $\mathcal E$ generated by $1-ss^*$ is a sum of terms of the following forms:
 \begin{enumerate}
   \item[(A)] $\psi (q)\Big[ f(s)^{-1}s^i\Big] (1-ss^*)\Big[ (s^*)^j(g(s)^*)^{-1}\Big]$, for $i,j\ge 0$ and $f, g\in \Sigma$, $q\in\mathcal Q$,
   \item[(B)] $\psi (q) \Big[ (f(s)^*)^{-1}(s^*)^i\Big] (1-s^*s) \Big[ s^jg(s)^{-1}\Big]$, for $i,j\ge 0$ and $f, g\in \Sigma$, $q\in \mathcal Q$,
   \item[(C)] $\psi (q) \Big[ (g(s)^*)^{-1}(s^*)^i\Big] (1-s^*s)\Big[ s^jh(s)^{-1}\Big] (1-ss^*)\Big[ (s^*)^k(f(s)^*)^{-1}\Big]$, for $i,j,k\ge 0$ and $g,f, h\in \Sigma$, $q\in \mathcal Q$,
   \item[(D)] $\psi (q) \Big[ f(s)^{-1}s^k\Big] (1-ss^*)\Big[ (h(s)^*)^{-1} (s^*)^j \Big] (1-s^*s) \Big[ s^i g(s)^{-1}\Big] $, for $i,j,k\ge 0$ and $g,f, h\in \Sigma$, $q\in \mathcal Q$,
   \item[(E)] elements from $\soc (A)$.
 \end{enumerate}
 \end{lemma}

 \begin{proof} It is clear that terms of the forms (A), (C), (D) lie in $\mathfrak I$, and terms of the form (B) lie in $\mathfrak I$ because of Proposition \ref{prop:BandB*}. For each $n\ge 0$, the ideal $h_nA$ is a minimal nonzero ideal of $\mathcal E$, and since $h_n(1-ss^*) \ne 0$, it follows that $h_nA \subseteq \mathfrak I$. Thus, $\soc(A) \subseteq \mathfrak I$, covering terms of the form (E).

  Observe that terms of the form (D) are obtained by applying the involution to terms of the form (C), while the involution sends terms of the forms (A), (B), (E) to terms of the same type. Thus, to show the result, it is enough to prove that the set of sums of
  terms of the forms (A)-(E) is stable under right multiplication. We will show this only for terms of the forms
  (B) and (D). The other cases follow by symmetric arguments. Observe that it is enough to work modulo $\soc (A)$.

  Consider first terms $a$ of the form (B). We show that
  $ab$ is a sum of terms of the forms (B) and (C), for any $b\in \{s,s^*, h(s)^{-1}, (h(s)^*)^{-1}  \}$, where $h\in \Sigma $.
  This is trivial for $b=s$ and for $b=h(s)^{-1}$, and it follows from the proof of Lemma \ref{lem:prodembeddSigmaA}
  for $b= s^*$. So it suffices to consider the case where $b=(h(s)^*)^{-1}$.   For that, it is enough to handle the case where $a= (1-s^*s) s^jg(s)^{-1}$.  But in this case, Lemma \ref{lem:computinginverse}
  gives, for $h(x) = 1+a_1x+\cdots + a_n x^n$,
 \begin{align*}
(1-s^*s)s^jg(s)^{-1} (h(s)^*)^{-1} & \equiv (1-s^*s) s^{j+n}(g(s)h_1(s))^{-1} \\
& - \sum _{i=0}^{n-1} (1-s^*s) s^{j+i}(g(s)h_1(s))^{-1} (1-ss^*) P_i(s)^*(h(s)^*)^{-1} ,
 \end{align*}
 where we have followed the notation used in that lemma. This is a sum of terms of the forms (B) and (C).

 Now for a term $a$ of the form (D), as before the only nontrivial case is that of a
 product $ab$, where $b= (f(s)^*)^{-1}$ for some $f\in \Sigma$.  To deal with such a product, it suffices to consider the case  $a= (1-ss^*) (h(s)^*)^{-1} (s^*)^j  (1-s^*s) g(s)^{-1} s^i$ and show that $ab$ is a sum of terms of the forms (A) and (D).  We have
 \begin{align*}
 \Big[  (1-ss^*) & (h(s)^*)^{-1} (s^*)^j  (1-s^*s) g(s)^{-1} s^i \Big] (f(s)^*)^{-1}\\
 &  \equiv  (1-ss^*) (h(s)^*)^{-1} (s^*)^j  (1-s^*s) (g(s)f_1(s))^{-1} s^{i+n} \\
  & - \sum _{l=0}^{n-1} (1-ss^*) (h(s)^*)^{-1} (s^*)^j  (1-s^*s) (g(s)f_1(s))^{-1} s^{i+l}(1-ss^*)P_l(s)^*(f(s)^*)^{-1} \, .
 \end{align*}
The term  $(1-ss^*) (h(s)^*)^{-1} (s^*)^j  (1-s^*s) (g(s)f_1(s))^{-1} s^{i+n}$ is of the form (D). Moreover, by (\ref{eq:mainrationalbehaviour}),  the term
$ (1-ss^*) (h(s)^*)^{-1} (s^*)^j  (1-s^*s) (g(s)f_1(s))^{-1} s^{i+l}(1-ss^*)$ can be simplified to
$\psi (\ol{A}(x)\odot B_l(x)) (1-ss^*)$, where $A(x)$ and $B_l(x)$ are the rational series representing $x^jh(x)^{-1}$ and $x^{i+l}(g(x)f_1(x))^{-1}$
respectively, so that the corresponding summands in the above formula are sums of terms of the form (A).

This concludes the proof.
  \end{proof}

We are ready to obtain the description of the $*$-regular closure of $A$ in $\mathcal U (G))$.

\begin{theorem}
 \label{thm:regclosureA}
 Let $\mathcal E$ be the subalgebra of $\prod_{n=0}^{\infty} h_nA$ generated by $\mathcal D$ and $\psi (\mathcal Q)(1-ss^*)$.
 Then $\mathcal E= \mathcal R (A, \mathcal U (G))$, and we have
 $$(1-ss^*)\mathcal E (1-ss^*)\cong \mathcal Q \qquad  \text{and} \qquad \mathcal E/\langle 1-ss^*\rangle \cong k(x).$$
 Moreover, $\mathcal E$ is a unit-regular ring, and it coincides with the division closure of $A$
 in $\mathcal U (G)$.
 \end{theorem}

\begin{proof}
 Let $\mathfrak I$ be the ideal of $\mathcal E$ generated by $1-ss^*$. It follows from Proposition \ref{prop:BandB*}  that $\mathfrak I$ is also
 the ideal generated by $1-s^*s$. Since $\mathfrak I \cap \mathcal D$ is a proper ideal of $\mathcal D$ containing $1-ss^*$, we see that $\mathfrak I \cap \mathcal D$ equals the ideal of $\mathcal D$ generated by $1-ss^*$, and so $\mathcal D/ (\mathfrak I \cap \mathcal D) \cong k(x)$. On the other hand, $\mathcal D+ \mathfrak I = \mathcal E$, and thus $\mathcal E/\mathfrak I \cong \mathcal D/ (\mathfrak I \cap \mathcal D) \cong k(x)$.

Let $\widetilde{\psi}\colon \mathcal Q \to (1-ss^*)\mathcal E (1-ss^*)= (1-ss^*) \mathfrak I (1-ss^*)$
 be the map defined by $\widetilde{\psi} (q)= (1-ss^*)\psi (q)$ for $q\in \mathcal Q$. Then $\widetilde{\psi}$ is an injective $*$-homomorphism.
 Using Lemma \ref{lem:normalformmathcalI} and (\ref{eq:mainrationalbehaviour}), we obtain that $\widetilde{\psi}$
 is surjective. Therefore $(1-ss^*)\mathcal E (1-ss^*)\cong \mathcal Q$ is a commutative $*$-regular ring by Lemma \ref{lem:closureofmathcalR}, isomorphic
 to the classical ring of quotients of $\mathcal R ^{{\rm o}}$.

 The next step consists in showing that $\ol{\mathfrak I}:= \mathfrak I/ \soc (\mathcal E)$ is regular.
 Since $1-\ol{s}\ol{s}^*$ is a full idempotent of $\ol{\mathfrak I}$ and $(1-\ol{s}\ol{s}^*)\ol{\mathfrak I} (1- \ol{s}\ol{s}^*)$
 is regular, it suffices to check that $\ol{\mathfrak I}$ is left and right $s$-unital, that is, for each $a\in \ol{\mathfrak I}$ there are elements
 $c$ and $d$ in $\ol{\mathfrak I}$ such that $ca= a = ad$ (see \cite[Theorem 2.3]{Arnmi}).
 By using the involution, it is enough to check the right $s$-unital condition. Using Lemma \ref{lem:normalformmathcalI} and taking common denominators,
 it is enough to show that given $f,g\in \Sigma $ and $N\ge 1$ there is an element $c\in \mathfrak I$ such that $(1-s^*s)s^if(s)^{-1}c\equiv  (1-s^*s) s^if(s)^{-1}$
 and $(1-ss^*)(s^*)^i(g(s)^* )^{-1}c \equiv (1-ss^*) (s^*)^i (g(s)^*)^{-1}$ for $i=0,\dots , N-1$.

Write $a := (1-s^N(s^*)^N)(g(s)^*)^{-1}= (\sum _{i=0}^{N-1} s^i (1-ss^*)(s^*)^i) (g(s)^*)^{-1}$ and $b := \linebreak (1-(s^*)^{N} s^N) f(s)^{-1}=
(\sum _{i=0}^{N-1} (s^*)^i (1-s^*s) s^i) f(s)^{-1}$. We will use the trick of \cite[Lemma  2.2]{Arnmi}, that is, if we are able to
find an element $c_1$ in $\mathfrak I$ such that $ac_1\equiv a$ and then an element $c_2$ in $\mathfrak I$ such that
$b(1-c_1) c_2 \equiv  b(1-c_1)$, then the element $c:= c_1-c_1c_2+c_2$ satisfies that $ac\equiv  a$ and $bc\equiv b$, and clearly $c$
satisfies the desired conditions. (Here we rely on the fact that $(s^i(1-ss^*)(s^*)^i \mid i\ge0)$ and $((s^*)^i(1-s^*s)s^i \mid i\ge0)$ are sequences of pairwise orthogonal idempotents.)

Set $c_1:= g(s)^*(1-s^N(s^*)^N) (g(s)^*)^{-1}\in \mathfrak I$. Then $ac_1= a$.
Now observe that
$$b(1-c_1)= (1-(s^*)^N s^N) f(s)^{-1} g(s)^*s^N(s^*)^N (g(s)^*)^{-1}\, . $$
We will find an element $c_2'$ in $\mathfrak I$ such that $b(1-c_1)g(s)^*c_2'=b(1-c_1)g(s)^*$. Then the element
$c_2= g(s)^*c_2'(g(s)^*)^{-1}$  will satisfy $b(1-c_1)c_2=b(1-c_1)$ and we will find our desired element $c=c_1-c_1c_2+c_2$.

Write $d := b(1-c_1) g(s)^*= (1-(s^*)^Ns^N) f(s)^{-1} g(s)^* s^N (s^*)^N$. Note that this is an element belonging to the subalgebra
$\Sigma^{-1}A$. By the proof of Lemma \ref{lem:prodembeddSigmaA}, $d$ can be written as a linear combination of terms
of the forms $(s^*)^i(1-s^*s)s^jf(s)^{-1} $ and $(s^*)^i(1-s^*s) s^jf(s)^{-1} (1-ss^*) (s^*)^l$. Hence, there exists a positive integer $M$ such that,
with $a' := (1-(s^*)^Ms^M) f(s)^{-1}$ and $b' := 1-s^M(s^*)^M$, if $a'c_2'\equiv a'$ and $b'c_2'\equiv b'$ then $dc_2'\equiv d$.
Set $c_1''= f(s)(1-(s^*)^M s^M) f(s)^{-1}$. Then $a'c_1''= a'$. Now observe that
$$b'(1-c_1'')= (1-s^M(s^*)^M)f(s)(s^*)^Ms^M f(s)^{-1}\, ,$$
and just as before it suffices to find $c_2'''$ in $\mathfrak I$ such that $b'(1-c_1'')f(s)c_2'''\equiv  b'(1-c_1'')f(s)$.
But $b'(1-c_1'')f(s)= (1-s^M(s^*)^M)f(s)(s^*)^Ms^M$ belongs to $A$, and so there exists indeed an idempotent $c_2'''$ in
the ideal of $A$ generated by
$1-ss^*$ (and so in $\mathfrak I$) such that
$b'(1-c_1'')f(s)c_2'''\equiv b'(1-c_1'')f(s)$ (see Lemma \ref{lem:AmodsocA}).
Now $c_2'' := f(s)c_2'''f(s)^{-1}$ satisfies $b'(1-c_1'')c_2''= b'(1-c_1'')$ and so $c_2' := c_1''-c_1''c_2''+c_2''$
satisfies $a'c_2'=a'$ and $b'c_2'=b'$ and, consequently, $dc_2'=d$.
This concludes the proof of the fact that $\ol{\mathfrak I}$ is an $s$-unital ring.
By \cite[Theorem 2.3]{Arnmi}, we get that $\ol{\mathfrak I}$ is a regular ring. Since also $\soc (\mathcal E) =\soc (A)$
is a regular ring, we get from \cite[Lemma 1.3]{vnrr} that $\mathfrak I$ is a regular ring. Finally, since $\mathcal E / \mathfrak I \cong k(x)$
is a field, a further application of \cite[Lemma 1.3]{vnrr} gives that $\mathcal E$ is a regular ring. Since $\mathcal E$ is a $*$-subalgebra of $\prod_{n=0}^\infty h_nA$, it is also $*$-regular.

We clearly have that $\mathcal E = \mathcal R _2(A, \mathcal U (G))= \mathcal R (A, \mathcal U (G))$. Indeed, we obtain
$\mathcal D$ by adjoining inverses of elements of $A$, namely $\Sigma^{-1} \cup (\Sigma^*)^{-1}$, so $\mathcal D \subseteq \mathcal R _1(A, \mathcal U (G))$,
and then we adjoin to $\mathcal D$ the relative inverses of the elements in $(1-ss^*)\mathcal D (1-ss^*)\cong \mathcal R ^{{\rm o}}$,
so that
$$\mathcal E \subseteq \mathcal R _2 (A, \mathcal U (G))\subseteq  \mathcal R (A, \mathcal U (G)).$$
Since $\mathcal E$ is *-regular, the above inclusions must be equalities.

We now show that $\mathcal E $ is unit-regular. Notice that $\mathfrak I$ is unit-regular, that is, all corner rings
$e\mathfrak I e$, for $e= e^2\in \mathfrak I$, are unit-regular. This follows from the fact that $(1-ss^*)\mathfrak I (1-ss^*)$
is unit-regular (because it is a commutative regular ring), and so $V(\mathfrak I ) = V(\mathfrak I (1-ss^*) \mathfrak I )  \cong
V((1-ss^*) \mathfrak I (1-ss^*))$ is cancellative. (The easy fact that $V(R e R) \cong V(eRe)$ for any idempotent $e$ of a non-necessarily unital ring $R$ was
observed in the first paragraph of the proof of \cite[Lemma 7.3]{AF}.) Since $\mathcal E / \mathfrak I\cong k(x) $ is also unit-regular, it suffices
to check that units from $\mathcal E / \mathfrak I$ lift to units in $\mathcal E$ (see \cite[Lemma 3.5]{Bac}). Units from $\Sigma $ obviously
lift to units in $\mathcal E$, so it remains to show that $x$ lifts to a unit in $\mathcal E$. Consider the elements $u :=s+(1-ss^*)(1-s^*)^{-1}(1-s^*s)$
and $v := s^*+(1-s^*s) (1-s)^{-1}(1-ss^*)$ in $\mathcal E$. Then $u+\mathfrak I = x$. Using \eqref{eq:mainrationalbehaviour} and \eqref{eq:mainrationalbehav2}, we compute that
\begin{align*}
(1-ss^*) (1-s^*)^{-1} (1-s^*s) (1-s)^{-1} (1-ss^*) &= 1-ss^*  \\
(1-s^*s) (1-s)^{-1} (1-ss^*) (1-s^*)^{-1} (1-s^*s) &= 1-s^*s ,
\end{align*}
and consequently $uv=vu=1$. Therefore $\mathcal E $ is unit-regular.

It remains to prove that $\mathcal E$ is the division closure of $A$ in $\mathcal U (G)$. Denote by $\mathcal C$ this division closure.
Then clearly $\mathcal D\subseteq \mathcal C$.  Moreover $\mathcal E$ is the ring generated by $\mathcal D$ and the inverses in $\mathcal U (G)$
of elements of the form $\psi (r)(1-ss^*) +ss^*$, where $r$ is a non-zero-divisor in $\mathcal R^{{\rm o}}$. So $\mathcal E\subseteq \mathcal C$.
Since $\mathcal E$, being regular, is closed under inversion, we get that $\mathcal E = \mathcal C $.
\end{proof}

For a subring $S$ of $\mathcal U (G)$, denote by $\mathcal C (S)$ the set $\rk (\bigcup _{i=1}^{\infty} M_i(S))$ of all ranks of matrices over $S$.
Note that $S\subseteq \mathbb R ^+$. Denote also by $\mathcal G (S)$ the subgroup of $\mathbb R$ generated by $\mathcal C (S)$.

\begin{corollary}
 \label{cor:QplusiscE}
 With the above notation,  we have $\mathcal C (\mathcal E) =\Q^+$ and $\mathcal G (\mathcal E)= \Q$. In particular, we obtain that
 $\Q= \mathcal G(A) \subsetneq \mathcal G (k G)$.
  \end{corollary}

   \begin{proof}
   Note that, $\mathcal E $ being a $*$-regular ring with positive definite involution, the set $\mathcal C (\mathcal E)$ is the set of numbers of the form
   $\rk (p)$, where $p$ ranges
   over the projections in matrices over $\mathcal E$. Moreover every projection in $M_n(\mathcal E)$ is equivalent to a diagonal projection
   (since $V(\mathcal E)$ is a refinement monoid).

    Since $(1-ss^*)\mathcal E (1-ss^*) = \psi (\mathcal Q) (1-ss^*)$, we get from Lemma \ref{lem:closureofmathcalR}  and Theorem \ref{thm:SML} that
    the set of ranks of elements in $(1-ss^*)\mathcal E (1-ss^*)$ is exactly  $\Q \cap [0,1/2]$. (Recall that $\rk(h_n(1-ss^*)) = 2^{-(n+2)}$ for all $n\ge0$.) It follows that
    $\mathcal C (\mathfrak I)= \Q^+$. If $p$ is a projection in $\mathcal E$, then, since $\mathcal E/\mathfrak I$ is a field, either $p\in \mathfrak I$
    or $1-p\in \mathfrak I$. It follows that $\rk (p)\in \Q^+$. Therefore $\mathcal C (\mathcal E) =\Q^+$.

    Let now $X\in M_ n(\mathcal E)$, for some $n\ge 1$. Since $\mathcal E $ is the division closure of $A$ and it is regular, it is also the rational closure of $A$.
    Hence, it follows from Cramer's rule (\cite[Proposition 7.1.3]{Cohn}) that there exist a non-negative integer $m$, invertible matrices $P,Q\in M_{n+m}(\mathcal E)$ and
    a matrix $Y\in M_{n+m}(A)$ such that
    $$P(X\oplus I_m) Q = Y.$$
    It follows that $\rk (Y)= \rk (X)+m$, whence $\rk(X) \in \mathcal G(A)$. Since $\mathcal C (\mathcal E)=\Q^+$, we conclude that
    $\Q = \mathcal G (A)\subset \mathcal G (kG)$. The fact that the above inclusion is strict follows from \cite{Grab2}.
      \end{proof}

  We conclude by computing, using our tools, the rank of the element $s+s^*$ (cf. \cite{DS}).

  \begin{example}
   Let $g_1= \psi ((1-x^2)^{-1})$. Observe that $g_1$ is the central projection in $\prod_{i=1}^{\infty} M_i(k)$ corresponding to the sequence
   $(1,0,1,0,1,0,\dots )$. Set $g_2: =1-g_1$. Observe that $(1-s^*s)g_1$ and $(1-s^*s)g_2$ belong to $\mathcal E$. Indeed, it follows from \eqref{eq:mainrationalbehav2} that
  $$g_1(1-s^*s)=
   (1-s^*s)(1-s^2)^{-1}(1-ss^*)(1-s^*)^{-1} (1-s^*s).$$
   We claim that $\alpha s^*(1+s^2)= g_1(1-s^*s) s$, where $\alpha = g_1(1-s^*s) s^2(1+s^2)^{-1}$. It suffices to show that
 \begin{equation}
 \label{pimrho}
 \pi_M \rho \bigl( \alpha s^*(1+s^2) - g_1(1-s^*s) s \bigr) = 0
 \end{equation}
 for all odd $m$, where $\rho : \prod_{n=0}^\infty h_nA \to \prod_{i=1}^\infty M_i(k)$ is the isomorphism stemming from Lemma \ref{lem:centralprojhn} and $\pi_m : \prod_{i=1}^\infty M_i(k) \to M_m(k)$ is the canonical projection. Observe that
\begin{align*}
\alpha s^*(1+s^2) - g_1(1-s^*s) s &= g_1(1-s^*s) s \bigl[ s(1+s^2)^{-1} s^*(1+s^2) - 1 \bigr]  \\
 &= g_1(1-s^*s) s \bigl[ (1+s^2)^{-1} \bigl( 1-(1-ss^*) \bigr) (1+s^2) - 1 \bigr]  \\
 &= - g_1 (1-s^*s)s (1+s^2)^{-1} (1-ss^*).
\end{align*}
For odd $m$, we have
$$\pi_m \rho \bigl( (1+s^2)^{-1} (1-ss^*) \bigr) = e_{11} - e_{31} + e_{51} - \cdots \pm e_{m1} .$$
Since $\pi_m \rho ((1-s^*s)s) = e_{m,m-1}$ with $m-1$ even, \eqref{pimrho} follows. A similar computation gives that $(s+s^*)\beta = g_2(1-s^*s) s$, where $\beta = s^*s (1+(s^*)^2)^{-1}s^*(1-s^*s)sg_2$.

 Next, we claim that
 \begin{equation}
  \label{eq:spluss*decomposition}
(s+s^*)\mathcal E = (1+\alpha)s^*(1+s^2)\mathcal E \oplus g_2(1-s^*s) s\mathcal E.
  \end{equation}
 Observe that
 $$s+s^*= s^*(1+s^2)+(1-s^*s)s= (1+\alpha )s^*(1+s^2) +g_2(1-s^*s) s \, ,$$
 so that $(s+s^*)\mathcal E \subseteq (1+\alpha)s^*(1+s^2)\mathcal E + g_2(1-s^*s) s\mathcal E $.
 On the other hand, $g_2(1-s^*s) s= (s+s^*)\beta$, so that $g_2(1-s^*s)s\in (s^*+s) \mathcal E$
 and consequently also $(1+\alpha) s^*(1+s^2)\in (s^*+s) \mathcal E$, which gives the reverse inclusion.
 If $z\in (1+\alpha) s^*(1+s^2)\mathcal E \cap g_2(1-s^*s) s \mathcal E$ then
 $z= (1-s^*s) g_2z= (1-s^*s)g_2(1+\alpha) s^*(1+s^2)w$ for some $w\in \mathcal E$.
 Since $g_2\alpha = 0$ and $(1-s^*s)s^*= 0$ we obtain that $z= 0$. This shows (\ref{eq:spluss*decomposition}).
 Since $\alpha ^2 =0$, we have that $1+\alpha $ is invertible, and so, we get from (\ref{eq:spluss*decomposition})
  $$\rk (s+s^*) = \rk ((1+\alpha) s^*(1+s^2))+ \rk (g_2(1-s^*s)s)= \rk (s^*s) + \rk (g_2(1-s^*s)ss^*). $$
In $\mathcal N(G)$, we have $s^*s = e_1$ and $g_2(1-s^*s)ss^* = \sum_{n=1,\, n\, \text{odd}}^\infty f_{-n} e_{-n+1} e_{-n+2} \cdots e_0 f_1$, whence $\rk(s^*s) = \frac12$ and $\rk(g_2(1-s^*s)ss^*) = \frac16$. Therefore
$$\rk(s+s^*) = \frac23.$$
 \end{example}



\begin{thebibliography}{10}

\bibitem{Aext} P. Ara,  \emph{Extensions of exchange rings}, J. Algebra \textbf{197} (1997), 409--423.

\bibitem{Arnmi} P. Ara, \emph{Rings without identity which are Morita equivalent to regular rings}, Algebra Colloq. \textbf{11} (2004), 533--540.

\bibitem{Areal} P.  Ara, \emph{The realization problem for von Neumann regular rings},
in ``Ring Theory 2007. Proceedings of the Fifth China-Japan-Korea
Conference", (H. Marubayashi, K. Masaike, K. Oshiro, M. Sato, eds.);
World Scientific, 2009, pp. 21--37.

\bibitem{Aposet} P. Ara, \emph{The regular algebra of a poset},  Trans. Amer. Math. Soc. \textbf{362} (2010), 1505--1546.

\bibitem{AB} P. Ara, M. Brustenga, \emph{The regular algebra of a quiver}. J. Algebra \textbf{309} (2007), 207--235.

\bibitem{AE} P. Ara, R. Exel, \emph{Dynamical systems associated to separated graphs, graph algebras, and paradoxical decompositions}.
Adv. Math. {\bf 252} (2014), 748--804.

\bibitem{AF} P. Ara, A. Facchini,  \emph{Direct sum decompositions of modules, almost trace ideals, and pullbacks of monoids}, Forum Math. \textbf{18} (2006), 365--389.

\bibitem{AG} P. Ara, K. R. Goodearl, \emph{Leavitt path algebras of
separated graphs}. J. reine angew. Math. {\bf 669} (2012), 165--224.

\bibitem{TW} P. Ara, K. R. Goodearl, \emph{Tame and wild refinement monoids}, Semigroup Forum \textbf{91} (2015), 1--27. 

 \bibitem{AGOP} P. Ara, K. R. Goodearl, K. C. O'Meara, E. Pardo,
 \emph{Separative cancellation for projective modules over exchange rings}, Israel J. Math. \textbf{105} (1998), 105-137.

\bibitem{AMP} P. Ara, M. A. Moreno, E. Pardo, \emph{Nonstable $K$-theory for graph
algebras}, Algebr. Represent. Theory \textbf{10} (2007), 157--178.

\bibitem{austin} T. Austin, \emph{Rational group ring elements with kernels having irrational dimension},
Proc. Lond. Math. Soc. (3) \textbf{107} (2013), 1424--1448.



\bibitem{Bac} G. Baccella,  \emph{Semi-Artinian V-rings and semi-Artinian von Neumann regular rings},
J. Algebra \textbf{173} (1995), 587--612.


\bibitem{Berb} S. K. Berberian, Baer *-rings, Die Grundlehren der mathematischen Wissenschaften,
Band 195. Springer-Verlag, New York-Berlin, 1972.

\bibitem{BR} J. Berstel, C. Reutenauer, ``Rational series and their languages".
EATCS Monographs on Theoretical Computer Science, 12. Springer--Verlag, Berlin, 1988.

\bibitem{CL} C. L. Chuang, P. H.  Lee, \emph{On regular subdirect products of
simple Artinian rings}, Pacific J. Math., \textbf{142} (1990),
17--21.

\bibitem{Cohn} P. M. Cohn, ``Free Rings and Their
Relations'', Second Edition, LMS Monographs 19, Academic Press,
London, 1985.

\bibitem{CrabbMunn} M. J. Crabb, W. D. Munn,   \emph{The centre of the algebra of a free inverse monoid}, Semigroup Forum {\bf 55} (1997), 215--220.


\bibitem{DS} W. Dicks, T. Schick, \emph{The spectral measure of certain elements of the complex group ring of a wreath product},
 Geom. Dedicata \textbf{93} (2002), 121--137.

 \bibitem{Dob}
 H. Dobbertin, \emph{Refinement monoids, Vaught monoids, and Boolean algebras}, Math. Annalen \textbf{265} (1983), 473--487.


\bibitem{DPP} N. Dubrovin, P. Pr\'{\i}hoda, G. Puninski,
\emph{Projective modules over the Gerasimov-Sakhaev counterexample},
J. Algebra {\bf 319} (2008), 3259--3279.

\bibitem{elek} G. Elek, \emph{Connes embeddings and von Neumann regular closures of amenable group algebras},
Trans. Amer. Math. Soc. \textbf{365} (2013), 3019--3039.

\bibitem{elek2}  G. Elek, \emph{Lamplighter groups and von Neumann's continuous regular
rings}, arXiv:1402.5499 [math.RA].

\bibitem{vnrr} K. R. Goodearl, ``Von Neumann Regular Rings'',
Pitman, London 1979; Second Ed., Krieger, Malabar, Fl., 1991.

\bibitem{directsum}
K. R.  Goodearl, \emph{Von Neumann regular rings and direct sum
decomposition problems}, in ``Abelian groups and modules" (Padova,
1994), Math. and its Applics. 343, pp. 249--255, Kluwer Acad. Publ., Dordrecht,
1995.

\bibitem{GZane}
K. R. Goodearl, \emph{Leavitt path algebras and direct limits}, in ``Rings, Modules and Representations" (N. V. Dung, et al., eds.), Contemp. Math. \textbf{480} (2009), 165--187.

\bibitem{GM}
K. R. Goodearl, P.  Menal, \emph{Stable range one for rings with
many units}, J. Pure Applied Algebra \textbf{54} (1988), 261--287.

\bibitem{GW}
K. R. Goodearl, R. B. Warfield, Jr., \emph{Algebras over zero-dimensional rings}, Math. Annalen \textbf{223} (1976), 157--168.

\bibitem{Grabowski} L. Grabowski, \emph{On Turing dynamical systems and the Atiyah Problem},  Invent. Math. \textbf{198} (2014), 27--69.

\bibitem{Grab2} L. Grabowski, \emph{Irrational $l^2$-invariants arising from the lamplighter group},  to appear in Groups Geom. Dyn., arXiv:1009.0229v4 [math.GR].

\bibitem{GZ}
R. I. Grigorchuk, A.  Zuk, \emph{The lamplighter group as a group generated by a 2-state automaton, and its spectrum},
Geom. Dedicata \textbf{87} (2001), 209--244.

\bibitem{GLSZ} R. I. Grigorchuk, P. Linnell, T. Schick, A. Zuk, \emph{On a question of Atiyah}, C. R. Acad. Sci. Paris S\' er. I Math. \textbf{331}
(2000),  663--668.

\bibitem{HR} R. Hancock, I. Raeburn, \emph{The C*-algebras of some inverse
semigroups}, Bull. Austral. Math. Soc. {\bf 42} (1990), 335--348.


\bibitem{Jac} N. Jacobson, \emph{Some remarks on one-sided inverses}, Proc. Amer. Math. Soc. \textbf{1} (1950), 352--355.

\bibitem{lam} T. Y. Lam, ``A First Course in Noncommutative Rings'', Grad. Texts in Math. 131, Springer-Verlag, New York, 1991; second ed., 2001.

\bibitem{Lawson} M. V. Lawson, ``Inverse Semigroups", World Scientific, Singapore, 1998.

\bibitem{LW}  F. Lehner, S. Wagner, \emph{Free lamplighter groups and a question of Atiyah},  Amer. J. Math. \textbf{135} (2013), 835--849.

\bibitem{LinnD} P. A. Linnell, \emph{Division rings and group von Neumann algebras}, Forum Math. \textbf{5} (1993), 561--576.


\bibitem{Linn-LMS} P. Linnell, \emph{Noncommutative localization in group rings}.
in ``Non-commutative localization in algebra and topology", 40--59, London Math. Soc. Lecture Note Ser., 330, Cambridge Univ. Press, Cambridge, 2006.

\bibitem{LLS} P. A. Linnell, W.  L\"uck, T. Schick, \emph{The Ore condition, affiliated operators, and the lamplighter group}, High-dimensional manifold topology,
315--321, World Sci. Publ., River Edge, NJ, 2003.

\bibitem{LS} P. Linnell, T. Schick, \emph{The Atiyah conjecture and Artinian rings}, Pure Appl. Math. Q. \textbf{8} (2012),  313--327.

\bibitem{luck} W. L\" uck, ``L2-Invariants: Theory and Applications to Geometry and K-Theory",
Ergebnisse der Mathematik und ihrer Grenzgebiete (3) 44,
Springer-Verlag, Berlin, 2002.


\bibitem{MM} P. Menal, J. Moncasi, \emph{On regular rings with stable range 2}, J. Pure Appl. Algebra \textbf{24} (1982), 25--40.


\bibitem{Munn} W. D. Munn, \emph{Free inverse semigroups}, Proc. London Math. Soc. (3) \textbf{29} (1974), 385-404.

\bibitem{nichol} W. K. Nicholson, \emph{Lifting idempotents and exchange rings}, Trans. Amer. Math. Soc. \textbf{229} (1977), 269--278.

\bibitem{PSZ} M. Pichot, T. Schick, A. Zuk, \emph{Closed manifolds with transcendental $L^2$-Betti numbers},  
J. London Math. Soc. (2) \textbf{92} (2015), 371--392.

\bibitem{Preston} G. B. Preston, \emph{Monogenic inverse semigroups}, J. Austral. Math. Soc. Ser. A \textbf{40} (1986), 321--342.

\bibitem{schofield} A. H. Schofield, ``Representations of Rings
over Skew Fields'', LMS Lecture Notes Series 92, Cambridge Univ.
Press, Cambridge, UK, 1985.

\bibitem{War} R. B. Warfield, Jr., \emph{Exchange rings and decompositions of modules}, Math. Annalen \textbf{199} (1972), 331--36.

\bibitem{W98IsraelJ}
F. Wehrung, \emph{Non-measurability properties of interpolation vector spaces}, Israel J. Math. \textbf{103} (1998), 177--206.

\bibitem{W13ART}
F. Wehrung, \emph{Lifting defects for nonstable K-theory of exchange rings and C*-algebras}, Algebras and Representation Theory \textbf{16} (2013), 553--589.


\end{thebibliography}
\end{document}